\documentclass[11pt]{amsart}
\usepackage{amssymb, amsmath,latexsym,amsfonts,amsbsy, amsthm,mathtools,graphicx,color}
\usepackage{float}
\usepackage{hyperref}

\numberwithin{equation}{section}
\allowdisplaybreaks

\usepackage{bm, extarrows, longtable, enumerate}
\usepackage{imakeidx}
\indexsetup{noclearpage}
\makeindex
\let\al=\alpha

\let\g=\gamma

\let\la=\lambda

\let\G= \Gamma

\let\wt=\widetilde
\let\wh=\widehat

\let\pa=\partial

\def\C{\mathbb C}
\def\R{\mathbb R}

\def\Z{\mathbb Z}

\def\dv{\mathrm{d}v}
\def\dZ{\mathrm{d}Z}
\def\ds{\mathrm{d}s}
\def\du{\mathrm{d}u}
\def\dz{\mathrm{d}z}

\newcommand{\beq}{\begin{equation}}
	\newcommand{\eeq}{\end{equation}}
\newcommand{\ben}{\begin{eqnarray}}
	\newcommand{\een}{\end{eqnarray}}
\newcommand{\beno}{\begin{eqnarray*}}
	\newcommand{\eeno}{\end{eqnarray*}}

\usepackage{fancyhdr}

\fancypagestyle{mystyle}{%
	\fancyhead[CE]{\scriptsize F. SHAO, D. WEI, AND Z. ZHANG}
	\fancyhead[LE]{\scriptsize \thepage}
	\fancyhead[RE]{}
	\fancyhead[CO]{\scriptsize SELF-SIMILAR SOLUTIONS OF RELATIVISTIC EULER EQUATIONS}
	\fancyhead[LO]{}
	\fancyhead[RO]{\scriptsize \thepage}
	\fancyfoot{}
}%

\newtheorem{theorem}{Theorem}[section]

\newtheorem{lemma}[theorem]{Lemma}
\newtheorem{proposition}[theorem]{Proposition}

\theoremstyle{remark}
\newtheorem{remark}[theorem]{Remark}

\begin{document}

\title[Self-similar solutions of relativistic Euler equations]
{Self-similar imploding solutions of the relativistic Euler equations}

\author[F. Shao]{Feng Shao}
\address{School of Mathematical Sciences, Peking University, Beijing 100871,  China}
\email{2101110016@stu.pku.edu.cn}

\author[D. Wei]{Dongyi Wei}
\address{School of Mathematical Sciences, Peking University, Beijing 100871,  China}
\email{jnwdyi@pku.edu.cn}

\author[Z. Zhang]{Zhifei Zhang}
\address{School of Mathematical Sciences, Peking University, Beijing 100871, China}
\email{zfzhang@math.pku.edu.cn}

\date{\today}

\begin{abstract}
Motivated by recent breakthrough on smooth imploding solutions of compressible Euler,  we construct self-similar smooth imploding solutions of isentropic relativistic Euler equations with isothermal equation of state
$p=\frac1\ell\varrho$ for {\it all} $\ell>1$ in physical space dimension $d=2,3$ and for $\ell>1$ close to 1 in higher dimensions.
This work is a crucial step toward solving the long-standing problem: finite time blow-up of the supercritical defocusing nonlinear wave equation.
\end{abstract}

\maketitle

\tableofcontents

\section{Introduction}

In this paper, we consider the relativistic Euler equations, which describe the motion of a relativistic fluid in the flat Minkowski space-time $\mathbb{M}^{1+d}$, $d\geq 1$, with the flat Lorentzian metric $g=(g_{\mu\nu})$ given by \footnote{Throughout this paper, we adopt the standard rectangular coordinates in Minkowski space-time, denoted by $(x^0=t, x^1, \cdots, x^d)$. Greek indices run from $0$ to $d$ and Latin indices from $1$ to $d$, and repeated indices appearing once upstairs and once downstairs are summed over their range.}
\[g_{00}=-1, \ \ g_{ii}=1 \ \ \text{for $1\leq i\leq d$},\ \ g_{\mu\nu}=0\ \ \text{for $\mu\neq\nu$}.\]
The fluid state is represented by the \emph{energy density} $\varrho\geq0$, and the \emph{relativistic velocity} $\vec{u}=(u^0, u^1, \cdots, u^d)$ normalized by
\begin{equation}\label{Eq.normal_vecu}
	|\vec{u}|_g^2=g_{\mu\nu}u^\mu u^\nu=u_\mu u^\mu=-1.
\end{equation}
The motion of the fluid is dominated by the isentropic relativistic Euler equations
\begin{equation}\label{Eq.relativistic_Euler}
	\begin{cases}
		u^\mu\pa_\mu\varrho+(p+\varrho)\pa_\mu u^\mu=0,\\
		\pa_\mu T^{\mu\nu}=0,
	\end{cases}
\end{equation}
where $p$ is the \emph{pressure} subject to the equation of state $p=p(\varrho)$, and $T^{\mu\nu}$ is the \emph{energy-momentum tensor} defined by
\begin{equation}\label{Eq.tensor}
	T^{\mu\nu}=(p+\varrho)u^\mu u^\nu+pg^{\mu\nu}.
\end{equation}

The local-in-time existence of smooth solution for relativistic Euler equations was proved by Makino and Ukai \cite{MU}. The first global existence result for the relativistic Euler equations with $d=1$ was obtained by Smoller and Temple \cite{ST} under the assumption of bounded total variation of the initial data. The condition regarding bounded total variation is weaken significantly in some recent works \cite{Ruan_Zhu, Chen_Schrecker}. Some results on the singularity formation were obtained by Guo and Tahvildar-Zadeh \cite{GTZ}, Pan and Smoller \cite{PS} and so on.  Christodoulou's landmark monograph \cite{Chris1} proved the shock formation for open sets of irrotational and isentropic solutions in the 3-D setting, see also \cite{Chris2}. The shock formation for the relativistic Euler equations with non-trivial vorticity and non-constant entropy remains open up to date. A crucial step  is made by Disconzi and Speck \cite{Disconzi-Speck} recently, who rewrite the relativistic Euler equations as a system possessing  null structures. For more developments on the mathematical aspects of relativistic fluids, we refer to recent review articles by Abbrescia and Speck \cite{AS} and Disconzi \cite{Disconzi}. Let's  mention some important works \cite{Sd, CM, Yin, LS1, LS2} and references therein on the shock formation of compressible Euler equations.

A new blow-up mechanism, known as the \textit{implosion}, has been constructed in the breakthrough work \cite{MRRJ2} by Merle, Raph\"el, Rodnianski and Szeftel for the  isentropic  compressible Euler equations with $d\geq 2$, where both the density and velocity themselves blow up in finite time. More importantly, in the companion papers \cite{MRRJ1,MRRJ3}, they construct radial and asymptotically self-similar blow-up solutions based on smooth self-similar solutions in \cite{MRRJ2} for both the compressible Navier-Stokes equations and the energy supercritical defocusing nonlinear Schr\"odinger equations, respectively. 
In a recent work by Buckmaster, Cao-Labora and Gómez-Serrano \cite{BCLGS}, the self-similar smooth imploding solutions have been extended to cover a wider range of adiabatic exponents in the 3-D setting. Moreover, non-radial imploding solutions are constructed in \cite{CLGSSS}. \smallskip

Motivated by breakthrough works \cite{MRRJ1, MRRJ2, MRRJ3}, we are going to construct self-similar smooth imploding solutions for the isentropic relativistic Euler equations with isothermal equation of state
\begin{equation}\label{Eq.state}
	p=\frac1\ell\varrho,
\end{equation}
where $\ell>1$ corresponds to reciprocal of the square of the sound speed. This equation of state is valid in the early stages of stellar formation, especially $\ell=3$, see Smoller and Temple \cite{ST} for more discussions.\smallskip

Our main result is stated as follows.

\begin{theorem}\label{Thm.mainthm}
	For $k=d-1\geq 1$, there exists a critical $\ell_*(k)\in(1, +\infty]$ (see \eqref{Eq.ell_*}) such that for all $1<\ell<\ell_*(k)$, we can find a $\beta\in(0, k)$ such that the relativistic Euler equations \eqref{Eq.relativistic_Euler} with \eqref{Eq.state} admit a spherically symmetric self-similar solution defined on $[0, T_*]\times \R^d$, smooth away from the concentrated point $(T_*, 0)$ (where $T_*>0$):
	\begin{equation}\label{rho5}
		\begin{cases}
			\varrho(t, x)=\frac1{(T_*-t)^\beta}\hat\varrho(Z),\quad u^0(t,x)=\wh{u^0}(Z),\quad u^i(t,x)=\wh u(Z)\frac{x^i}r\quad (1\leq i\le d),\\
			r=|x|=\big((x^1)^2+\cdots+(x^d)^2\big)^{1/2},\quad Z=r/(T_*-t),
		\end{cases}
	\end{equation}
	with the asymptotics
	\begin{equation*}
		\varrho(t,x)=\frac{\varrho_*(1+o_{Z\to+\infty}(1))}{|x|^\beta},\quad \lim_{Z\to+\infty}\wh{u^0}(Z)=u^0_*, \quad \lim_{Z\to+\infty}\wh{u}(Z)=u_*,
	\end{equation*}
	for some real constants $\varrho_*>0, u^0_*\in\R\setminus\{0\}, u_*\in\R$. 
	\end{theorem}

Several remarks are in order.

\begin{enumerate}
	\item[1.] We have $\ell_*(1)=\ell_*(2)=+\infty$ and for $3\leq k\leq 6$, the value of $\ell_*(k)$ ($=\ell_1(k)$, see \eqref{Eq.ell_*}) can be found in Table \ref{Tab.1}.
For $k\geq3$, the upper bound $\ell_*(k)$ is not optimal and can be improved using the method in Section \ref{Sec.Proof_k=2_l_large}. Here we only consider the case $R\in(3,4)$ defined by \eqref{def:R}, and the upper bound $\ell_*(k)$ can be improved by considering larger $R$.
	
	\item[2.] Our result is not perturbative, in view of Table \ref{Tab.1}. Moreover, our proof does not rely on any numerical computations, except for Table \ref{Tab.1}.

	\item[3.] Our solution constructed in this paper can be used to construct a finite time blow-up for the supercritical defocusing nonlinear wave equation. Indeed, we consider nonlinear wave equation of the form $\Box u=|u|^{p-1}u$ for the unknown field $u:\R^{1+d}\to\C$, where $\Box=\partial^{\mu}\partial_{\mu}=-\partial_t^2+\sum_{i=1}^d\partial_i^2$ is the d'Alembertian operator on $\R^{1+d}.$ We write the modular-phase decomposition $u=\rho \mathrm{e}^{\mathrm{i}\phi}$, with $\rho,\phi:\R^{1+d}\to\R$ and $\rho>0$, then
	\[\Box u=(\Box\rho+2\mathrm{i}\partial^{\mu}\rho\partial_{\mu}\phi+\mathrm{i}\rho\Box\phi-\rho\partial^{\mu}\phi\partial_{\mu}\phi)\mathrm{e}^{\mathrm{i}\phi}.\]
	Now we introduce the front renormalization $\rho\mapsto b^{-\frac{1}{p-1}}\rho$, $ \phi\mapsto b^{-\frac{1}{2}}\phi$, where $b>0$ is a small number, then the above system becomes
	\begin{align*}
		b\Box\rho=\rho^{p}+\rho\partial^{\mu}\phi\partial_{\mu}\phi,\quad 2\partial^{\mu}\rho\partial_{\mu}\phi+\rho\Box\phi=0.
	\end{align*}
	Letting $b\downarrow0$ gives 
	\begin{align*}
		\rho^{p}+\rho\partial^{\mu}\phi\partial_{\mu}\phi=0,\quad 2\partial^{\mu}\rho\partial_{\mu}\phi+\rho\Box\phi=0,
	\end{align*}
	or equivalently,
	\begin{equation}\label{Eq.wave_analogous}
		\rho^{p-1}+\pa^\mu\phi\pa_\mu\phi=0,\quad \pa_\mu(\rho^2\pa^\mu\phi)=0.
	\end{equation}
Let $ \ell=1+\frac{4}{p-1}$, $ \varrho=\rho^{p+1}$, then \eqref{Eq.wave_analogous} becomes 
\eqref{Eq.normal_varrho_phi} and \eqref{Eq.continuity_varrho_phi}, which are deduced from \eqref{Eq.relativistic_Euler} in the irrotational regime. After introducing self-similarity and spherical symmetry into \eqref{Eq.wave_analogous}, we can convert \eqref{Eq.wave_analogous} into an ODE which is exactly the same as \eqref{Eq.main_ODE}. Therefore, this paper gives a crucial  step toward the construction of blow-up solution to the supercritical defocusing nonlinear wave equation, which is the main content of our forthcoming paper \cite{Shao-Wei-Zhang}. In view of Lemma \ref{Lem.beta_neq_l+1}, it seems that we can prove the blow-up for the nonlinear wave equation for $d=4$ and all $p\geq 29$ and $d=5$ for all $p\geq 17$.
\end{enumerate}

Let's conclude this introduction by reviewing some recent progress on self-similar type singularities for the fluid PDEs.\smallskip

Firstly, we refer to \cite{Chen-Hou1, Elgindi, EGM} and \cite{Chen-Hou2, Chen-Hou3, Luo-Hou} and \cite{Chen,Chen-Hou-Huang,Huang-Qin-Wang-Wei1, Huang-Qin-Wang-Wei2, Huang-Tong-Wei}  for the finite time singularity formation of self-similar type on the 3-D incompressible Euler equations and related lower dimensional models. Secondly, in a series of works \cite{Buck-Shko-Vicol1, Buck-D-Shko-Vicol, Buck-Iyer, Buck-Shko-Vicol2,Buck-Shko-Vicol3}, Buckmaster, Shkoller, Vicol et al developed a systematic approach via constructing asymptotically self-similar solutions to study the formation and development of shocks for the compressible Euler equations even in the presence of vorticity and entropy.  Thirdly, there are many recent progress on the gravitational collapse in the field of astrophysics, referring to a process as star implosion, see \cite{Guo-H-Jang2, Guo-H-Jang-Schrecker, Guo-H-Jang1, AHS, Guo-H-Jang3} and references therein. These singularity results are all of self-similar type in the class of radially symmetric solutions, and the key ingredient is to solve some non-autonomous ODEs having sonic points.

\section{The self-similar equation and phase portrait}

Plugging the equation of state \eqref{Eq.state} into the first equation of \eqref{Eq.relativistic_Euler}, we obtain
\begin{equation}\label{Eq.continuity_eq}
	\pa_\mu\left(\varrho^{\frac\ell{\ell+1}}u^\mu\right)=0.
\end{equation}
We find  the solution of  spherical symmetry, which takes 
\begin{equation*}
	\varrho=\varrho(t, r), \quad u^0=u^0(t, r), \quad u^i=u_r(t, r)\frac{x^i}{r}, \quad 1\leq i\leq d,
\end{equation*}
where
 $$t=x^0 \quad \text{and}\quad r=\big((x^1)^2+\cdots+(x^d)^2\big)^{1/2}.$$ 
Under the spherical symmetry, \eqref{Eq.continuity_eq} is equivalent to
\begin{equation}\label{Eq.mainPDE1}
	\pa_t\left(\varrho^{\frac\ell{\ell+1}}u^0\right)+\pa_r\left(\varrho^{\frac\ell{\ell+1}}u_r\right)+\frac kr\varrho^{\frac\ell{\ell+1}}u_r=0,
\end{equation}
where $k=d-1$.\index{$k=d-1$}

Furthermore, we assume that the fluid is irrotational, i.e.,
\begin{equation}\label{Eq.vecu_profile}
	u^\mu=\varrho^{-\frac1{\ell+1}}\pa^\mu\phi
\end{equation}
for some potential function $\phi$. For spherically symmetric $\phi=\phi(t, r)$, we have $u^0=-\varrho^{-\frac1{\ell+1}}\pa_t\phi$ and $u_r=\varrho^{-\frac1{\ell+1}}\pa_r\phi$, or equivalently,
\begin{equation}\label{Eq.mainPDE2}
	\pa_r\left(\varrho^{\frac1{\ell+1}}u^0\right)+\pa_t\left(\varrho^{\frac1{\ell+1}}u_r\right)=0.
\end{equation}

It follows from \eqref{Eq.normal_vecu} and \eqref{Eq.vecu_profile} that
\begin{equation}\label{Eq.normal_varrho_phi}
	0=\varrho^{\frac2{1+\ell}}+\pa^\mu\phi\pa_\mu\phi=\varrho^{\frac2{1+\ell}}-|\pa_t\phi|^2+\sum_{i=1}^d|\pa_i\phi|^2.
\end{equation}
Plugging \eqref{Eq.vecu_profile} into \eqref{Eq.continuity_eq}, we get
\begin{equation}\label{Eq.continuity_varrho_phi}
	\pa_\mu\left(\varrho^{\frac{\ell-1}{\ell+1}}\pa^\mu\phi\right)=0.
\end{equation}
Now  the second equation in \eqref{Eq.relativistic_Euler} is automatically satisfied. Indeed, inserting \eqref{Eq.vecu_profile} and \eqref{Eq.state} into \eqref{Eq.tensor} gives 
\[T^{\mu\nu}=\frac{\ell+1}{\ell}\varrho^{\frac{\ell-1}{\ell+1}}\pa^\mu\phi\pa^\nu\phi+\frac1\ell\varrho g^{\mu\nu},\]
and then by \eqref{Eq.normal_varrho_phi} and \eqref{Eq.continuity_varrho_phi}, we get
\begin{align*}
	\pa_\mu T^{\mu\nu}&=\frac{\ell+1}{\ell}\pa_\mu\left(\varrho^{\frac{\ell-1}{\ell+1}}\pa^\mu\phi\right)\pa^\nu\phi+
\frac{\ell+1}{\ell}\varrho^{\frac{\ell-1}{\ell+1}}\pa^\mu\phi\pa_\mu\pa^\nu\phi+\frac1\ell\pa_\mu\varrho g^{\mu\nu}\\
	&=\frac{\ell+1}{2\ell}\varrho^{\frac{\ell-1}{\ell+1}}\pa^\nu\left(\pa_\mu\phi\pa^\mu\phi\right)+\frac1\ell\pa^\nu\varrho=
-\frac{\ell+1}{2\ell}\varrho^{\frac{\ell-1}{\ell+1}}\pa^\nu\left(\varrho^{\frac2{\ell+1}}\right)+\frac1\ell\pa^\nu\varrho=0.
\end{align*}

As a result, it suffices to solve the system consisting of \eqref{Eq.mainPDE1}, \eqref{Eq.mainPDE2} and
\begin{equation}\label{Eq.mainPDE3}
	u^0(t,r)^2-u_r(t, r)^2=1
\end{equation}
for the unknowns $(\varrho, u^0, u_r)$.

For each $\beta>0$, the system consisting of \eqref{Eq.mainPDE1}, \eqref{Eq.mainPDE2} and \eqref{Eq.mainPDE3} is invariant under the scaling
\[\varrho_\lambda(t,r)=\lambda^\beta\varrho(\la t, \la r),\quad u^0_\la(t, r)=u^0(\la t, \la r),\quad u_{r,\la}(t,r)=u_r(\la t, \la r),\qquad \la>0.\]
We look for solutions $(\varrho(t, r), u^0(t, r), u_r(t,r))$ that are \emph{self-similar}\footnote{In \cite{Lai}, Lai considered the case where $\beta=0$, in which Lai constructed continuous self-similar radially symmetric solutions to the relativistic Euler equations with or without shock. We emphasize that in \cite{Lai}, even if the shock is absent, the global self-similar profile belongs to $C^0\setminus C^1$.}:
\begin{equation}
	\varrho(t,r)=(T_*-t)^{-\beta}\wh\varrho(Z), \quad u^0(t, r)=\wh{u^0}(Z),\quad u_r(t, r)=\wh u(Z), \qquad t\in[0, T_*),
\end{equation}
where $T_*>0$ and $Z=r/(T_*-t)$. Then we compute (here the prime $'$ represents derivatives with respect to $Z$)
\begin{align*}
	&\pa_t\left(\varrho^{\frac{\ell}{\ell+1}}u^0\right)=(T_*-t)^{-\frac{\beta\ell}{\ell+1}-1}
	\left(\frac{\beta\ell}{\ell+1}\widehat{\varrho}^{\frac{\ell}{\ell+1}}\widehat{u^0}(Z)+Z\left(\widehat{\varrho}^{\frac{\ell}{\ell+1}}\widehat{u^0}\right)'(Z)
	\right),\\
	&\pa_t\left(\varrho^{\frac{1}{\ell+1}}u_r\right)=(T_*-t)^{-\frac{\beta}{\ell+1}-1}
	\left(\frac{\beta}{\ell+1}\widehat{\varrho}^{\frac{1}{\ell+1}}\widehat{u}(Z)+Z\left(\widehat{\varrho}^{\frac{1}{\ell+1}}\widehat{u}\right)'(Z)
	\right),\\
	&\pa_r\left(\varrho^{\frac{\ell}{\ell+1}}u_r\right)=(T_*-t)^{-\frac{\beta\ell}{\ell+1}-1}
	\left(\widehat{\varrho}^{\frac{\ell}{\ell+1}}\widehat{u}\right)'(Z),\\
	&\pa_r\left(\varrho^{\frac{1}{\ell+1}}u^0\right)=(T_*-t)^{-\frac{\beta}{\ell+1}-1}
	\left(\widehat{\varrho}^{\frac{1}{\ell+1}}\widehat{u^0}\right)'(Z),\\
	&\frac{k}{r}\varrho^{\frac{\ell}{\ell+1}}u_r=(T_*-t)^{-\frac{\beta\ell}{\ell+1}-1}\frac{k}{Z}\widehat{\varrho}^{\frac{\ell}{\ell+1}}\widehat{u}(Z).
\end{align*}
Hence, \eqref{Eq.mainPDE1} and \eqref{Eq.mainPDE2} are equivalent to
\begin{align*}
	&\frac{\beta\ell}{\ell+1}\widehat{\varrho}^{\frac{\ell}{\ell+1}}\widehat{u^0}(Z)+
	Z\left(\widehat{\varrho}^{\frac{\ell}{\ell+1}}\widehat{u^0}\right)'(Z)+\left(\widehat{\varrho}^{\frac{\ell}{\ell+1}}\widehat{u}\right)'(Z)+
	\frac{k}{Z}\widehat{\varrho}^{\frac{\ell}{\ell+1}}\widehat{u}(Z)=0,\\
	&\frac{\beta}{\ell+1}\widehat{\varrho}^{\frac{1}{\ell+1}}\widehat{u}(Z)+Z\left(\widehat{\varrho}^{\frac{1}{\ell+1}}\widehat{u}\right)'(Z)
	+\left(\widehat{\varrho}^{\frac{1}{\ell+1}}\widehat{u^0}\right)'(Z)=0,
\end{align*}
which are further reduced  to
\begin{align}\label{rho3}
	&\frac{\beta\ell}{\ell+1}\widehat{u^0}(Z)+
	Z\widehat{u^0}'(Z)+\widehat{u}'(Z)+\frac{k}{Z}\widehat{u}(Z)+
	\frac{\ell}{\ell+1}\left(Z\widehat{u^0}(Z)+\widehat{u}(Z)\right)\frac{\widehat{\varrho}'(Z)}{\widehat{\varrho}(Z)}=0,\\
	&\label{rho4}\frac{\beta}{\ell+1}\widehat{u}(Z)+
	Z\widehat{u}'(Z)+\widehat{u^0}'(Z)
	+\frac{1}{\ell+1}\left(Z\widehat{u}(Z)+\widehat{u^0}(Z)\right)\frac{\widehat{\varrho}'(Z)}{\widehat{\varrho}(Z)}=0.
\end{align}
By \eqref{Eq.mainPDE3} we have $\widehat{u^0}(Z)^2-\widehat{u}(Z)^2=1 $. Let $v(Z)=-\widehat{u}(Z)/\widehat{u^0}(Z)$. Then we obtain
\begin{align}\label{u1}
	&\widehat{u^0}(Z)=-\frac{1}{\sqrt{1-v(Z)^2}},\quad \widehat{u}(Z)=\frac{v(Z)}{\sqrt{1-v(Z)^2}},\\& \widehat{u^0}'(Z)=-\frac{v(Z)v'(Z)}{(1-v(Z)^2)^{3/2}},\quad \widehat{u}'(Z)=\frac{v'(Z)}{(1-v(Z)^2)^{3/2}}.
\end{align}
Here we assume $\widehat{u^0}(Z)<0 $. Let $m=\frac{\beta\ell}{\ell+1}>0$.\index{$m=\frac{\beta\ell}{\ell+1}$} Then \eqref{rho3} and \eqref{rho4} become
\begin{align}\label{rho1}
	&\frac{-m+{kv(Z)}/{Z}}{\sqrt{1-v(Z)^2}}+\frac{v'(Z)(1-Zv(Z))}{(1-v(Z)^2)^{3/2}}
	+\frac{\ell}{\ell+1}\frac{v(Z)-Z}{\sqrt{1-v(Z)^2}}\frac{\widehat{\varrho}'(Z)}{\widehat{\varrho}(Z)}=0,\\
	&\label{rho2}\frac{mv(Z)}{\ell\sqrt{1-v(Z)^2}}+\frac{v'(Z)(Z-v(Z))}{(1-v(Z)^2)^{3/2}}
	+\frac{1}{\ell+1}\frac{Zv(Z)-1}{\sqrt{1-v(Z)^2}}\frac{\widehat{\varrho}'(Z)}{\widehat{\varrho}(Z)}=0.
\end{align}
Eliminating  $\hat \varrho$ from \eqref{rho1} and \eqref{rho2} gives 
\begin{equation}\label{Eq.2.15}
	\begin{aligned}
		&\left[(-m+{kv(Z)}/{Z})\left(1-v(Z)^2\right)+{v'(Z)(1-Zv(Z))}\right](Zv(Z)-1)\\
		=&\left[mv(Z)(1-v(Z)^2)+\ell v'(Z)(Z-v(Z))\right](v(Z)-Z),
	\end{aligned}
\end{equation}
which can be rearranged as
\begin{equation}\label{Eq.main_ODE0}
	\left[(1-Zv)^2-\ell(v-Z)^2\right]v'(Z)=(1-v^2)\left[m(1-v^2)-kv\left(\frac1Z-v\right)\right],
\end{equation}
or
\begin{equation}\label{Eq.main_ODE}
	\begin{aligned}
		\frac{\dv}{\dZ}&=\frac{(1-v^2)\left[m(1-v^2)-kv\left(\frac1Z-v\right)\right]}{(1-Zv)^2-\ell(v-Z)^2}\\
		&=\frac{(1-v^2)\left[m(1-v^2)Z-kv\left(1-vZ\right)\right]}{Z\left[(1-Zv)^2-\ell(v-Z)^2\right]}=\frac{\Delta_v(Z, v)}{\Delta_Z(Z, v)},
	\end{aligned}
\end{equation}
where\index{$\Delta_v, \Delta_{Z}$}
\begin{align*}
&\Delta_v(Z, v)=(1-v^2)\left[m(1-v^2)Z-kv\left(1-vZ\right)\right],\\
&\Delta_Z(Z, v)=Z\left[(1-Zv)^2-\ell(v-Z)^2\right].
\end{align*}

\begin{lemma}\label{Lem.parameter_range}
	If $v(Z):[0,1]\to(-1,1)$ is a $C^1$ solution to \eqref{Eq.main_ODE0} with $v(0)=0$ and $\ell>1$, $m>0$, $k>0$,  then $k>m\left(1+\frac1{\sqrt\ell}\right)$.
\end{lemma}

\begin{proof}
	Let $F(Z)=(1-Zv)+\sqrt{\ell}(v-Z)$. Then $F\in C([0,1]) $ is real valued, $F(0)=1>0,$ $F(1)=(\sqrt{\ell}-1)(v(1)-1)<0$. Thus, there exists $Z_0\in (0,1)$ such that $F(Z_0)=0$.
	Then 
	\begin{align*} &[(1-Zv)^2-\ell(v-Z)^2]v'|_{Z=Z_0}\\
		=&[(1-Zv)+\sqrt{\ell}(v-Z)][(1-Zv)-\sqrt{\ell}(v-Z)]v'|_{Z=Z_0}\\=&F(Z_0)[(1-Zv)-\sqrt{\ell}(v-Z)]v'|_{Z=Z_0}=0, 
		\end{align*}
	which gives $(1-v^2)[m(1-v^2)-kv(1/Z-v)]|_{Z=Z_0}=0 $. Let $v_0=v(Z_0)$. Then $v_0\in (-1,1)$, and $(1-v_0^2)[m(1-v_0^2)-kv_0(1/Z_0-v_0)]=0, $ thus $m(1-v_0^2)-kv_0(1/Z_0-v_0)=0 $. This implies that $v_0\in(0,1)$, since otherwise $v_0\in(-1,0]$, hence $m(1-v_0^2)>0$ but $kv_0(1/Z_0-v_0)\leq 0$,  a contradiction. We also have $0=F(Z_0)=(1-Z_0v_0)+\sqrt{\ell}(v_0-Z_0)$, so $ Z_0=\frac{1+\sqrt{\ell}v_0}{\sqrt{\ell}+v_0}$. Hence,
	\begin{align*}
		k&=\frac{m(1-v_0^2)}{v_0\left(\frac1{Z_0}-v_0\right)}=\frac{m(1-v_0^2)}{v_0\left(\frac{\sqrt{\ell}+v_0}{1+\sqrt{\ell}v_0}-v_0\right)}=
\frac{m(1-v_0^2)}{v_0\frac{\sqrt{\ell}(1-v_0^2)}{1+\sqrt{\ell}v_0}}\\
		&=\frac{m(1+\sqrt\ell v_0)}{\sqrt\ell v_0}=\frac m{\sqrt{\ell}}\left(\frac1{v_0}+\sqrt{\ell}\right)>\frac m{\sqrt\ell}(1+\sqrt\ell)=m\left(1+\frac1{\sqrt\ell}\right).
	\end{align*}
	This completes the proof.
	\end{proof}

In view of the above lemma, it is natural to restrict the parameters $(m, \ell, k)$ in the following regime 
\begin{equation}\label{Eq.parameter_range}
	m>0, \qquad \ell>1, \qquad k\in\Z_{+}, \qquad k>m\left(1+\frac1{\sqrt\ell}\right).
\end{equation}

The solutions to $\Delta_v(Z, v)=0$ are $v=\pm1$ and
\begin{equation}\label{Eq.v_12}
	v_1(Z)=\frac{k-\sqrt{k^2-4(k-m)mZ^2}}{2(k-m)Z},\qquad v_2(Z)=\frac{k+\sqrt{k^2-4(k-m)mZ^2}}{2(k-m)Z}
\end{equation}
for $|Z|\leq Z_e=\frac k{2\sqrt{(k-m)m}}$, which is equivalent to $Z=Z_\text b(v)=\frac{kv}{m+(k-m)v^2}$,\index{$Z_\text b(v)$} see the black curve passing $P_1$ in Figure \ref{Fig.Phase_Portrait_vZ}.
The solutions to $\Delta_Z(Z, v)=0$ are $Z=0$ and
\begin{equation}\label{Eq.v_pm}
	v_+(Z)=\frac{\sqrt\ell Z+1}{Z+\sqrt\ell},\qquad v_-(Z)=\frac{-\sqrt\ell Z+1}{Z-\sqrt\ell}.
\end{equation}
As a result, it is easy to find that the solutions to $\Delta_v=\Delta_Z=0$ are
\begin{align*}
	P_0&=(Z=0, v=0),\qquad P_1=\left(Z_1=\frac{k\sqrt\ell}{\ell(k-m)+m}, v_1=\frac{m}{(k-m)\sqrt\ell}\right)\index{$P_1(Z_1,v_1)$},\\
	P_2&=(Z=1, v=1), \qquad P_3=(Z=0, v=1),
\end{align*}
and $P_{-j}=-P_j$ for $j=1, 2, 3$.

\begin{figure}
	\centering
	\includegraphics[width=1\textwidth]{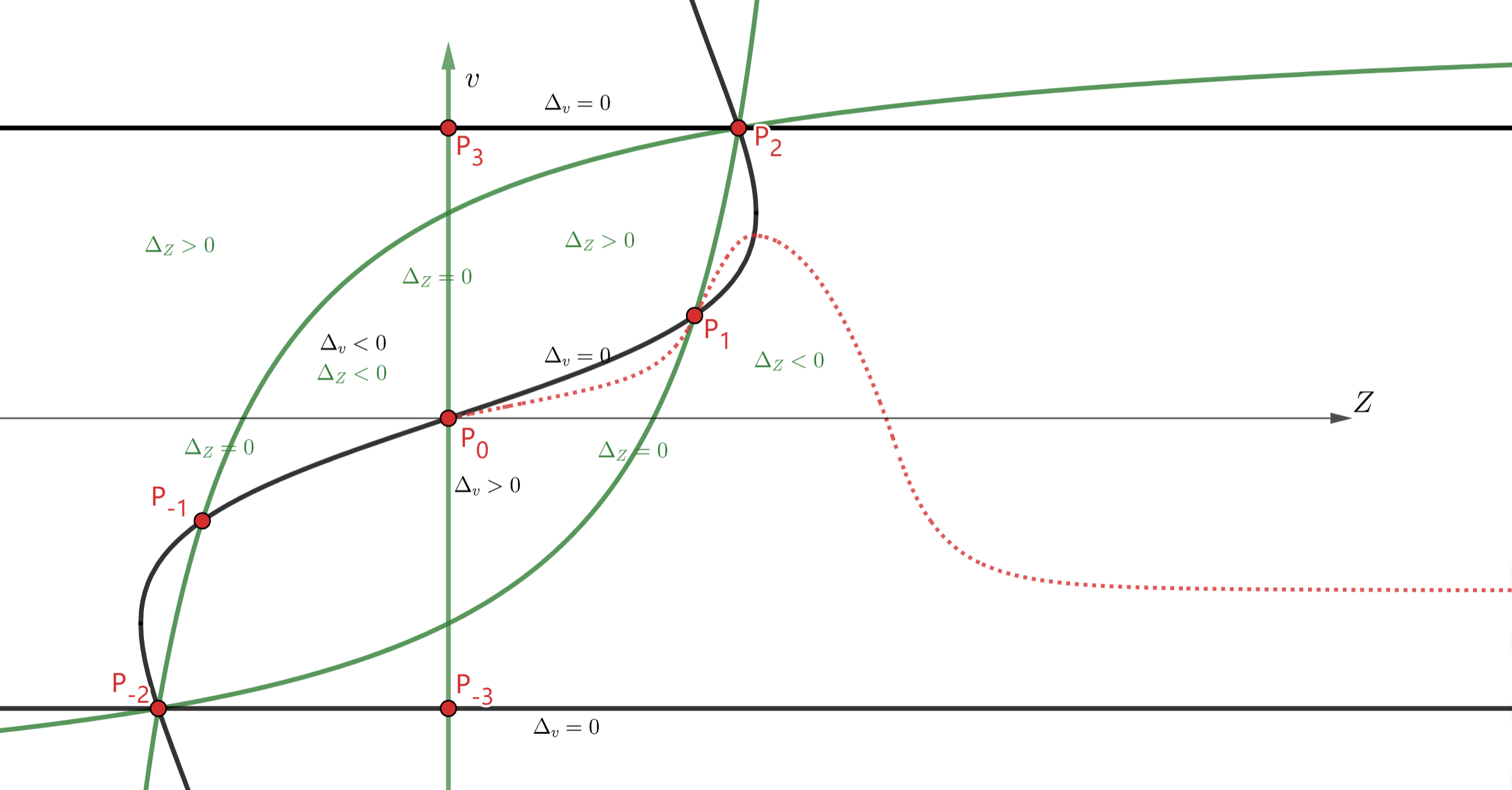}
	\caption{Phase portrait for the $Z-v$ system, with $m=1, \ell=2$ and $k=3$.}
	\label{Fig.Phase_Portrait_vZ}
\end{figure}

Note that the condition \eqref{Eq.parameter_range} implies that $0<Z_1<1$ and $0<v_1<1$. For $Z\in(0, Z_e)$, we have $0<v_1(Z)<v_2(Z)$; moreover, it follows from $m(1-v_1^2)Z-kv_1(1-v_1Z)=(k-m)Zv_1^2-kv_1+mZ=0$ that
\begin{align*}
	v_1'(Z)=-\frac{m+(k-m)v_1^2(Z)}{2(k-m)Zv_1(Z)-k}=\frac{\frac kZv_1(Z)}{\sqrt{k^2-4(k-m)mZ^2}}>0
\end{align*}
and similarly $v_2'(Z)=-\frac{\frac kZv_2(Z)}{\sqrt{k^2-4(k-m)mZ^2}}<0$ for $Z\in(0, Z_e)$. On the other hand, we have
\begin{align*}
	\sqrt{k^2-4(k-m)mZ_1^2}&=\sqrt{\frac{k^2(\ell(k-m)+m)^2-4(k-m)mk^2\ell}{(\ell(k-m)+m)^2}}\\
	&=\sqrt{\frac{k^2(\ell(k-m)-m)^2}{(\ell(k-m)+m)^2}}=\frac{k(\ell(k-m)-m)}{\ell(k-m)+m},
\end{align*}
where we have used \eqref{Eq.parameter_range} to get $\frac km>1+\frac1{\sqrt\ell}>1+\frac1\ell$ and hence $\ell(k-m)-m>0$, thus
\[v_1(Z_1)=\frac{k-\frac{k(\ell(k-m)-m)}{\ell(k-m)+m}}{2(k-m)\frac{k\sqrt\ell}{\ell(k-m)+m}}=\frac{2km}{2(k-m)k\sqrt\ell}=v_1.\]
Therefore, the point $P_1$ lies on the graph of $v_1$, and we have
\begin{equation}
	v_1'(Z)>0,\qquad \text{for}\qquad Z\in(0, Z_1)\subset (0, Z_e).
\end{equation}

Finally, it is easy to check that the point $P_1$ lies on the graph of $v_-$ and
\begin{equation}
	v_+'(Z)=\frac{\ell-1}{(Z+\sqrt\ell)^2}>0,\qquad v_-'(Z)=\frac{\ell-1}{(Z-\sqrt\ell)^2}>0
\end{equation}
in view of \eqref{Eq.parameter_range}. For further usage, we write $v_-(Z)$ as $
Z_g(v)=\frac{\sqrt\ell v+1}{v+\sqrt\ell}$\index{$Z_g(v)$}, see the green curve passing $P_1$ in Figure \ref{Fig.Phase_Portrait_vZ}.\smallskip

Our aim in this paper is to show the existence of a $C^\infty$ solution to the $Z-v$ ODE \eqref{Eq.main_ODE}.

\begin{theorem}\label{Thm.mainthm_ODE}
	Let $k=d-1\geq 1$. Then there exists $\ell_*(k)\in(1, +\infty]$ such that for any $1<\ell<\ell_*(k)$, there exists $\beta\in(0, k)$ such that \eqref{Eq.main_ODE} with this $\beta$ admits a global $C^\infty([0, +\infty))$ solution $v=v(Z)$ defined on $Z\in[0,+\infty)$ satisfying $v(0)=0$ and $|v(Z)|<1$ for all $Z\in[0,+\infty)$; moreover, we have $ \Delta_Z(Z,v(Z))>0$ for $Z\in (0,Z_1)$, $ \Delta_Z(Z,v(Z))<0$ for $Z\in (Z_1,+\infty)$, $v(Z_1)=Z_1$ and $v(Z)<Z$ for all $Z\in(0, +\infty)$. See the dashed red curve in Figure \ref{Fig.Phase_Portrait_vZ}.
\end{theorem}

Theorem \ref{Thm.mainthm} is a corollary of Theorem \ref{Thm.mainthm_ODE}, see Section \ref{Sec.proof}. Now let's give a sketch of  the proof of Theorem \ref{Thm.mainthm_ODE}.\smallskip

The first step is to construct a smooth local solution to \eqref{Eq.main_ODE} starting at $(0,0)$ and  show that this solution can be extended until reaching $P_1$, and we denote it by $v_F$. However, we don't know anything about the smoothness at $P_1$ of $v_F$. Our strategy, alike \cite{MRRJ2}, is to construct a smooth local solution $v_L$ near $P_1$, and to show $v_L=v_F$, hence $v_F$ is smooth at $P_1$. Due to high nonlinearity in the $Z-v$ ODE, we will introduce a renormalization $(Z, v)\mapsto(z, u)$, which converts the $Z-v$ ODE to a $z-u$ ODE, see \eqref{Eq.ODE_P_1}, and we will find the corresponding curve of $v_F$ in the $z-u$ plane, denoted by $u_F$, which is a smooth curve connecting $Q_0$ to $Q_1$, see Figure \ref{Fig.Phase_Portrait_zu}. See Section \ref{Sec.P_0-P_1} for the details.\smallskip

The second step is to construct the local smooth solution $u_L$ to the $z-u$ ODE near $Q_1$, using the power series method. Then through the inverse mapping of the renormalization, we get a local solution $v_L$ to the $Z-v$ ODE near $P_1$. This is our main goal in Section \ref{Sec.local_P_1}.\smallskip

The third step is to show that $u_L=u_F$, hence $v_L=v_F$. However, this is not true for some parameters. By analyzing the coefficients of the power series defining $u_L$ and using the barrier function method, we can show that for fixed $k\geq 1$ and  $\ell\in(1, \ell_*(k))$ (not for all $\ell>1$), there exists $\zeta_0$ such that for some $\beta_1$ we have $u_L(\zeta_0)>u_F(\zeta_0)$, and for some $\beta_2$ we have $u_L(\zeta_0)<u_F(\zeta_0)$, hence by the continuous dependence on parameters, the intermediate value theorem gives a $\beta_0$ such that for this $\beta_0$ we have $u_L(\zeta_0)=u_F(\zeta_0)$. Then by the uniqueness of solutions to ODEs at regular points, we have $u_L\equiv u_F$. This task will be done in Sections \ref{Sec.qualitative} and \ref{Sec.left_P_1}. We will enclose the treatment of $\ell_*(k)$ in appendix \ref{Appen.B_2>0}. We emphasize that for $k=2$ and $\ell$ large, we need to modify the barrier functions used in Sections \ref{Sec.qualitative} and \ref{Sec.left_P_1}, see Section \ref{Sec.Proof_k=2_l_large} for the details.\smallskip

Finally, we show that the $P_0-P_1$ curve $v_F$, after passing $P_1$ smoothly, will then pass through the black curve $Z=Z_\text b(v)$ and then be defined on the whole half-line $Z\in[0,+\infty)$. Here we use the barrier function method to exclude the possibility of touching the green curve $Z_g(v)$ or $P_2$. See Section \ref{Sec.right_P_1} for the details.\smallskip

In Section \ref{Sec.proof}, we will combine all things to give a proof of Theorem \ref{Thm.mainthm_ODE}.

\section{Construction of  $P_0-P_1$ curve and renormalization}\label{Sec.P_0-P_1}

In this section, we construct the unique solution $v=v_F(Z)$ to \eqref{Eq.main_ODE} connecting $P_0$ and $P_1$. After a renormalization introduced in Subsection \ref{Subsec.re-normalization}, the $Z-v$ ODE \eqref{Eq.main_ODE} is transformed into a $z-u$ ODE, see \eqref{Eq.ODE_P_1}, which is easier to analyze when we consider the analytic solutions crossing the sonic point. Before we analyze the analytic solutions of the $z-u$ ODE crossing the sonic point $Q_1$, we will show in Subsection \ref{Subsec.Q_0-Q_1_curve} that the $P_0-P_1$ solution to the $Z-v$ ODE \eqref{Eq.main_ODE} is transformed by the renormalization mapping \eqref{Eq.change_of_variables} into a $Q_0-Q_1$ solution to the $z-u$ ODE.

\subsection{The $P_0-P_1$ curve}
Here we construct the unique solution $v=v_F(Z)$ to \eqref{Eq.main_ODE} connecting $P_0$ and $P_1$. The main idea is to first construct a local solution near the origin $P_0$ using the power series method, then prove that this local solution can be continued up to  $P_1$. We recall the classical theory on the extension of the solutions to ODEs.

\begin{proposition}\label{Prop.IVP}
	Consider the initial value problem
	\begin{equation}\label{Eq.IVP}
		\frac{\mathrm d\bm{x}}{\mathrm dt}=\bm f(t, \bm x),\qquad \bm x(t_0)=\bm{x_0},
	\end{equation}
	where $\bm f\in {C^1}(U; \R^n)$, $U\subset \R^{1+n}$ is an open set and $(t_0, \bm{x_0})\in U$. Let $I=(T_-, T_+)$ be the maximal interval of existence of a solution to \eqref{Eq.IVP}. If $T_+<+\infty$, then the solution must eventually leave every compact set $C$ with $[t_0, T_+]\times C\subset U$ as $t$ approaches $T_+$.
\end{proposition}

The proof of Proposition \ref{Prop.IVP} can be found in section 2.6 of the classical textbook \cite{T}.

\begin{remark}\label{rem1}
	If $T_+<+\infty$ and the limit $ \bm x_*:=\lim_{T\uparrow T_+}\bm{x}(t)$ exists, then $(T_+, \bm x_*)\not\in U$.
\end{remark}

\begin{proposition}\label{Prop.P_0-P_1}
	The ODE
	\begin{equation}\label{Eq.mainODE}
		\frac{\dv}{\dZ}=\frac{(1-v^2)\left[m(1-v^2)Z-kv\left(1-vZ\right)\right]}{Z\left[(1-Zv)^2-\ell(v-Z)^2\right]}=\frac{\Delta_v(Z, v)}{\Delta_Z(Z, v)}
	\end{equation}
	has a unique solution $v\in C^\infty([0, Z_1))$ with $v(0)=0$. 
\end{proposition}

\begin{proof}
	{\textbf{Step 1.} Existence of local solution.} \smallskip
	
	Let $\Phi=\frac vZ$, which satisfies 
	\begin{align*}
		\Phi+Z\frac{\mathrm d\Phi}{\dZ}&=\frac{\dv}{\dZ}=\frac{(1-Z^2\Phi^2)\left[m(1-Z^2\Phi^2)-k(\Phi-Z^2\Phi^2)\right]}{(1-Z^2\Phi)^2-\ell(1-\Phi)^2Z^2}\\
		&=\frac{m-k\Phi+(k-2m)Z^2\Phi^2+kZ^2\Phi^3+(m-k)Z^4\Phi^4}{1-\ell Z^2+2(\ell-1)Z^2\Phi-\ell Z^2\Phi^2+Z^4\Phi^2},
	\end{align*}
	which gives
	\begin{align*}
		&\left[1-\ell Z^2+2(\ell-1)Z^2\Phi-\ell Z^2\Phi^2+Z^4\Phi^2\right]Z\frac{\mathrm d\Phi}{\dZ}=m-(k+1)\Phi\\
		&\qquad +\ell Z^2\Phi+(k-2m-2\ell+2)Z^2\Phi^2+(k+\ell)Z^2\Phi^3-Z^4\Phi^3+(m-k)Z^4\Phi^4.
	\end{align*}
	Let $\varkappa=Z^2$. Then $Z\frac{\mathrm d\Phi}{\dZ}=2\varkappa\frac{\mathrm d\Phi}{\mathrm d\varkappa}$, and the above ODE is equivalent to
	\begin{align*}
		&2\left[1-\ell\varkappa+2(\ell-1)\varkappa\Phi-\ell\varkappa\Phi^2+\varkappa^2\Phi^2\right]\varkappa\frac{\mathrm d\Phi}{\mathrm d\varkappa}=m-(k+1)\Phi\\
		&\qquad +\ell\varkappa\Phi+(k-2m-2\ell+2)\varkappa\Phi^2+(k+\ell)\varkappa\Phi^3-\varkappa^2\Phi^3+(m-k)\varkappa^2\Phi^4,
	\end{align*}
	which can be further rewritten as
	\begin{equation}\label{eq_Phi_Z}
		2\varkappa\frac{\mathrm d\Phi}{\mathrm d\varkappa}+(k+1)\Phi=G,
	\end{equation}
	where
	\begin{align*}
		G&=m+\ell\varkappa\Phi+(k-2m-2\ell+2)\varkappa\Phi^2+(k+\ell)\varkappa\Phi^3-\varkappa^2\Phi^3+(m-k)\varkappa^2\Phi^4\\
		&\qquad-2\left[-\ell\varkappa+2(\ell-1)\varkappa\Phi-\ell\varkappa\Phi^2+\varkappa^2\Phi^2\right]\varkappa\frac{\mathrm d\Phi}{\mathrm d\varkappa}.
	\end{align*}
	We plug the formal expansion
	\[\Phi=\sum_{n=0}^\infty \phi_n\varkappa^n,\qquad G=\sum_{n=0}^\infty g_n\varkappa^n\]
	into \eqref{eq_Phi_Z}, then we get $(k+1+2n)\phi_n=g_n$ for each $n\geq0$. For a function $f$, we denote $f_n=\frac{f^{(n)}(0)}{n!}$, then by Leibniz's rule,
	\begin{equation}\label{Eq.Leibniz}
		(fg)_n=\frac1{n!}\frac{\mathrm d^n}{\mathrm d\varkappa^n}(fg)(0)=\frac1{n!}\sum_{n_1+n_2=n}\frac{n!}{n_1!n_2!}f^{(n_1)}(0)g^{(n_2)}(0)=\sum_{n_1+n_2=n}f_{n_1}g_{n_2}.
	\end{equation}
	Therefore, \eqref{eq_Phi_Z} yields the induction relation: $(k+1)\phi_0=m$ and for $n\geq1$,
		\begin{align}
			(k+1+2n)\phi_n=&\ell\phi_{n-1}+(k-2m-2\ell+2)\sum_{n_1+n_2=n-1}\phi_{n_1}\phi_{n_2}\nonumber\\
			&+(k+\ell)\sum_{n_1+n_2+n_3=n-1}\phi_{n_1}\phi_{n_2}\phi_{n_3}-\sum_{n_1+n_2+n_3=n-2}\phi_{n_1}\phi_{n_2}\phi_{n_3}\nonumber\\
			&+(m-k)\sum_{n_1+n_2+n_3+n_4=n-2}\phi_{n_1}\phi_{n_2}\phi_{n_3}\phi_{n_4}+2\ell(n-1)\phi_{n-1}\label{eq_phi_n_induction}\\
			&-4(\ell-1)\sum_{n_1+n_2=n-1}\phi_{n_1}(n_2\phi_{n_2})+2\ell\sum_{n_1+n_2+n_3=n-1}\phi_{n_1}\phi_{n_2}(n_3\phi_{n_3})\nonumber\\
			&-2\sum_{n_1+n_2+n_3=n-2}\phi_{n_1}\phi_{n_2}(n_3\phi_{n_3}).\nonumber
		\end{align}
	
	Next we prove that the corresponding power series is analytic in a small neighborhood of the origin. We inductively assume that for some $n\geq1$ we have
	\[|\phi_i|\leq \mathfrak{C}_i M^i\quad \text{for all \ }0\leq i\leq n-1,\]
	where $M>1$ is a big constant to be determined later, and $\mathfrak{C}_i$\index{Catalan numbers $\mathfrak{C}_i$} denotes the Catalan numbers defined by $\mathfrak{C}_0=1$ and the recurrence relation $\mathfrak{C}_n=\sum_{i=1}^n\mathfrak{C}_{i-1}\mathfrak{C}_{n-i}$ for $n\geq1$. Combining the Catalan $k$-fold convolution formula (see Lemma 27 in \cite{BR})
	\[\sum_{i_1+\cdots+i_m=n\atop i_1,\ldots,i_m\ge 0} \mathfrak{C}_{i_1}\cdots \mathfrak{C}_{i_m} = \begin{cases}
		\displaystyle \frac{m(n+1)(n+2)\cdots (n+m/2-1)}{2(n+m/2+2)(n+m/2+3)\cdots (n+m)}\mathfrak{C}_{n+\frac m2}, & m \text{ even,}\\
		\displaystyle\frac{m(n+1)(n+2)\cdots (n+(m-1)/2)}{(n+(m+3)/2)(n+(m+3)/2+1)\cdots (n+m)}\mathfrak{C}_{n+\frac{m-1}2}, & m \text{ odd}
	\end{cases}\]
	with $\mathfrak{C}_{n-1}\sim \mathfrak{C}_n$, \eqref{eq_phi_n_induction} and the inductive assumption, we infer
	\begin{align*}
		\frac{|\phi_n|}{\mathfrak{C}_nM^n}&\leq \frac{C_0}{nM}+\frac{C_0}{nM\mathfrak{C}_n}\sum_{n_1+n_2=n-1}\mathfrak{C}_{n_1}\mathfrak{C}_{n_2}+
\frac{C_0}{nM\mathfrak{C}_n}\sum_{n_1+n_2+n_3=n-1}\mathfrak{C}_{n_1}\mathfrak{C}_{n_2}\mathfrak{C}_{n_3}\\
		&\qquad+\frac{C_0}{nM^2\mathfrak{C}_n}\sum_{n_1+n_2+n_3=n-2}\mathfrak{C}_{n_1}\mathfrak{C}_{n_2}\mathfrak{C}_{n_3}
+\frac{C_0}{nM^2\mathfrak{C}_n}\sum_{n_1+n_2+n_3+n_4=n-2}\mathfrak{C}_{n_1}\mathfrak{C}_{n_2}\mathfrak{C}_{n_3}\mathfrak{C}_{n_4}\\
		&\qquad+\frac{C_0}{M}+\frac{C_0}{M\mathfrak{C}_n}\sum_{n_1+n_2=n-1}\mathfrak{C}_{n_1}\mathfrak{C}_{n_2}+
\frac{C_0}{M\mathfrak{C}_n}\sum_{n_1+n_2+n_3=n-1}\mathfrak{C}_{n_1}\mathfrak{C}_{n_2}\mathfrak{C}_{n_3}\\
		&\qquad+\frac{C_0}{M^2\mathfrak{C}_n}\sum_{n_1+n_2+n_3=n-2}\mathfrak{C}_{n_1}\mathfrak{C}_{n_2}\mathfrak{C}_{n_3}\\
		&\leq \frac{2C_0}{M}+\frac{2C_0}{M\mathfrak{C}_n}\mathfrak{C}_n+\frac{2C_0}{M\mathfrak{C}_n}\frac{3n}{n+2}\mathfrak{C}_n+
\frac{2C_0}{M^2\mathfrak{C}_n}\frac{3(n-1)}{n+1}\mathfrak{C}_{n-1}+\frac{C_0}{nM^2\mathfrak{C}_n}\frac{2(n-1)}{n+2}\mathfrak{C}_n\\
		&\leq \frac{C_1}{M},
	\end{align*}
	where $C_0, C_1>0$ are constants independent of $n$ and $M$. Hence, by taking $M>C_1$, we obtain $|\phi_n|\leq \mathfrak{C}_nM^n$,  and then close the induction. Due to  $\mathfrak{C}_n\leq 4^n$, we know that the power series $\Phi(\varkappa)=\sum_{n=0}^\infty \phi_n\varkappa^n$ is analytic in a small neighborhood of the origin $\varkappa=0$, i.e., $Z=0$. Recalling that $v(Z)=Z\Phi(Z^2)$, we have proved that the ODE
	\[\frac{\dv}{\dZ}=\frac{(1-v^2)\left[m(1-v^2)Z-kv\left(1-vZ\right)\right]}{Z\left[(1-Zv)^2-\ell(v-Z)^2\right]}=\frac{\Delta_v(Z, v)}{\Delta_Z(Z, v)}\]
	has an unique analytic local solution $v=v(Z), Z\in[0, \varepsilon_0]$ for some small $\varepsilon_0>0$ with the asymptotics
	\[v(Z)=\frac{m}{k+1}Z+\phi_1Z^3+O(Z^5)\quad \text{as }\quad Z\downarrow0,\]
	where $\phi_1$ is given by
	\begin{equation}\label{Eq.phi_1_1}
		(k+3)\phi_1=\ell\phi_0+(k-2m-2\ell+2)\phi_0^2+(k+\ell)\phi_0^3.
	\end{equation}
	
	{\textbf{Step 2.} Extension up to $P_1$.}\smallskip

	 Let $v_F=v_F(Z)$\index{$v_F(Z)$} be the unique curve entering $P_0$ constructed as above. It follows from the asymptotic behavior of $v_F$ that
	\[\Delta_v(Z, v_F(Z))=(1-v_F^2)\left((k-m)Zv_F^2-kv_F+mZ\right)=\frac m{k+1}Z+O(Z^3)\quad \text{as }\quad Z\downarrow0,\]
	and $v_F'(Z)>0$ if $Z\in(0, \wt\varepsilon_0)$ for some small $\wt\varepsilon_0\in(0, \varepsilon_0)$, hence $\Delta_Z(Z, v_F(Z))>0$ for $Z\in(0,\wt\varepsilon_0)$. It follows from \eqref{Eq.v_12} and \eqref{Eq.v_pm} that $v_-(0)=-\frac1{\sqrt\ell}<0$ and $v_1(Z)\sim \frac mkZ$ as $Z\downarrow0$, so
	\[v_-(Z)<v_F(Z)<v_1(Z)\qquad \text{for }\quad Z\in(0,\wt\varepsilon_0)\]
	by adjusting $\wt\varepsilon_0>0$ to a smaller number if necessary. We consider the open region
	\[\mathcal{R}_F=\left\{(Z, v): v_-(Z)<v<v_1(Z),\  Z\in(0, Z_1)\right\}.\]
	In the region $\mathcal R_F$, we have $\Delta_v>0$ and $\Delta_Z>0$. Assume that the maximal interval of existence of the solution $v_F$ (such that $(Z,v)\in\mathcal{R}_F$)
is $(0,Z_0)$. Then $Z_0\in(0,Z_1]$, $ (Z,v_F(Z))\in \mathcal R_F$, $v_F'(Z)=\frac{\Delta_v(Z,v_F(Z))}{\Delta_Z(Z,v_F(Z))}>0$ for $Z\in(0, Z_0)$.
Thus, the limit $v_0:=\lim_{Z\uparrow Z_0}v_F(Z)$ exists, and $v_-(Z_0)\leq v_0\leq v_1(Z_0) $. By Remark \ref{rem1}, we have $ (Z_0,v_0)\not\in \mathcal{R}_F$.
If $Z_0\in (0, Z_1)$, then
	\begin{itemize}
		\item either $v_0=v_-(Z_0)$, in which case we have $v_F'(Z_0)=\lim_{Z\uparrow Z_0}v_F'(Z)=+\infty$, and we have a contradiction since $v_-(Z)<v_F(Z)$
for all $Z\in(0, Z_0)$, which implies 
		\[\frac{v_F(Z)-v_F(Z_0)}{Z-Z_0}=\frac{v_F(Z)-v_-(Z_0)}{Z-Z_0}<\frac{v_-(Z)-v_-(Z_0)}{Z-Z_0},\ \ \text{for all }\ Z\in(0,Z_0);\]
		\item or $v_0=v_1(Z_0)$, in which case we have $v_F'(Z_0)=\lim_{Z\uparrow Z_0}v_F'(Z)=0$, and we also have a contradiction since $v_F(Z)<v_1(Z)$ for all $Z\in(0, Z_0)$ and $v_1'(Z_0)>0$.
	\end{itemize}
Therefore, $Z_0=Z_1$, the curve $v_F(Z)$ is defined on $[0, Z_1)$ and  $C^\infty$ on this interval, and also satisfies
$ (Z,v_F(Z))\in \mathcal{R}_F$, $v_F'(Z)>0$ and (using $0<v_1(Z)<v_1(Z_1)=v_1=v_-(Z_0)<1$) 
	\begin{align}
\label{vF1}&\Delta_v(Z,v_F(Z))>0,\quad \Delta_Z(Z,v_F(Z))>0,\quad v_F(Z)\in(0,1),\quad \text{for }\ Z\in(0,Z_1),\\
&\label{v01}{\lim_{Z\uparrow Z_1}v_F(Z)=\lim_{Z\uparrow Z_0}v_F(Z)=v_0=v_1},\\
\label{Eq.v_F_asymptotic}
		&v_F(Z)=\frac{m}{k+1}Z+\phi_1Z^3+O(Z^5)\quad \text{as}\qquad Z\downarrow0.
	\end{align}
	
	This completes the proof.
\end{proof}

\subsection{Renormalization}\label{Subsec.re-normalization}
In the last subsection, we construct the unique smooth $P_0-P_1$ solution $v_F(Z)$ to \eqref{Eq.mainODE}. However, it remains unknown whether $v_F$ is smooth at $P_1$ or not. Our idea is to construct a smooth local solution to \eqref{Eq.mainODE} near $P_1$ that passes through $P_1$ and then prove that this local solution can be connected to $v_F$. We can again use the power series method to construct a smooth local solution to \eqref{Eq.mainODE} near $P_1$. However, since we have the $v^4$ nonlinear term in \eqref{Eq.mainODE}, it will be quite complicated if we directly substitute the power series $v=\sum_{n=0}^\infty (v)_n(Z-Z_1)^n$ into \eqref{Eq.mainODE} (noting that the center of the power series is $Z_1\neq0$). As a consequence, we introduce a change of variables $(Z, v)\mapsto(z, u)$ such that in the $z-u$ ODE \eqref{Eq.ODE_P_1}, all nonlinear terms are quadratic, and moreover the point $P_1$ is transformed into $Q_1(z=0, u=\varepsilon)$. Then we substitute the Maclaurin series $u=\sum_{n=0}^\infty a_nz^n$ into the $z-u$ ODE (noting that the center of the power series is $0$ now). This will simplify our computations to a great extent, see Section \ref{Sec.local_P_1} for the details. Another advantage of our new formulation \eqref{Eq.ODE_P_1} is that it allows us to get a larger $\ell_*(k)$ than the critical value gotten through a direct analysis on the original $Z-v$ ODE \eqref{Eq.mainODE}.

\begin{lemma}\label{Lem.renormalization}
	Let
	\begin{equation}\label{Eq.change_of_variables}
		z=\frac{m(1-v^2)Z-kv\left(1-vZ\right)}{mZ(1-v^2)},\qquad u=\frac{k^2(1-vZ)^2}{m^2(1-v^2)Z^2}.
	\end{equation}
	Then \eqref{Eq.mainODE} is mapped to a quasi-linear problem
	\begin{equation}\label{Eq.ODE_P_1}
		\left[(1-(k+1)z)u+(1-z)f(z)\right]{\du}+2u\left[u+f(z)+kz(1-z)\right]{\dz}=0,
	\end{equation}
	where $f(z)=-\varepsilon-Az+Bz^2$\index{$f(z)$} with
	\begin{equation}\label{Eq.epsilon_A_B_gamma}\index{$\varepsilon, A, B, \gamma$}
		\varepsilon=\ell\g^2-1,\quad A=k+2-(k-2\ell)\g, \quad B=2k+1-\ell, \quad \gamma=\frac km-1.
	\end{equation}
\end{lemma}
\begin{remark}
	The inversion of \eqref{Eq.change_of_variables} is
	\begin{equation}\label{Eq.inversion_transform}
		Z=\frac{\sqrt{u+(1-z)^2}}{\frac mku+1-z},\qquad v=\frac{1-z}{\sqrt{u+(1-z)^2}}.
	\end{equation}
\end{remark}
\begin{remark}
	In the quasi-linear formulation \eqref{Eq.ODE_P_1}, the point $P_1$ corresponds to the point $Q_1(z=0, u=\varepsilon)$ and $P_0$ corresponds to the point $Q_0\left(z=\frac1{k+1}, u=+\infty\right)$.\index{$Q_0, Q_1$}
\end{remark}
\begin{remark}
	It follows from \eqref{Eq.parameter_range} that $\g>\frac1{\sqrt\ell}$, and thus $\varepsilon=\ell\g^2-1>0$.
\end{remark}
\begin{proof}[Proof of Lemma \ref{Lem.renormalization}]
We denote by $\Psi=(\Psi_z, \Psi_u)$\index{$\Psi=(\Psi_z, \Psi_u)$} the map $(Z, v)\mapsto(z, u)$, hence using \eqref{Eq.epsilon_A_B_gamma},
	\begin{equation}\label{Eq.Psi_z_u}
		\Psi_z(Z, v)=\frac{(1+\g v^2)Z-(1+\g)v}{Z(1-v^2)},\qquad \Psi_u(Z, v)=\frac{(1+\g)^2(1-Zv)^2}{Z^2(1-v^2)}.
	\end{equation}
	The inversion of $\Psi$ is denoted by $\Theta=(\Theta_Z, \Theta_v)$\index{$\Theta=(\Theta_Z, \Theta_v)$}, hence
	\begin{equation}\label{Eq.Theta}
		(Z, v)=\Theta(z, u)=\left(\frac{(1+\g)\sqrt{u+(1-z)^2}}{u+(1+\g)(1-z)}, \frac{1-z}{\sqrt{u+(1-z)^2}}\right).
	\end{equation}
	Using the notation of differential forms, we have
	\[\du=\frac{\pa\Psi_u}{\pa Z}\dZ+\frac{\pa\Psi_u}{\pa v}\dv,\qquad \dz=\frac{\pa\Psi_z}{\pa Z}\dZ+\frac{\pa\Psi_z}{\pa v}\dv.\]
	If $v$ solves \eqref{Eq.mainODE}, then
	\begin{align}\label{Eq.du}
		&\frac{\mathrm{d}u}{\mathrm{d}Z}=\frac{\pa\Psi_u}{\pa Z}+\frac{\pa\Psi_u}{\pa v}\frac{\mathrm{d}v}{\mathrm{d}Z}=\frac{\pa\Psi_u}{\pa Z}+\frac{\pa\Psi_u}{\pa v}\frac{\Delta_v}{\Delta_Z}=\frac{1}{\Delta_Z}\left(\frac{\pa\Psi_u}{\pa Z}\Delta_Z+\frac{\pa\Psi_u}{\pa v}{\Delta_v}\right),\\&\label{Eq.dz}\frac{\mathrm{d}z}{\mathrm{d}Z}=\frac{\pa\Psi_z}{\pa Z}+\frac{\pa\Psi_z}{\pa v}\frac{\mathrm{d}v}{\mathrm{d}Z}=\frac{\pa\Psi_z}{\pa Z}+\frac{\pa\Psi_z}{\pa v}\frac{\Delta_v}{\Delta_Z}=\frac{1}{\Delta_Z}\left(\frac{\pa\Psi_z}{\pa Z}\Delta_Z+\frac{\pa\Psi_z}{\pa v}{\Delta_v}\right).
	\end{align}
	A direct computation gives
	\begin{align}
		\frac{\pa\Psi_z}{\pa Z}=\frac{(1+\g)v}{Z^2(1-v^2)},\qquad \frac{\pa\Psi_z}{\pa v}=\frac{(1+\g)[1-v^2-2(1-Zv)]}{Z(1-v^2)^2},\label{Eq.dPsi_z}\\
		\frac{\pa\Psi_u}{\pa Z}=-\frac{2(1+\g)^2(1-Zv)}{Z^3(1-v^2)},\qquad \frac{\pa\Psi_u}{\pa v}=\frac{2(1+\g)^2(1-Zv)(v-Z)}{Z^2(1-v^2)^2}.\label{Eq.dPsi_u}
	\end{align}
	By \eqref{Eq.epsilon_A_B_gamma}, we may rewrite $\Delta_v, \Delta_Z$ as
	\begin{align}\label{Dev}
		\Delta_v(Z, v)&=\frac{k(1-v^2)}{1+\g}\left[Z(1-v^2)-(1+\g)v(1-Zv)\right],\\
		\label{DeZ}\Delta_Z(Z, v)&=Z\left[(1-Zv)^2-\ell(v-Z)^2\right].
	\end{align}
	For further usage, we compute from \eqref{Eq.Theta} that
	\begin{equation}\label{Eq.4.14}
		\begin{aligned}
			1-v^2&=\frac u{u+(1-z)^2},\qquad 1-Zv=\frac u{u+(1+\g)(1-z)},\\
			v-Z&=\frac{-u(z+\g)}{\left[u+(1+\g)(1-z)\right]\sqrt{u+(1-z)^2}},\qquad \frac vZ=\frac{(1-z)\left[u+(1+\g)(1-z)\right]}{(1+\g)\left[u+(1-z)^2\right]}.
		\end{aligned}
	\end{equation}
	Recalling $f(z)=-\varepsilon-Az+Bz^2$ and \eqref{Eq.epsilon_A_B_gamma}, we obtain
	\begin{align}
		&\quad \frac{\pa\Psi_z}{\pa Z}\Delta_Z+\frac{\pa\Psi_z}{\pa v}\Delta_v=\frac{1+\g}{1-v^2}\frac vZ\left[(1-Zv)^2-\ell(v-Z)^2\right]\nonumber\\
		&\qquad\qquad+\frac k{1-v^2}\left[1-v^2-2(1-Zv)\right]\left[1-v^2-(1+\g)\frac vZ(1-Zv)\right]\label{Eq.4.15}\\
		&\xlongequal{\eqref{Eq.4.14}}\frac{-u\left[\big((1+k)z-1\big)u+(z-1)f(z)\right]}{\left[u+(1-z)^2\right]\left[u+(1+\g)(1-z)\right]},\nonumber
	\end{align}
	and
	\begin{align}
		&\quad \frac{\pa\Psi_u}{\pa Z}\Delta_Z+\frac{\pa\Psi_u}{\pa v}\Delta_v=-\frac{2(1+\g)^2(1-Zv)}{Z^2(1-v^2)}\left[(1-Zv)^2-\ell(v-Z)^2\right]\nonumber\\
		&\qquad\qquad+\frac{2k(1+\g)(1-Zv)(v-Z)}{Z^2(1-v^2)}\left[Z(1-v^2)-(1+\g)v(1-Zv)\right]\nonumber\\
		&=-\frac{2(1+\g)^2(1-Zv)}{Z^2(1-v^2)}\bigg[(1-Zv)^2-\ell(v-Z)^2\label{Eq.4.16}\\
		&\qquad\qquad\qquad-\frac {k(v-Z)}{1+\g}\left[Z(1-v^2)-(1+\g)v(1-Zv)\right]\bigg]\nonumber\\
		&\xlongequal[\eqref{Eq.Theta}]{\eqref{Eq.4.14}}\frac{-2u^2\left[u+f(z)+kz(1-z)\right]}{\left[u+(1-z)^2\right]\left[u+(1+\g)(1-z)\right]}.\nonumber
	\end{align}
	Therefore, \eqref{Eq.ODE_P_1} follows from \eqref{Eq.du}, \eqref{Eq.dz}, \eqref{Eq.4.15} and \eqref{Eq.4.16}.
\end{proof}

As a consequence, it suffices to construct a solution to the $z-u$ ODE \eqref{Eq.ODE_P_1} corresponding to the desired solution of the $Z-v$ ODE \eqref{Eq.mainODE}. See Figure \ref{Fig.Phase_Portrait_zu} for the phase portrait of the  $z-u$ ODE \eqref{Eq.ODE_P_1}.

\begin{figure}
	\centering
	\includegraphics[width=1\textwidth]{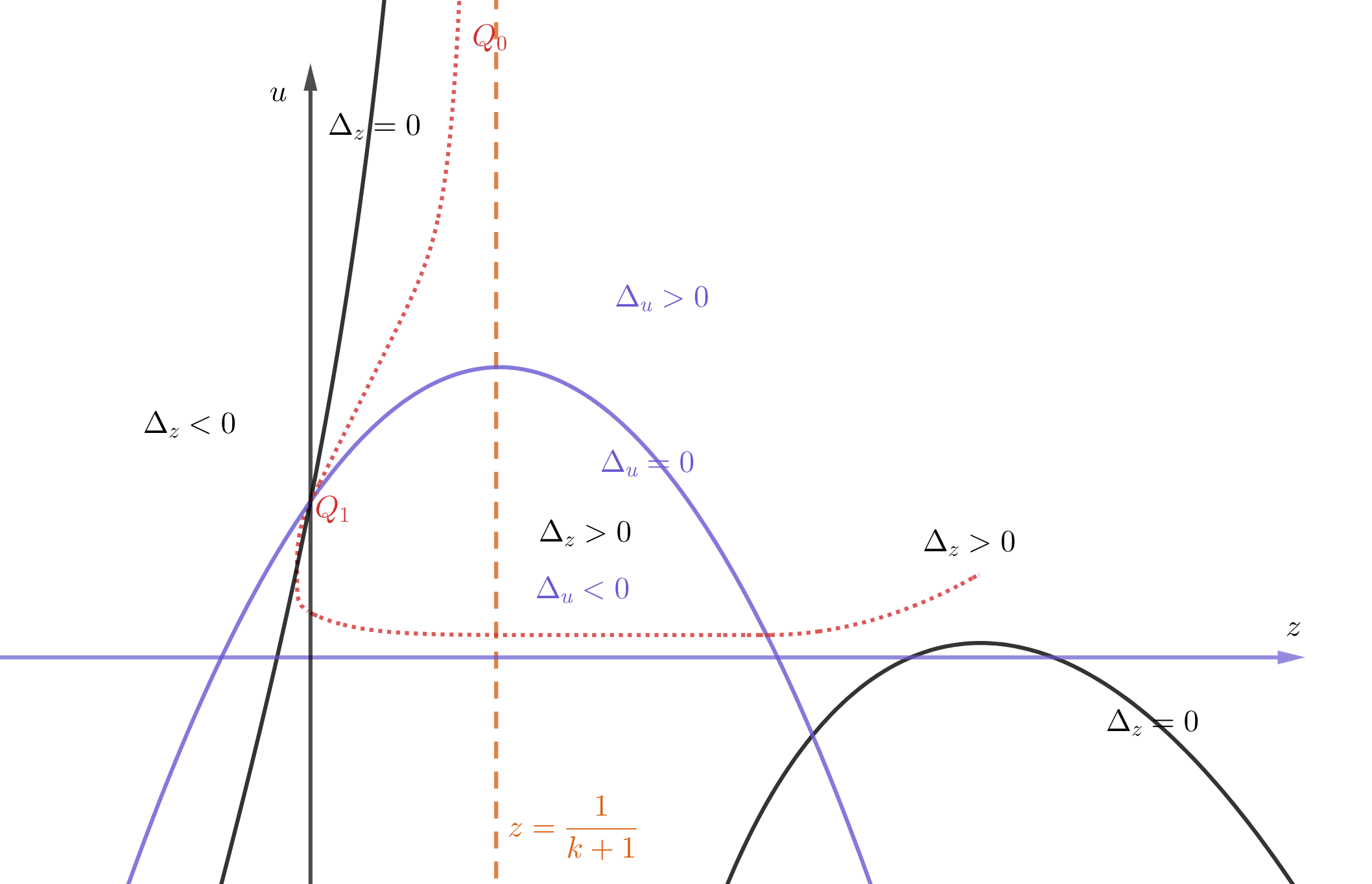}
	\caption{Phase portrait for the $z-u$ system, with $m=1.5, \ell=1.21$ and $k=3$.}
	\label{Fig.Phase_Portrait_zu}
\end{figure}

We rewrite \eqref{Eq.ODE_P_1} as (if $u$ is a function of $z$)
\begin{align}\label{Eq.ODE_z_u}
	\frac{\du}{\dz}=\frac{\Delta_u(z,u)}{\Delta_z(z,u)}=\frac{2u[u+f(z)+kz(1-z)]}{((k+1)z-1)u+(z-1)f(z)}\Leftrightarrow
\mathcal L(u)(z)=0,
\end{align}
where
\begin{align}\label{uz}&\Delta_u(z,u)=2u[u+f(z)+kz(1-z)]\index{$\Delta_u, \Delta_z$},\qquad\Delta_z(z,u)=((k+1)z-1)u+(z-1)f(z),\\
\label{Luz}&\mathcal L(u)(z):=-\Delta_z(z,u){\du}/{\dz}+\Delta_u(z,u).\index{$\mathcal L(u)(z)$}\end{align}

For further usage, we solve $\Delta_u=0$ and $\Delta_z=0$. The solutions to $\Delta_u=0$ are $u=0$ and the purple curve
\begin{equation}\label{Eq.u_p}
	u_{\text{p}}(z)=-f(z)-kz(1-z)=(k-B)z^2+(A-k)z+\varepsilon;\index{$u_\text p(z)$}
\end{equation}
the solution to $\Delta_z=0$ is the black curve
\begin{equation}\label{Eq.u_b}
	u_{\text{b}}(z)=-\frac{(1-z)f(z)}{1-(k+1)z}=\frac{\varepsilon+(A-\varepsilon)z-(A+B)z^2+Bz^3}{1-(k+1)z}.\index{$u_\text b(z)$}
\end{equation}

Now we prove that the change-of-variables mapping $\Psi$ defined in \eqref{Eq.Psi_z_u} is a bijection in the regions that we are interested in. Let
\begin{align}
	\mathcal R_0&=\left\{(Z, v): 0<v<1, 0<Zv<1\right\},\label{Eq.R_0}\index{$\mathcal R_0$}\\
	\mathcal D_0&=\left\{(z, u): u>0, z<1\right\}.\label{Eq.D_0}\index{$\mathcal D_0$}
\end{align}

\begin{lemma}\label{Lem.Psi(R_0)=D_0}
	The map $\Psi: \mathcal R_0\to\mathcal D_0$ is a bijection.
\end{lemma}
\begin{proof}
	We first show that $\Psi(\mathcal R_0)\subset\mathcal D_0$. Let $(Z, v)\in\mathcal R_0$ and $(z, u)=\Psi(Z, v)$. Then $Z>0$, $v>0$, $1-v^2>0$, $1-vZ>0$  and
	\begin{align*}
		1-z=\frac{(\g+1)v(1-vZ)}{Z(1-v^2)}>0,\quad u=\frac{(\g+1)^2(1-vZ)^2}{(1-v^2)Z^2}>0,
	\end{align*}
	thus $(z,u)\in\mathcal D_0$. Hence $\Psi(\mathcal R_0)\subset\mathcal D_0$. Next we show that $\Theta(\mathcal D_0)\subset\mathcal R_0$. Let $(z, u)\in\mathcal D_0$ and $(Z, v)=\Theta(z, u)$, then $u>0$, $1-z>0$, $1-z<\sqrt{u+(1-z)^2}$, thus $0<v<1$, $Z>0$, $ Zv=\frac{1-z}{\frac u{1+\g}+1-z}<1$, and $(Z, v)\in\mathcal R_0$. Hence $\Theta(\mathcal D_0)\subset\mathcal R_0$. Finally, it is direct to check that $\Theta(\Psi(Z, v))=(Z, v)$
	for all $(Z, v)\in\mathcal R_0$ and $\Psi(\Theta(z, u))=(z, u)$ for all $(z, u)\in\mathcal D_0$. Therefore, $\Psi: \mathcal R_0\to\mathcal D_0$ is a bijection.
\end{proof}

\begin{remark}\label{rem2}
	We define
	\begin{equation*}
		\mathcal N(z, u)={\left[u+(1-z)^2\right]\left[u+(1+\g)(1-z)\right]}/{u}\in C^\infty(\mathcal D_0).\index{$\mathcal N(z, u)$}
	\end{equation*}Then  \eqref{Eq.4.15} and \eqref{Eq.4.16} become(with $(z,u)=\Psi(Z,v)=(\Psi_z,\Psi_u)$)\begin{align}
		\label{Eq1}& \frac{\pa\Psi_z}{\pa Z}\Delta_Z+\frac{\pa\Psi_z}{\pa v}\Delta_v=\frac{-\Delta_z(z,u)}{\mathcal N(z, u)},
		\\
		\label{Eq2}& \frac{\pa\Psi_u}{\pa Z}\Delta_Z+\frac{\pa\Psi_u}{\pa v}\Delta_v=\frac{-\Delta_u(z,u)}{\mathcal N(z, u)}.
	\end{align}
	If $(z,u)$ solves \eqref{Eq.ODE_P_1}, let $ (Z,v)=\Theta(z,u)$, then $(z,u)=\Psi(Z,v)$ and\begin{align*}
		0=&\Delta_z\mathrm{d}u-\Delta_u\mathrm{d}z=\Delta_z\left(\frac{\pa\Psi_u}{\pa Z}\mathrm{d}Z+\frac{\pa\Psi_u}{\pa v}\mathrm{d}v\right)-\Delta_u\left(\frac{\pa\Psi_z}{\pa Z}\mathrm{d}Z+\frac{\pa\Psi_z}{\pa v}\mathrm{d}v\right)\\
		=&-\mathcal N\left(\frac{\pa\Psi_z}{\pa Z}\Delta_Z+\frac{\pa\Psi_z}{\pa v}\Delta_v\right)\left(\frac{\pa\Psi_u}{\pa Z}\mathrm{d}Z+\frac{\pa\Psi_u}{\pa v}\mathrm{d}v\right)\\&+\mathcal N\left(\frac{\pa\Psi_u}{\pa Z}\Delta_Z+\frac{\pa\Psi_u}{\pa v}\Delta_v\right)\left(\frac{\pa\Psi_z}{\pa Z}\mathrm{d}Z+\frac{\pa\Psi_z}{\pa v}\mathrm{d}v\right)\\
		=&\mathcal N\left(\frac{\pa\Psi_z}{\pa v}\frac{\pa\Psi_u}{\pa Z}-\frac{\pa\Psi_z}{\pa Z}\frac{\pa\Psi_u}{\pa v}\right)(\Delta_Z\mathrm{d}v-\Delta_v\mathrm{d}Z).
	\end{align*}
	By \eqref{Eq.dPsi_z} and \eqref{Eq.dPsi_u}, we have
	\begin{align*}
		\frac{\pa\Psi_z}{\pa v}\frac{\pa\Psi_u}{\pa Z}-\frac{\pa\Psi_z}{\pa Z}\frac{\pa\Psi_u}{\pa v}=\frac{2(1+\g)^3(1-Zv)^2}{Z^4(1-v^2)^3}>0\quad\text{in}\ \mathcal R_0.
	\end{align*}
	We also have $\mathcal N(z, u)>0$ in $\mathcal D_0 $. Thus, $\Delta_Z\mathrm{d}v-\Delta_v\mathrm{d}Z=0 $.
\end{remark}

Let
\begin{align}
	\mathcal R_1&=\left\{(Z, v): v\in(0, v_1), Z\in\left(\frac{(\g+1)v}{1+\g v^2}, \frac{\sqrt\ell v+1}{v+\sqrt\ell}\right)=(Z_\text b(v), Z_g(v))\right\},\label{Eq.R_1}\index{$\mathcal R_1$}\\
	\mathcal R_2&=\left\{(Z, v): v\in(v_1, 1), Z\in\left(\frac{\sqrt\ell v+1}{v+\sqrt\ell}, \frac{(\g+1)v}{1+\g v^2}\right)=(Z_g(v), Z_\text b(v))\right\}.\label{Eq.R_2}\index{$\mathcal R_2$}
\end{align}
Then we have
\begin{align}
\label{R1}	&\mathcal R_1=\left\{(Z, v)\in \mathcal R_0: \Delta_v(Z,v)>0,\ \Delta_Z(Z,v)>0\right\}\subset
\left\{(Z, v)\in \mathcal R_0: v<Z\right\},\\
\label{R2}	&\mathcal R_2=\left\{(Z, v)\in \mathcal R_0: \Delta_v(Z,v)<0,\ \Delta_Z(Z,v)<0,\ v<Z\right\}.
\end{align}
Let $\mathcal D_i=\Theta^{-1}(\mathcal R_i)$. Then $\Psi: \mathcal R_i\to\mathcal D_i$ is a bijection for $i=1,2$, and (here $ \chi(v,Z)=v-Z$)
\begin{align*}
	&\mathcal D_1=\left\{(z, u)\in \mathcal D_0: \Delta_v(\Theta(z, u))>0,\ \Delta_Z(\Theta(z, u))>0\right\},\index{$\mathcal D_1$}\\
	&\mathcal D_2=\left\{(z, u)\in \mathcal D_0: \Delta_v(\Theta(z, u))<0,\ \Delta_Z(\Theta(z, u))<0,\ \chi(\Theta(z, u))<0\right\}.\index{$\mathcal D_2$}
\end{align*}
Now we need to evaluate $ \Delta_v$ and $ \Delta_Z$ in $(z,u)$. Let $(z, u)\in\mathcal D_0$ and $(Z, v)=\Theta(z, u)\in\mathcal R_0$. Then by \eqref{Eq.Psi_z_u}, \eqref{Dev}, \eqref{DeZ} and \eqref{Eq.4.14}, we get
\begin{align}
	\Delta_v(Z, v)&=\frac{k(1-v^2)}{1+\g}\left[Z(1-v^2)-(1+\g)v(1-Zv)\right]=\frac{k(1-v^2)^2Z}{1+\g}z,\label{Eq.Delta_v}\\
	\Delta_Z(Z, v)&=Z\left[(1-Zv)^2-\ell(v-Z)^2\right]=\frac{Zu^2\left[u+(1-z)^2-\ell(z+\g)^2\right]}{\left[u+(1-z)^2\right]\left[u+(1+\g)(1-z)\right]^2}.\label{Eq.Delta_Z}
\end{align}
Therefore, we  obtain (as $u>0$, $1-z>0$, $Z>0$, $0<v<1$, $k>0$, $\gamma>0$)
\begin{align*}
	&\mathcal D_1=\left\{(z, u)\in \mathcal D_0: z>0,\ u+(1-z)^2-\ell(z+\gamma)^2>0\right\},\\
	&\mathcal D_2=\left\{(z, u)\in \mathcal D_0: z<0,\ u+(1-z)^2-\ell(z+\gamma)^2<0,\ z+\gamma>0\right\}.
\end{align*}
Let
\begin{align}\label{Eq.u_g}	
u_g(z)&=\ell(z+\gamma)^2-(1-z)^2\index{$u_g(z)$}=\ell\g^2-1+2(1+\ell\g)z+(\ell-1)z^2,
\end{align}
then $u_g(z)=(\ell-1)(z-z_g)(z-z_g^-)$ with
\begin{equation*}
	z_g:=\frac{1-\sqrt\ell\g}{\sqrt\ell+1},\quad z_g^-:=-\frac{1+\sqrt\ell\g}{\sqrt\ell-1}, \quad z_g^-<-\g<z_g<0.\index{$z_g, z_g^-$}
\end{equation*}
Thus, we can write $\mathcal D_1, \mathcal D_2$ as
\begin{align*}
	\mathcal D_1=\left\{(z, u): u>u_g(z),\ 0<z<1\right\},\quad \mathcal D_2=\left\{(z, u): 0<u<u_g(z),\ z_g<z<0\right\}.
\end{align*}

Similarly, let
\begin{align*}
	&\mathcal R_2^+:=\left\{(Z, v): v\in(v_1, 1), Z=\frac{\sqrt\ell v+1}{v+\sqrt\ell}=Z_g(v)\right\}.\index{$\mathcal R_2^+$}
\end{align*}
Then we have
\begin{align*}
	&\mathcal R_2^+=\left\{(Z, v)\in \mathcal R_0: \Delta_v(Z,v)<0,\ \Delta_Z(Z,v)=0,\ v<Z\right\}.
\end{align*}
Let $\mathcal D_2^+=\Theta^{-1}(\mathcal R_2^+)$.\index{$\mathcal D_2^+$} Then $\Psi: \mathcal R_2^+\to\mathcal D_2^+$ is a bijection and
\begin{align*}
	\mathcal D_2^+&=\left\{(z, u)\in \mathcal D_0: \Delta_v(\Theta(z, u))<0,\ \Delta_Z(\Theta(z, u))=0,\ \chi(\Theta(z, u))<0\right\}\\
	&=\left\{(z, u)\in \mathcal D_0: z<0,\ u+(1-z)^2-\ell(z+\gamma)^2=0,\ z+\gamma>0\right\}\\
	&=\left\{(z, u): u=u_g(z),\ z_g<z<0\right\}.
\end{align*}

\subsection{The $Q_0-Q_1$ curve}\label{Subsec.Q_0-Q_1_curve}
Here we prove that the $P_0-P_1$ solution curve $v_F$ of the $Z-v$ ODE \eqref{Eq.mainODE} corresponds to a solution curve $u_F$ of the $z-u$ ODE \eqref{Eq.ODE_z_u}  that connects $Q_0$ to $Q_1$, where $Q_0=\left(z=\frac1{k+1}, u=+\infty\right)$, see the dashed red curve in Figure \ref{Fig.Phase_Portrait_zu} connecting $Q_0$ and $Q_1$.

\begin{lemma}\label{Lem.4.8}
	It holds that
	\begin{align}
		&\lim_{Z\downarrow0}\Psi_z(Z, v_F(Z))=\frac1{k+1},\qquad \lim_{Z\downarrow0}\Psi_u(Z, v_F(Z))=+\infty,\\
		&\qquad\qquad\lim_{Z\downarrow0}\frac1{Z}\frac{\mathrm{d}}{\mathrm{d}Z}\Psi_z(Z, v_F(Z))<0.\label{Eq.4.37}
	\end{align}
\end{lemma}
\begin{proof}
	Since $v_F$ has the asymptotic behavior \eqref{Eq.v_F_asymptotic}, we compute
	\begin{align*}
		\lim_{Z\downarrow0}\Psi_u(Z, v_F(Z))&=\lim_{Z\downarrow0}\frac{(1+\g)^2\big(1-Zv_F(Z)\big)^2}{Z^2\big(1-v_F(Z)^2\big)}=(1+\g)^2\lim_{Z\downarrow0}\frac1{Z^2}=+\infty,\\
		\lim_{Z\downarrow0}\Psi_z(Z, v_F(Z))&=\lim_{Z\downarrow0}\frac{1+\g v_F(Z)^2-(1+\g)\frac{v_F(Z)}{Z}}{1-v_F(Z)^2}=1-(1+\g)\frac m{k+1}=\frac1{k+1}.
	\end{align*}
	Also, as $Z\downarrow 0$, we have
	\begin{align*}
		&\frac1{k+1}-\Psi_z(Z, v_F(Z))=\frac1{k+1}-\frac{1+\g v_F(Z)^2-(1+\g)\frac{v_F(Z)}{Z}}{1-v_F(Z)^2}\\
		&=\frac{(1+\g)\frac{v_F(Z)}{Z}+\frac1{k+1}-1-\left(\g+\frac1{k+1}\right)v_F(Z)^2}{1-v_F(Z)^2}\\
		&=\frac{(1+\g)\left(\frac m{k+1}+\phi_1Z^2+O(Z^4)\right)-\frac k{k+1}-\left(\g+\frac1{k+1}\right)\left(\frac{m^2}{(k+1)^2}Z^2+O(Z^4)\right)}{1-v_F(Z)^2}\\
		&=\left[(1+\g)\phi_1-\left(\g+\frac1{k+1}\right)\frac{m^2}{(k+1)^2}\right]Z^2+O(Z^4).
	\end{align*}
	As $\frac{v_F(Z)}{Z} $ is real analytic in ${Z} $, so is $\Psi_z(Z, v_F(Z)) $. Thus, for \eqref{Eq.4.37}, it suffices to show
	\begin{equation}\label{Eq.4.39_1}
		\phi_1>\frac1{1+\g}\left(\g+\frac1{k+1}\right)\frac{m^2}{(k+1)^2}=\frac{k^2\left[(k+1)\g+1\right]}{(k+1)^3(\g+1)^3}.
	\end{equation}
	By \eqref{Eq.phi_1_1} and $m=k/(\g+1)$, we compute
	\begin{align*}
		(k+3)\phi_1&=\ell\frac{k}{(k+1)(\g+1)}+\left(k-2\ell+2-\frac{2k}{\g+1}\right)\frac{k^2}{(k+1)^2(\g+1)^2}+\frac{k^3(k+\ell)}{(k+1)^3(\g+1)^3}\\
		&=\frac{k\Big[\ell(k+1)^2(\g+1)^2+k(k+1)\left[(k+2-2\ell)(\g+1)-2k\right]+k^2(k+\ell)\Big]}{(k+1)^3(\g+1)^3}\\
		&=\frac{k\Big[\ell(k+1)^2\g^2+(k+1)(k^2+2k+2\ell)\g+k^2+2k+\ell\Big]}{(k+1)^3(\g+1)^3}\\
		&=\frac{k[(k+1)\g+1]\Big[\ell(k+1)\g+k^2+2k+\ell\Big]}{(k+1)^3(\g+1)^3}.
	\end{align*}
	Therefore, we arrive at
		\begin{align*}
			&\phi_1-\frac{k^2\left[(k+1)\g+1\right]}{(k+1)^3(\g+1)^3}=\frac{k\left[(k+1)\g+1\right]\big[\ell(k+1)\g+k^2+2k+\ell-k(k+3)\big]}{(k+3)(k+1)^3(\g+1)^3}\\
			&=\frac{k\left[(k+1)\g+1\right]\big[\ell(k+1)\g+\ell-k\big]}{(k+3)(k+1)^3(\g+1)^3}=\frac{k\left[(k+1)\g+1\right]\big[(\ell\g-1)k+\ell+\ell\g\big]}{(k+3)(k+1)^3(\g+1)^3}>0,
		\end{align*}
	where we have used $\ell\g>\ell\cdot\frac1{\sqrt\ell}=\sqrt\ell>1$. This proves \eqref{Eq.4.39_1} and thus \eqref{Eq.4.37}.
	\end{proof}

\begin{lemma}\label{Lem.4.9}
	It holds that
	\begin{align}
		u_\text p(z)<u_g(z)<u_\text b(z)\quad \forall\ z\in\left(0,\frac1{k+1}\right),\label{Eq.u_g_location}
	\end{align}
	where $u_p, u_b, u_g$ are given by \eqref{Eq.u_p}, \eqref{Eq.u_b} and \eqref{Eq.u_g} respectively.  
\end{lemma}
\begin{proof}
	For $z\in\R$, we compute (using \eqref{Eq.epsilon_A_B_gamma})
	\begin{equation}\label{Eq.u_p-u_g}
		u_\text p(z)-u_g(z)=(k-B-\ell+1)z^2+(A-k-2(1+\ell\g))z=-kz(z+\g),
	\end{equation}
	and
	\begin{align*}
		u_\text b(z)-u_g(z)&=\frac{\varepsilon+(A-\varepsilon)z-(A+B)z^2+Bz^3}{1-(k+1)z}-\big(\varepsilon+2(1+\ell\g)z+(\ell-1)z^2\big)\\
		&=\frac{k\g(\ell\g-1)z+k(2\ell\g+\g-1)z^2+k(\ell+1)z^3}{1-(k+1)z},
	\end{align*}
	which gives \eqref{Eq.u_g_location} by using $\ell\g-1>0$ and $2\ell\g+\g-1=\ell\g-1+\ell\g+\g>0$.
\end{proof}

We consider the domain
\begin{align}
	\mathcal D_1'=&\left\{(z, u): u_g(z)<u<u_\text b(z),\quad z\in\left(0,\frac1{k+1}\right)\right\}\\\index{$\mathcal D_1'$}
	\notag=&\left\{(z, u)\in \mathcal D_1: \Delta_z(z,u)>0,\quad z\in\left(0,\frac1{k+1}\right)\right\}.
\end{align}

The main result of this subsection is stated as follows.

\begin{proposition}\label{Prop.Q_0-Q_1}
	The unique $P_0-P_1$ smooth solution $v_F(Z)\ (0<Z<Z_1)$ to the $Z-v$ ODE \eqref{Eq.mainODE} is mapped through $\Psi$ to a smooth solution $u=u_F(z)\ \left(0<z<\frac1{k+1}\right)$\index{$u_F(z)$} of the $z-u$ ODE \eqref{Eq.ODE_z_u} with the following properties:
	\begin{itemize}
		\item $u_F$ is a strictly increasing function with respect to $z\in\left(0,\frac1{k+1}\right)$;
		\item we have \begin{equation}\label{Eq.u_p<u_F<u_b}
			u_\text p(z)<u_g(z)<u_F(z)<u_\text b(z),\quad\Delta_z(z,u_F(z))>0\quad \forall\ z\in\left(0,\frac1{k+1}\right);
		\end{equation}
		\item $\lim_{z\downarrow0}u_F(z)=\varepsilon$, $\lim_{z\uparrow1/(k+1)}u_F(z)=+\infty$.
	\end{itemize}
	Moreover, we have
	\begin{equation}\label{Eq.v_F_expression}
		\Theta_v(z, u_F(z))=v_F\big(\Theta_Z(z, u_F(z))\big)\quad \forall\ z\in\left(0,\frac1{k+1}\right).
	\end{equation}
\end{proposition}

\begin{proof}
By \eqref{vF1}, \eqref{R1} and $Z_1\in(0,1)$, we have $ (Z,v_F(Z))\in\mathcal{R}_1$ for $Z\in(0,Z_1)$, and then $$(z(Z),u(Z)):=\Psi(Z,v_F(Z))\in\mathcal{D}_1\subset\mathcal{D}_0,\qquad \forall\  Z\in(0,Z_1).$$ By \eqref{Eq.du}, \eqref{Eq.dz}, \eqref{Eq1} and \eqref{Eq2}, we get
	\begin{align*}
		\frac{\mathrm{d}z}{\mathrm{d}Z}=\frac{-\Delta_z(z,u)}{\Delta_Z(Z,v_F)\mathcal{N}(z,u)},
		\quad\frac{\mathrm{d}u}{\mathrm{d}Z}=\frac{-\Delta_u(z,u)}{\Delta_Z(Z,v_F)\mathcal{N}(z,u)}.
	\end{align*}
	Let $\eta_1(Z):=[\Delta_Z(Z,v_F(Z))\mathcal{N}(z(Z),u(Z))]^ {-1}$. As $\Delta_Z>0 $ in $ \mathcal{R}_1$ and $\mathcal{N}>0 $ in $ \mathcal{D}_0$, we have 
	\begin{align}\label{dz}
		\frac{\mathrm{d}z}{\mathrm{d}Z}=-\eta_1(Z)\Delta_z(z,u),
		\quad\frac{\mathrm{d}u}{\mathrm{d}Z}=-\eta_1(Z)\Delta_u(z,u),\quad \eta_1(Z)>0\quad \text{for}\ Z\in(0, Z_1).
	\end{align}
	According to Lemma \ref{Lem.4.8}, there exists a small $\Upsilon_0>0$ such that for all $Z\in(0, \Upsilon_0)$, $\frac{\mathrm{d}z}{\mathrm{d}Z}<0 $ and $ \lim_{Z\downarrow0}z(Z)=\frac{1}{k+1}$. Then $0<z(Z)<\frac{1}{k+1} $ for all $Z\in(0, \Upsilon_0)$. By \eqref{dz}, we have $ \Delta_z(z(Z),u(Z))>0$ for all $Z\in(0, \Upsilon_0)$. 
	Now we claim that 
	$$ \Delta_z(z(Z),u(Z))>0\quad \forall\ Z\in(0, Z_1).$$ 
	Otherwise, there exists $Z_2\in(0, Z_1)$ such that $ \Delta_z(z(Z),u(Z))>0$ for all $Z\in(0, Z_2)$ and $ \Delta_z(z(Z_2),u(Z_2))=0$. Then $\frac{\mathrm{d}}{\mathrm{d}Z}\Delta_z(z,u)|_{Z=Z_2}\leq0 $.
	By \eqref{dz}, we have $\frac{\mathrm{d}z}{\mathrm{d}Z}<0 $ for all $Z\in(0, Z_2)$, $\frac{\mathrm{d}z}{\mathrm{d}Z}|_{Z=Z_2}=0 $ and $0<z(Z)<\frac{1}{k+1} $ for all $Z\in(0, Z_2]$. By $ (z(Z),u(Z))\in\mathcal{D}_1$ and Lemma \ref{Lem.4.9}, we have $u(Z)> u_g(z(Z))> u_{\text p}(z(Z))$, $u(Z)> 0$ and $ \Delta_u(z(Z),u(Z))=2u(Z)(u(Z)- u_{\text p}(z(Z)))>0$ for all $Z\in(0, Z_2]$. Then by \eqref{dz} again, we have $\frac{\mathrm{d}u}{\mathrm{d}Z}<0 $ for all $Z\in(0, Z_2]$. As $(k+1)z-1<0 $ and $\frac{\mathrm{d}u}{\mathrm{d}Z}<0 $ for $z=Z_2$, we also have
	\begin{align*}
		\frac{\mathrm{d}}{\mathrm{d}Z}\Delta_z(z,u)\Big|_{Z=Z_2}&=\left(\partial_z\Delta_z(z,u)\frac{\mathrm{d}z}{\mathrm{d}Z}+
		\partial_u\Delta_z(z,u)\frac{\mathrm{d}u}{\mathrm{d}Z}\right)\Big|_{Z=Z_2}=\partial_u\Delta_z(z,u)\frac{\mathrm{d}u}{\mathrm{d}Z}\Big|_{Z=Z_2}\\
		&=((k+1)z-1)\frac{\mathrm{d}u}{\mathrm{d}Z}\Big|_{Z=Z_2}>0,
	\end{align*}
	which reaches a contradiction. So, $ \Delta_z(z(Z),u(Z))>0$ for all $Z\in(0, Z_1)$.
	
By $ (z(Z),u(Z))\in\mathcal{D}_1$ and Lemma \ref{Lem.4.9}, we have $u(Z)> u_g(z(Z))> u_{\text p}(z(Z))$, and $$ \Delta_u(z(Z),u(Z))=2u(Z)\big(u(Z)- u_{\text p}(z(Z))\big)>0\quad \forall\  Z\in(0, Z_1).$$
Therefore, we conclude by \eqref{dz} that $\frac{\mathrm{d}z}{\mathrm{d}Z}<0 $ and $\frac{\mathrm{d}u}{\mathrm{d}Z}<0$. Hence, both $z(Z)$ and $u(Z)$ are strictly decreasing
	for $Z\in(0, Z_1)$. Thus, $0<z(Z)<\frac{1}{k+1} $ for all $Z\in(0, Z_1)$, which along with $ (z(Z),u(Z))\in\mathcal{D}_1$,
	$ \Delta_z(z(Z),u(Z))>0$ for all $Z\in(0, Z_1)$ implies $ (z(Z),u(Z))\in\mathcal{D}_1'$ for all $Z\in(0, Z_1)$.
	By Lemma \ref{Lem.4.8}, we have 
	\begin{equation}\label{Eq.4.57_1}
		\lim_{Z\downarrow0}z(Z)=\lim_{Z\downarrow0}\Psi_z(Z, v_F(Z))=\frac1{k+1},\quad
		\lim_{Z\downarrow0}u(Z)=\lim_{Z\downarrow0}\Psi_u(Z, v_F(Z))=+\infty.
	\end{equation}
	By $\Psi(Z_1, v_1)=Q_1=(0,\varepsilon)$ and \eqref{v01}, we have
	\begin{equation}\label{Eq.4.58}
		\lim_{Z\uparrow Z_1}(z(Z), u(Z))=\lim_{Z\uparrow Z_1}\Psi(Z, v_F(Z))=Q_1=(0, \varepsilon).
	\end{equation}
	Let $Z(z)\ (0<z<\frac{1}{k+1})$ denote the inverse function of $z=z(Z)\ (0<Z<Z_1)$, and we define $$u_F(z)=u(Z(z))\quad\forall \ z\in\left(0, \frac1{k+1}\right).$$
	
	We check that $u_F$ is a desired solution to the $z-u$ ODE connecting $Q_0$ and $Q_1$. Since $z(Z)$ and $u(Z)$ are strictly decreasing, $u_F$ is also strictly increasing; by \eqref{dz}, we have
	\[\frac{\mathrm{d}u_F}{\mathrm{d}z}=\frac{\mathrm{d}u/\mathrm{d}Z}{\mathrm{d}z/\mathrm{d}Z}=\frac{\Delta_u(z, u_F(z))}{\Delta_z(z, u_F(z))};\]
	by $(z,u_F(z))=(z,u)(Z(z))\in \mathcal{D}_1'$, we have $u_g(z)<u_F(z)<u_\text b(z)$ and $\Delta_z(z,u_F(z))>0$ for all $z\in(0, 1/(k+1))$; by \eqref{Eq.4.57_1} and \eqref{Eq.4.58}, we have
	\[\lim_{z\downarrow0}u_F(z)=\varepsilon,\quad  \lim_{z\uparrow1/(k+1)}u_F(z)=+\infty.\]
	
	Finally, by the definition, we have $(z(Z), u(Z))=\Psi\big(Z, v_F(Z)\big)\in\mathcal D_0$ for all $Z\in(0, Z_1)$, hence by Lemma \ref{Lem.Psi(R_0)=D_0} and the definition of $u_F$, we have $\big(Z, v_F(Z)\big)=\Theta\big(z(Z), u_F(Z)\big)$ for all $Z\in(0, Z_1)$,
	and thus $\Theta_v\big(z(Z), u_F(z(Z))\big)=v_F\Big(\Theta_Z\big(z(Z), u_F(z(Z))\big)\Big)$ for all $Z\in(0, Z_1)$. Now since $z(Z)$ is a bijection from $Z\in(0, Z_1)$ onto $z\in(0, 1/(k+1))$, we deduce \eqref{Eq.v_F_expression}. 
\end{proof}

\section{Local smooth solution around  the sonic point}\label{Sec.local_P_1}

In this section, we construct a smooth local solution $u_L$ to the $z-u$ ODE \eqref{Eq.ODE_z_u} passing through the sonic point $Q_1(z=0, u=\varepsilon)$. In Subsection \ref{Subsec.slope_Q_1}, we first show that all smooth solutions to the $z-u$ ODE \eqref{Eq.ODE_z_u} passing through $Q_1$ has two possible slopes at $Q_1$, and we determine the desired slope of $u_L$ at $Q_1$ using the information of $u_F$ constructed in the previous section.  In Subsection \ref{Subsec.eigenvalues}, we show that the ratio $R>1$ of the eigenvalues of the linearized system of \eqref{Eq.ODE_z_u} at $Q_1$ plays a very important role in the construction of $u_L$. In Subsection \ref{Subsec.local_z_u}, we use the power series method to construct $u_L$, which can be transformed by $\Theta$ to a local solution $v_L(Z)$ of the $Z-v$ ODE passing through $P_1$. 

Let's recall the notations introduced in last section:
\begin{equation}\nonumber
		\varepsilon=\ell\g^2-1,\quad A=k+2-(k-2\ell)\g, \quad B=2k+1-\ell, \quad \gamma=k/m-1.
	\end{equation}

\subsection{Slope at $Q_1$}\label{Subsec.slope_Q_1}
Assume that $u(z)$ is a local smooth solution of \eqref{Eq.ODE_z_u} passing through $Q_1$. Let $a_1=u'(0)$. By L'Hôpital's rule, we have
\[a_1=\frac{\du}{\dz}\Big|_{z=0}=\frac{\pa_z\Delta_u(Q_1)+\pa_u\Delta_u(Q_1)a_1}{\pa_z\Delta_z(Q_1)+\pa_u\Delta_z(Q_1)a_1}.\index{$a_1$}\]
We define
\begin{equation}\label{Eq.Q1_coefficients}
	\begin{aligned}\index{$\tilde c_1, \tilde c_2, \tilde c_3, \tilde c_4$}
		\tilde c_1:&=\pa_u \Delta_u(Q_1)=2\varepsilon,\qquad \qquad \tilde c_2:=\pa_u \Delta_z(Q_1)=-1,\\
		\tilde c_3:&=\pa_z \Delta_u(Q_1)=2\varepsilon(k-A),\quad \tilde c_4:=\pa_z \Delta_z(Q_1)=k\varepsilon+A.
	\end{aligned}
\end{equation}
Then $a_1$ satisfies the quadratic equation $-\tilde c_2a_1^2+(\tilde c_1-\tilde c_4)a_1+\tilde c_3=0$, i.e.,
\begin{equation}\label{Eq.quad_u_1}\index{$p_0$}
	p_0(a_1):=a_1^2-((k-2)\varepsilon+A)a_1+2(k-A)\varepsilon=0.
\end{equation}
Since
\begin{align*}
	&p_0(2(1+\ell\g))=\big(2(1+\ell\g)\big)^2-2\big((k-2)\varepsilon+A\big)(1+\ell\g)+2(k-A)\varepsilon\\
	=&4(1+\ell\g)^2-2(1+\ell\g)\big(k(\ell\g^2-\g)+4-2\ell\g^2+2\ell\g\big)+2k\g(\ell\g^2-1)-4(1+\ell\g)(\ell\g^2-1)\\
	=&2k\ell\g^3(1-\ell)<0,
\end{align*}
we know that the quadratic polynomial $p_0$ has two real roots, where the larger one is greater than $2(1+\ell\g)>4>0$ and  the smaller one is smaller than $2(1+\ell\g)$, and the discriminant of $p_0$ is $(\tilde c_1-\tilde c_4)^2+4\tilde c_2\tilde c_3>0 $. Similarly, we have
\begin{align}\label{p0g}
	p_0({2}(1+s\g))&=2(2s+k)(s-\ell)\g^2+2(k-2\ell-(k-2)s)\ell\g^3,
\end{align}
whose proof will be delayed until we prove Lemma \ref{Lem.A.3}.

We claim that $a_1$ must be the larger root of $p_0$. In fact, as $u_F(z)>u_g(z)$ for $z\in(0,1/(k+1))$, $u_F(0)=\varepsilon=\ell\g^2-1=u_g(0)$, we have $a_1= u'(0)=u_F'(0)\geq u_g'(0)=2(1+\ell\g)$. Thus
\begin{align}\label{Eq.a_1>4}
	&a_1>2(1+\ell\g)>4,
\\ \label{Eq.a_1_expression}
	&a_1=\frac{\tilde c_4-\tilde c_1+\sqrt{(\tilde c_1-\tilde c_4)^2+4\tilde c_2\tilde c_3}}{2}
	=\frac{(k-2)\varepsilon+A+\sqrt{\big((k-2)\varepsilon+A\big)^2-8(k-A)\varepsilon}}{2}.
\end{align}

\subsection{Eigenvalues}\label{Subsec.eigenvalues}
We start with some heuristic discussions. Consider the ODE
\begin{equation}\label{Eq.heuristiceq}
	\frac{\mathrm dy}{\mathrm dx}=\frac{a_{11}y+a_{12}x+R_1(x,y)}{a_{21}y+a_{22}x+R_2(x,y)},
\end{equation}
where $a_{11}, a_{12}, a_{21}, a_{22}$ are constants and $R_1, R_2$ are smooth functions with $R_1(x,y)=O(x^2)+O(y^2)$, $R_2(x,y)=O(x^2)+O(y^2)$ as $(x,y)\to(0,0)$. We introduce a time variable $\tau$ such that $\frac{\mathrm dx}{\mathrm d\tau}=a_{21}y+a_{22}x+R_2(x,y)$, then \eqref{Eq.heuristiceq} is converted to the ODE system
\begin{equation}\label{Eq.heuristic_ODE_system}
	\frac{\mathrm d}{\mathrm d\tau}\begin{pmatrix} y\\ x \end{pmatrix}=\begin{pmatrix}a_{11} & a_{12}\\ a_{21}& a_{22}\end{pmatrix}\begin{pmatrix} y\\ x \end{pmatrix}+\begin{pmatrix} R_1(x,y)\\ R_2(x,y) \end{pmatrix}.
\end{equation}
Clearly, \eqref{Eq.heuristic_ODE_system} has a trivial solution $(x,y)=(0,0)$. Near this equilibrium, the behavior of the solutions of \eqref{Eq.heuristic_ODE_system} can be approximated by the linear ODE system
\begin{equation}\label{Eq.heuristic_linear_ODE}
	\frac{\mathrm d}{\mathrm d\tau}\begin{pmatrix} y\\ x \end{pmatrix}=\mathcal{A}\begin{pmatrix} y\\ x \end{pmatrix}=\begin{pmatrix}a_{11} & a_{12}\\ a_{21}& a_{22}\end{pmatrix}\begin{pmatrix} y\\ x \end{pmatrix}.
\end{equation}
We assume that the matrix $\mathcal A$ has two distinct positive eigenvalues $\lambda_1>\lambda_2>0$ with $R=\frac{\lambda_1}{\la_2}\in(1,+\infty)\setminus\Z$. Let $\mathcal P$ be a $2\times 2$ matrix such that
\[\mathcal P\mathcal A\mathcal P^{-1}=\begin{pmatrix}\lambda_1 & \\ & \lambda_2\end{pmatrix},\]
and let
\begin{equation}\label{Eq.heuristic_changevariable}
	\begin{pmatrix} Y\\ X \end{pmatrix}=\mathcal P\begin{pmatrix} y\\ x \end{pmatrix},
\end{equation}
then the linear ODE system \eqref{Eq.heuristic_linear_ODE} is reduced to
\begin{equation}\label{Eq.heuristic_diagonal_ODE}
	\frac{\mathrm d}{\mathrm d\tau}\begin{pmatrix} Y\\ X \end{pmatrix}=\begin{pmatrix}\lambda_1 & \\ & \lambda_2\end{pmatrix}\begin{pmatrix} Y\\ X \end{pmatrix}.
\end{equation}
Thus, $Y(\tau)=Y(0)e^{\la_1\tau}, X(\tau)=X(0)e^{\la_2\tau}$. We observe that the curve $\tau\mapsto(X(\tau), Y(\tau))$ gives a smooth curve in the phase space $(X,Y)$ if and only $X(0)=0$ or $Y(0)=0$, which corresponds to the solution $X\equiv 0$ or $Y\equiv0$ respectively; and if $X(0)\neq 0$ and $Y(0)\ne 0$, the solution satisfies $Y=C_0X^R$ with some non-zero constant $C_0$, hence has only $C^R$ regularity, recalling that we assume $R=\frac{\lambda_1}{\la_2}\in(1,+\infty)\setminus\Z$.

Due to the smoothness of the change of variables \eqref{Eq.heuristic_changevariable}, we expect that among all solutions of nonlinear ODE \eqref{Eq.heuristiceq} passing through $(0,0)$, only two of them are smooth and any other solutions have $C^R$ regularity; moreover, we can select the only smooth solution of \eqref{Eq.heuristiceq} passing through $(0,0)$ as long as we specify the slope of this solution. As a result, when we use the power series method to construct the smooth solution $y=\sum_{n=0}^\infty y_nx^n$ of \eqref{Eq.heuristiceq} passing through $(0,0)$, we need to analyze the coefficients $y_n$ at least from $y_0$ to $y_{\lfloor R\rfloor+1}$, in order to distinguish it from  non-smooth solutions, which are only $C^R$ regular.

\smallskip

The above heuristic discussions inspire us to consider the eigenvalues of our $z-u$ ODE \eqref{Eq.ODE_z_u} near $Q_1$. The linearization of \eqref{Eq.ODE_z_u} near $Q_1$ is
$$\frac{\mathrm d(u-\varepsilon)}{\dz}=\frac{\pa_z\Delta_u(Q_1)z+\pa_u\Delta_u(Q_1)(u-\varepsilon)}{\pa_z\Delta_z(Q_1)z+\pa_u\Delta_z(Q_1)(u-\varepsilon)},$$
whose coefficient matrix is
\begin{equation}\label{Eq.linear_P_1}
	\mathcal{A}=\begin{pmatrix} \pa_u\Delta_u(Q_1) & \pa_z\Delta_u(Q_1) \\ \pa_u\Delta_z(Q_1) & \pa_z\Delta_z(Q_1)\end{pmatrix}=\begin{pmatrix}\tilde c_1 & \tilde c_3 \\ \tilde c_2 & \tilde c_4\end{pmatrix}=\begin{pmatrix}2\varepsilon & 2\varepsilon(k-A)\\ -1 &k\varepsilon+A\end{pmatrix},
\end{equation}
with characteristic polynomial $\varphi_{\mathcal{A}}(\lambda)=\la^2-(\tilde c_1+\tilde c_4)\la+\tilde c_1\tilde c_4-\tilde c_2\tilde c_3=\la^2-((k+2)\varepsilon+A)\la+2k\varepsilon(\varepsilon+1)$, which has two positive roots $0<\la_-<\la_+$\index{$\la_-,\la_+$}, since $2k\varepsilon(\varepsilon+1)>0$,
\begin{align*}
	\tilde c_1+\tilde c_4&=(k+2)\varepsilon+A=(k+2)(\ell\g^2-1)+k+2-(k-2\ell)\g=\left[(k+2)\ell\g-k+2\ell\right]\g\\
	&>\left[(k+2)\ell\frac1{\sqrt\ell}-k+2\ell\right]\g=\left[k(\sqrt\ell-1)+2\sqrt\ell(1+\sqrt\ell)\right]\g>0,
\end{align*}
and the discriminant $(\tilde c_1-\tilde c_4)^2+4\tilde c_2\tilde c_3>0$ (which is the same as the discriminant of $p_0$).

By a direct computation, we have
\begin{equation}\label{la}
		\la_\pm=\frac{\tilde c_1+\tilde c_4\pm\sqrt{(\tilde c_1-\tilde c_4)^2+4\tilde c_2\tilde c_3}}{2}.
\end{equation}
Let $R=\frac{\la_+}{\la_-}>1$.\index{$R=\frac{\la_+}{\la_-}$} Then we have
\begin{equation}\label{Eq.(R+1)^2/R_1}
	\frac{(R+1)^2}{R}=\frac{(\la_++\la_-)^2}{\la_+\la_-}=\frac{(\tilde c_1+\tilde c_4)^2}{\tilde c_1\tilde c_4-\tilde c_2\tilde c_3}=\frac{[(k+2)\ell\g-(k-2\ell)]^2}{2k\ell(\ell\g^2-1)}.
\end{equation}

The subsequent analysis is based on the assumption $R\in(3,4)$. However, it turns out that in some circumstances, $R$ may not belong to $(3,4)$. As a consequence, we need to explore the conditions of  parameters ensuring $R\in(3,4)$.

\begin{lemma}\label{Lem.ell_0(k)}
	For each $\ell>1, k\in\Z_{+}, \g>\frac1{\sqrt\ell}$, we define the ratio
	\begin{equation}
		R=R(k, \ell,  \g)=\frac{\la_+}{\la_-}>1.\label{def:R}
	\end{equation}
	Let $\mathcal K$ be the admissible set of $(k, \ell)$ defined by
	\begin{equation}\label{Eq.admissible_set}
		\mathcal{K}=\left\{(k, \ell)\in\Z_{+}\times[1,+\infty): k\leq 5 \text{ or } 1<\ell<\ell_0(k) \text{ for } k\geq6 \right\},\index{$\mathcal K$}
	\end{equation}
	where
	\begin{equation}\label{Eq.ell_0(k)}
		\ell_0(k):=\frac{3k^2-8k+12-\sqrt{(3k-2)(k-6)(3k^2+4k+12)}}{24}\qquad \text{for}\quad k\geq6.\index{$\ell_0(k)$}
	\end{equation}
For each $\ell>1, k\in\Z_{+}$, there exists $R_{\text{inf}}=R_{\text{inf}}(k, \ell)\geq1$\index{$R_{\text{inf}}(k, \ell)$} such that the range of the function $\g\mapsto R(k, \ell, \g)$ is the interval $(R_{\text{inf}}, +\infty)$ or $[R_{\text{inf}}, +\infty)$; furthermore, $R_{\text{inf}}(k, \ell)<3$ if and only if $(k, \ell)\in\mathcal K$.
\end{lemma}
\begin{remark}
	The symbol $R_{\text{inf}}(k, \ell)$ stands for the infimum value of the function $\g\mapsto R(k, \ell, \g)$. We have $1<\ell_0(k)<k/2$ for $k>6$, $\lim_{k\to\infty}\ell_0(k)=1$ and $\ell_0(7)\approx 1.81$, $\ell_0(6)=3$.
\end{remark}
\begin{proof}
	For fixed $k,\ell$, we denote the right hand side of \eqref{Eq.(R+1)^2/R_1} by $F(\g)$, and then
	\begin{align*}
		F'(\g)=-\frac{[(k+2)-(k-2\ell)\g][(k+2)\ell\g-(k-2\ell)]}{k(\ell\g^2-1)^2}.
	\end{align*}
	As $\g>\frac1{\sqrt\ell}$, we have $ \ell\g>\ell\frac1{\sqrt\ell}>1$ and\begin{align*}
		(k+2)\ell\g-(k-2\ell)>(k+2)-(k-2\ell)=2+2\ell>0.
	\end{align*}
	
	If $k\leq2\ell$, then $(k+2)-(k-2\ell)\g>0 $ and $ F'(\g)<0$ for $\g>\frac1{\sqrt\ell} $, the function $F(\g)$ is strictly decreasing for
	$\g>\frac1{\sqrt\ell} $, with $\lim_{\g\downarrow\frac1{\sqrt\ell}}F(\g)=+\infty$ and  $\lim_{\g\to+\infty}F(\g)=\frac{(k+2)^2}{2k}$.
	Also note that the function $R\mapsto \frac{(R+1)^2}{R}$ is strictly increasing for $R>1$, hence the function $\g\mapsto R(k, \ell, \g)$ is strictly decreasing for $\g>\frac1{\sqrt\ell}$, with $\lim_{\g\downarrow\frac1{\sqrt\ell}}R(k, \ell, \g)=+\infty$ and  $R_{\text{inf}}:=\lim_{\g\to+\infty}R(k, \ell, \g)\geq1$ satisfying
	\begin{equation}\label{Eq.R_inf=k/2}
		\frac{(R_{\text{inf}}+1)^2}{R_{\text{inf}}}=\frac{(k+2)^2}{2k}\geq4\Longrightarrow R_{\text{inf}}=\max\left\{\frac k2, \frac 2k\right\}\geq1.
	\end{equation}
	Therefore, $R_{\text{inf}}<3$ if and only if $1\leq k\leq5$.
	
	If $k>2\ell$, let $\g_1=\frac{k+2}{k-2\ell}$, then $\g_1>1>\frac1{\sqrt\ell}$ and
	\begin{align*}
		&(k+2)-(k-2\ell)\g>0,\quad F'(\g)<0\quad \text{for}\ 1/\sqrt\ell<\g<\g_1,\\
		&(k+2)-(k-2\ell)\g<0,\quad F'(\g)>0\quad \text{for}\ \g>\g_1.
	\end{align*}
	Thus, the function $F(\g)$ is strictly decreasing for
	$\frac1{\sqrt\ell}<\g\leq\g_1 $, strictly increasing for
	$\g\geq\g_1 $ with $\lim_{\g\downarrow\frac1{\sqrt\ell}}F(\g)=+\infty$ and\begin{align*}
		&F(\g_1)=\frac{[(k+2)\ell\g_1-(k-2\ell)]^2}{2k\ell(\ell\g_1^2-1)}=\frac{[(k+2)^2\ell/(k-2\ell)-(k-2\ell)]^2}{2k\ell(\ell(k+2)^2/(k-2\ell)^2-1)}=
		\frac{(k+2)^2\ell-(k-2\ell)^2}{2k\ell}.
	\end{align*}
	Since $R\mapsto \frac{(R+1)^2}{R}$ is strictly increasing for $R>1$,  the function $\g\mapsto R(k, \ell, \g)$ is strictly decreasing for $\frac1{\sqrt\ell}<\g\leq\g_1 $, strictly increasing for
	$\g\geq\g_1 $ with $\lim_{\g\downarrow\frac1{\sqrt\ell}}R(k, \ell, \g)=+\infty$ and it has a minimum value $R_{\text{inf}}>1$ at $\g=\g_1$ such that
	\begin{equation*}
		\frac{(R_{\text{inf}}+1)^2}{R_{\text{inf}}}=F(\g_1)=\frac{(k+2)^2\ell-(k-2\ell)^2}{2k\ell}=\frac{(k^2-4\ell)(\ell-1)}{2k\ell}+4>4.
	\end{equation*}
	Therefore,
	\begin{align*}
		R_{\text{inf}}<3&\Longleftrightarrow \frac{\ell(k+2)^2-(k-2\ell)^2}{2k\ell}<\frac{16}{3}\\
		&\Longleftrightarrow P_k(\ell):=12\ell^2-(3k^2-8k+12)\ell+3k^2>0.
	\end{align*}
	The discriminant of the quadratic polynomial $P_k$ is $$\Delta(k)=(3k^2-8k+12)^2-(12k)^2=(3k-2)(k-6)(3k^2+4k+12).$$ As a consequence, $P_k(\ell)>0$ if $1\leq k\leq5$. For $k\geq7$, we have $\Delta(k)>0$ and thus $P_k(\ell)$ has two real roots; also note that $P_k(1)=8k>0$, $P_k\left(\frac k2\right)=-\frac k2(3k-2)(k-6)<0$, hence the inequality $P_k(\ell)>0$ holds if and only if $1<\ell<\ell_0(k)$ (recalling that we are dealing with the case $1<\ell<\frac k2$), where $\ell_0(k)$ is given by \eqref{Eq.ell_0(k)} (such that $P_k(\ell_0(k))=0$). At the same time, we obtain $1<\ell_0(k)<\frac k2$ for $k>6$. For $k=6$, we have $P_k(\ell)=12\ell^2-72\ell+108=12(\ell-3)^2>0$ for $1<\ell<\frac k2=\ell_0(k)$.
	
	Combining all things together completes the proof.
\end{proof}

\begin{remark}\label{Rmk.R_parameter}
	For further usage, recall that $A=k+2-(k-2\ell)\g\in \mathbb R$ and we define
	\begin{equation}
		\G=\G(k,\ell):=\left\{\g>\frac1{\sqrt\ell}\Big|A(k,\ell,\g)>0\right\}.\index{$\Gamma(k,\ell)$}
	\end{equation}
From the proof of Lemma \ref{Lem.ell_0(k)}, when $k\in\mathbb Z_{+}$ and $\ell>1$, the function $\g\mapsto R(k, \ell, \g)$ is strictly decreasing for $\g\in\G(k,\ell)$, in which case we can view $(k, \ell, R)$ as free parameters and write $\g\in\G(k,\ell)$ as a function $\g=\g(k, \ell, R)$ that depends uniquely on $(k, \ell, R)$ for $k\in\mathbb Z_{+}$, $\ell>1$ and $R\in(R_{\text{inf}}, +\infty)$. Moreover, the function $(\ell, R)\mapsto \g(k,\ell, R)$ is smooth, hence $(\ell, R)\mapsto A$, $(\ell, R)\mapsto \varepsilon$ are also smooth, recalling \eqref{Eq.epsilon_A_B_gamma}.
\end{remark}

In the rest of the proof, we assume that $(k,\ell)\in\mathcal K$, hence for any $R_v\in(3,4)$, there exists a unique $\g\in\G(k,\ell)$ such that $R(k,\ell,\g)=R_v$.

Let $ \delta=\la_-$\index{$\delta$} then $ \delta>0$, $ \la_+=R\la_-=R\delta$. It follows from \eqref{Eq.a_1_expression} and \eqref{la} that
\begin{equation}\label{Eq.a_1+c_1=Rdelta}
	a_1+\tilde c_1=a_1+2\varepsilon=\la_+=R\la_-=R\delta.
\end{equation}
We infer from $\varphi_{\mathcal{A}}(\lambda)=\la^2-((k+2)\varepsilon+A)\la+2k\varepsilon(\varepsilon+1)=(\lambda-\la_+)(\lambda-\la_-)$
that
\begin{equation}\label{Eq.Rdelta^2}
	R\delta^2=\la_+\la_-=2k\varepsilon(1+\varepsilon),\quad (R+1)\delta=\la_++\la_-=(k+2)\varepsilon+A.
\end{equation}
Hence we get
\begin{equation}\label{Eq.delta}
	\delta=(R+1)\delta-R\delta=(k+2)\varepsilon+A-(a_1+2\varepsilon)=A+k\varepsilon-a_1>0.
\end{equation}

By \eqref{Eq.epsilon_A_B_gamma} and \eqref{Eq.Rdelta^2}, we get $k+2-(k-2\ell)\g=A=(1+R)\delta-(k+2)\varepsilon$ and
$$(k+2)(1+\varepsilon)=(1+R)\delta+(k-2\ell)\g=(1+R)\sqrt{\frac{2k\varepsilon(1+\varepsilon)}{R}}+(k-2\ell)\g.$$
Dividing the above identity by $\sqrt{1+\varepsilon}=\sqrt\ell\g$ gives
\begin{equation}\label{Eq.4.39}
	(k+2)\sqrt{1+\varepsilon}=(1+R)\sqrt{\frac{2k\varepsilon}{R}}+\frac{k-2\ell}{\sqrt\ell}.
\end{equation}

\subsection{Analytic solutions near $Q_1$}\label{Subsec.local_z_u}

Here we construct an analytic local solution to \eqref{Eq.ODE_z_u} passing through $Q_1$ with slope $a_1$ at $Q_1$. We write the Taylor series of a general analytic function $u$ around the point $Q_1$ with $u(0)=\varepsilon$ and $u'(0)=a_1$ as
\begin{equation}\label{Eq.u_z_series}
	u(z)=\sum_{n=0}^\infty u_nz^n=\varepsilon+\sum_{n=1}^\infty u_nz^n,
\end{equation}
where $u_0=\varepsilon$ and $u_1=a_1$ is the larger root of \eqref{Eq.quad_u_1}.

\begin{lemma}\label{Lem.recurrence_relation}
	Let $u(z)$ be defined by \eqref{Eq.u_z_series}, and $\mathcal L(u)(z)$ be defined by \eqref{Luz}. Then we have
	\begin{equation}
		\mathcal L(u)(z)=\sum_{n=0}^\infty \mathcal U_n z^n,
	\end{equation}
where $\mathcal U_0=\mathcal U_1=0$ and for each $n\geq2$,
	\begin{equation}\label{Eq.mathcal_U_n}
		\mathcal U_n=(R-n)\delta u_n-E_n,
	\end{equation}
	with $\delta=\lambda_-$, $R=\lambda_+/\lambda_-$ and $E_n$ given by
	\begin{equation}\label{Eq.E_n}
		\begin{aligned}
			E_n&=-\frac{n+1}2\sum_{j=2}^{n-1}u_{j}u_{n+1-j}-\left(2-\frac{k+1}{2}n\right)\sum_{j=1}^{n-1}u_{j}u_{n-j}\\
			&\qquad -\big((n-3)A+(n-1)B+2k\big)u_{n-1}+\big((n-4)B+2k\big)u_{n-2}.
		\end{aligned}
	\end{equation}
\end{lemma}
\begin{proof}
	In a neighborhood of $Q_1$, we have $\frac{\du}{\dz}=\sum_{n=1}^\infty nu_nz^{n-1}$. For notational simplicity, we set $u_n=0$ for $n<0$ and
	\begin{equation}
		(u^2)_n=\sum_{n_1+n_2=n}u_{n_1}u_{n_2}\quad\forall\ n\geq0.
	\end{equation}
	With this notation, we have
	\begin{equation}\label{Eq.u^2}
		u^2=\sum_{n=0}^\infty (u^2)_nz^n,\qquad uu'=\frac12(u^2)'=\frac12\sum_{n=1}^\infty n(u^2)_nz^{n-1}.
	\end{equation}
	Recall that $\mathcal L(u)(z)=-\Delta_z(z,u)\frac{\du}{\dz}+\Delta_u(z,u)$ with (see \eqref{Luz} and \eqref{uz})
	\begin{equation}\label{Eq.Delta_z_u}
		\begin{aligned}
			\Delta_u(z,u)&=2u\left[u+f(z)+kz(1-z)\right]=2u^2+2(B-k)z^2u+2(k-A)zu-2\varepsilon u,\\
			-\Delta_z(z,u)&=(1-(k+1)z)u+(1-z)f(z)\\
			&=-(k+1)zu+u-Bz^3+(A+B)z^2+(\varepsilon-A)z-\varepsilon.
		\end{aligned}
	\end{equation}
	Substituting \eqref{Eq.u^2} into \eqref{Eq.Delta_z_u}  gives
	\begin{align*}
		\Delta_u(z,u)&=2\sum_{n=0}^\infty (u^2)_nz^n+2(B-k)\sum_{n=0}^\infty u_nz^{n+2}+2(k-A)\sum_{n=0}^\infty u_nz^{n+1}-2\varepsilon\sum_{n=0}^\infty u_nz^{n},\\
		-\Delta_z(z,u)\frac{\du}{\dz}&=-\frac{k+1}{2}\sum_{n=0}^\infty n(u^2)_nz^{n}+\frac12\sum_{n=1}^\infty n(u^2)_nz^{n-1}-B\sum_{n=0}^\infty nu_nz^{n+2}\\
		&\qquad+(A+B)\sum_{n=0}^\infty nu_nz^{n+1}+(\varepsilon-A)\sum_{n=0}^\infty nu_nz^{n}-\varepsilon\sum_{n=1}^\infty nu_nz^{n-1}.
	\end{align*}
	We write $\mathcal L(u)(z)=\sum_{n=0}^\infty \mathcal U_nz^n$, then
	\begin{align*}
		\mathcal U_0=2(u^2)_0-2\varepsilon u_0+\frac12(u^2)_1-\varepsilon u_1=2u_0^2-2\varepsilon u_0+u_0u_1-\varepsilon u_1=0,
	\end{align*}
	recalling  $u_0=\varepsilon$; and
	\begin{align*}
		\mathcal U_1&=2(u^2)_1+2(k-A)u_0-2\varepsilon u_1-\frac{k+1}2(u^2)_1+(u^2)_2+(\varepsilon-A)u_1-2\varepsilon u_2\\
		&=u_1^2-\big((k-2)\varepsilon+A\big)u_1+2(k-A)\varepsilon=0;
	\end{align*}
	and for $n\geq 2$, 
	\begin{equation}\label{Eq.mathcal_U_n1}
		\begin{aligned}
			\mathcal U_n&=2(u^2)_n+2(B-k)u_{n-2}+2(k-A)u_{n-1}-2\varepsilon u_n-\frac{k+1}2n(u^2)_n-B(n-2)u_{n-2}\\&\qquad+\frac12(n+1)(u^2)_{n+1}+(A+B)(n-1)u_{n-1}+(\varepsilon-A)nu_n-\varepsilon(n+1)u_{n+1}\\
			&=\frac{n+1}{2}(u^2)_{n+1}-\varepsilon(n+1)u_{n+1}+\left(2-\frac{k+1}{2}n\right)(u^2)_n+\big((\varepsilon-A)n-2\varepsilon\big)u_n\\
			&\qquad+\big((A+B)(n-1)+2(k-A)\big)u_{n-1}+\big(2(B-k)-B(n-2)\big)u_{n-2}.
		\end{aligned}
	\end{equation}
	Recalling that
	\begin{align*}
		(u^2)_n=2u_0u_n+\sum_{j=1}^{n-1}u_{j}u_{n-j},\quad
		(u^2)_{n+1}=2u_0u_{n+1}+2u_1u_n+\sum_{j=2}^{n-1}u_{j}u_{n+1-j},
	\end{align*}
	we isolate the terms containing $u_{n+1}$ and $u_n$ from \eqref{Eq.mathcal_U_n1}, i.e., $u_{n+1}\big((n+1)u_0-\varepsilon(n+1)\big)=0$ and
	\begin{align*}
		&u_n\left((n+1)u_1+\left(2-\frac{k+1}{2}n\right)\cdot 2u_0+(\varepsilon-A)n-2\varepsilon\right)\\
		=&u_n\big(u_1+2\varepsilon+n(u_1-A-k\varepsilon)\big)=(A+k\varepsilon-u_1)\left(\frac{u_1+2\varepsilon}{A+k\varepsilon-u_1}-n\right)u_n\\
		=&(R-n)\delta u_n,
	\end{align*}
	where we have used \eqref{Eq.delta} and \eqref{Eq.a_1+c_1=Rdelta}. Substituting these into \eqref{Eq.mathcal_U_n1} and rearranging terms by order of $u_n$, we get \eqref{Eq.mathcal_U_n} and thus the proof is complete.
\end{proof}

\begin{lemma}\label{Lem.a_n_boun}
	Let $\{a_n\}_{n=0}^\infty$ be a sequence such that $a_0=\varepsilon$, $a_1$ is given by \eqref{Eq.a_1_expression} and $a_n\ (n\geq2)$\index{$a_n (n\geq 0)$} is given by the recurrence relation
	\begin{equation}\label{Eq.a_n_recurrence}
		\begin{aligned}
			(R-n)\delta a_n&=\wt E_n=-\frac{n+1}2\sum_{j=2}^{n-1}a_{j}a_{n+1-j}-\left(2-\frac{k+1}{2}n\right)\sum_{j=1}^{n-1}a_{j}a_{n-j}\\
			&\qquad -\big((n-3)A+(n-1)B+2k\big)a_{n-1}+\big((n-4)B+2k\big)a_{n-2}.
		\end{aligned}
	\end{equation}
	Assume that $\al\in(1, 2)$, $\mathcal C\subset(1, +\infty)\setminus\Z$ is compact and $R\in\mathcal C$. Then there exists a constant $K=K(\mathcal C)>1$ independent of $R$ such that
	\begin{equation}\label{Eq.a_n_bound}
		|a_n|\leq \mathfrak{C}_{n-1}K^{n-\al}\quad \forall\ n\geq2.
	\end{equation}
	Here $ \mathfrak{C}_{n}$ is the Catalan number.
\end{lemma}
\begin{proof}
	Assume that for all $R\in\mathcal C$ and $n\in\Z_+$ we have $R\leq R_M$ and $|R-n|\geq R_m>0$. It follows from Remark \ref{Rmk.R_parameter}, \eqref{Eq.a_1_expression} and \eqref{Eq.delta} that $R\mapsto a_1$, $R\mapsto \delta$ are smooth functions.
By \eqref{Eq.a_1+c_1=Rdelta}, \eqref{Eq.a_1>4} and $\varepsilon>0$, we have
	\begin{equation}\label{Eq.delta_low_bound}
		\delta=\frac{R\delta}{R}=\frac{a_1+2\varepsilon}{R}>\frac{a_1}{R}>
		\frac{4}{R}\geq\frac 4{R_M},\quad \forall\ R\in\mathcal C.
	\end{equation}
	Letting $n=2$ in \eqref{Eq.a_n_recurrence} gives
	\begin{equation}\label{Eq.a_2}
		(R-2)\delta a_2=(k-1)a_1^2-(-A+B+2k)a_1+2(k-B)a_0,
	\end{equation}
	hence $R\mapsto a_2$ is smooth for $R\in(1, +\infty)\setminus\{2\}$ and $|a_2|\lesssim_{k,\ell}1$ uniformly for $R\in\mathcal C$, thus \eqref{Eq.a_n_bound} holds for $n=2$ if we take $K>1$ large enough. Letting $n=3$ in \eqref{Eq.a_n_recurrence} yields
	\begin{equation}\label{Eq.a_3}
		\begin{aligned}
			(R-3)\delta a_3&=\left[(3k-1)a_1-2a_2-2(B+k)\right]a_2+(2k-B)a_1\\
			&=\left[(3k-1)a_1-2a_2-2(B+k)\right]a_2+(\ell-1)a_1,
		\end{aligned}
	\end{equation}
	hence $R\mapsto a_3$ is smooth for $R\in(1, +\infty)\setminus\{2, 3\}$ and $|a_3|\lesssim_{k,\ell}1$ uniformly for $R\in\mathcal C$, thus \eqref{Eq.a_n_bound} holds for $n=3$ if we take $K>1$ large enough. Letting $n=4$ gives
	\begin{equation}\label{Eq.a_4}
		(R-4)\delta a_4=\left[-5a_2+4ka_1-(2k+A+3B)\right]a_3+2ka_2(a_2+1),
	\end{equation}
	hence $R\mapsto a_4$ is smooth for $R\in(1, +\infty)\setminus\{2, 3, 4\}$ and $|a_4|\lesssim_{k,\ell}1$ uniformly for $R\in\mathcal C$, thus \eqref{Eq.a_n_bound} holds for $n=4$ if we take $K>1$ large enough.
	
	Now we assume for induction that $R\mapsto a_n$ is continuous in $R\in(1, +\infty)\setminus\{2, 3, \cdots, n\}$ and \eqref{Eq.a_n_bound} holds for $2\leq n\leq N-1$, where $N\geq5$. We have
	\begin{align*}
		\left|\sum_{j=2}^{N-1}a_ja_{N+1-j}\right|\leq \sum_{j=2}^{N-1}\mathfrak{C}_{j-1}\mathfrak{C}_{N-j}K^{N+1-2\al}\lesssim \mathfrak{C}_{N-1}K^{N+1-2\al},
	\end{align*}
	where we have used $\mathfrak{C}_n\sim \mathfrak{C}_{n-1}$; and similarly,
	\begin{align*}
		\left|\sum_{j=1}^{N-1}a_ja_{N-j}\right|&=\left|2a_1a_{N-1}+\sum_{j=2}^{N-2}a_ja_{N-j}\right|\lesssim \mathfrak{C}_{N-2}K^{N-1-\al}\\
		&\qquad+ \sum_{j=2}^{N-2}\mathfrak{C}_{j-1}\mathfrak{C}_{N-j-1}K^{N-2\al}\lesssim \mathfrak{C}_{N-1}K^{N-1-\al};
	\end{align*}
	hence,
	\begin{align*}
		\left|\wt E_N\right|&\lesssim_{k,\ell}(N+1)\left(\left|\sum_{j=2}^{N-1}a_ja_{N+1-j}\right|+\left|\sum_{j=1}^{N-1}a_ja_{N-j}\right|+|a_{N-1}|+|a_{N-2}|\right)\\
		&\lesssim (N+1)\big(\mathfrak{C}_{N-1}K^{N+1-2\al}+\mathfrak{C}_{N-1}K^{N-1-\al}\\
		&\qquad+\mathfrak{C}_{N-2}K^{N-1-\al}+\mathfrak{C}_{N-3}K^{N-2-\al}\big)
		\lesssim N \mathfrak{C}_{N-1}K^{N+1-2\al}.
	\end{align*}
	Therefore, by the recurrence relation \eqref{Eq.a_n_recurrence}, we know that $R\mapsto a_N$ is smooth in $R\in(1, +\infty)\setminus\{2, 3, \cdots, N\}$ and
(by \eqref{Eq.delta_low_bound}, $R\leq R_M$ and $|R-N|\geq R_m>0$)
	\[|a_N|=\frac{\left|\wt E_N\right|}{|R-N|\delta}\leq \frac{C_0(k,\ell,\mathcal C)}N\left|\wt E_N\right|\leq C_1(k,\ell,\mathcal C) \mathfrak{C}_{N-1}K^{N+1-2\al}\leq\mathfrak{C}_{N-1}K^{N-\al},\]
	if we choose $K>C_1^{1/(\al-1)}$ and thus we close the induction.
\end{proof}

\begin{remark}
	From the proof of Lemma \ref{Lem.a_n_boun}, we see that for all $n\geq0$, $a_n$ is a continuous function in $R\in(3,4)$, and $a_0, a_1, a_2$ are continuous in the closed interval $R\in[3,4]$. Similarly, by the induction, we can show that for all $n\geq 0$, $a_n$ is continuous with respect to $\ell>1$.
\end{remark}

\begin{proposition}\label{Prop.local_solution_Q_1}
	Assume that $(k,\ell)\in\mathcal K$ and $\mathcal C\subset(3, 4)$ is compact. Let $\{a_n\}_{n=0}^\infty$ be defined as in Lemma \ref{Lem.a_n_boun}. Then there exists $\zeta_0=\zeta_0(k,\ell, \mathcal C)>0$ such that the Taylor series
	\[u_L(z)=u_L(z;R)=\sum_{n=0}^\infty a_nz^n\index{$u_L(z)$}\]
	converges absolutely for $z\in[-\zeta_0, \zeta_0]$ and it gives the unique analytic solution to the $z-u$ ODE \eqref{Eq.ODE_P_1} in $[-\zeta_0, \zeta_0]$ with $u_L(0)=\varepsilon$ and $u_L'(0)=a_1$. Moreover, for fixed $(k,\ell)\in\mathcal K$, the function $u_L(z)=u_L(z; R)$ is continuous in the domain $\mathcal D_{\mathcal C}=\{(z, R): z\in[-\zeta_0, \zeta_0], R\in\mathcal C\}$.
\end{proposition}
\begin{proof}
	Since $\mathfrak{C}_n\leq 4^n$ for all $n\geq 0$, by Lemma \ref{Lem.a_n_boun}, we know that there exists a constant $K=K(k,\ell, \mathcal C)>1$ such that $|a_n|\leq K^n$ for all $n\geq0$ and all $R\in\mathcal C$. Taking $\zeta_0=\zeta_0(\mathcal C)=\frac{1}{2K}>0$, the power series $u_L(z)=\sum_{n=0}^\infty a_nz^n$ is uniformly convergent and analytic in $z\in[-\zeta_0, \zeta_0]$. It follows from Lemma \ref{Lem.recurrence_relation} that the function $u_L$ is the unique analytic solution to \eqref{Eq.ODE_P_1} in $[-\zeta_0, \zeta_0]$ with $u_L(0)=\varepsilon$ and $u_L'(0)=a_1$. Finally, the continuity of $u_L(z; R)$ in $\mathcal D_\mathcal C$ follows from the fact that uniform convergence preserves continuity and the continuity of $a_n$ in $R$.
\end{proof}

\begin{remark}\label{Rmk.u_F_continuous}
	The $Q_0-Q_1$ solution curve $u_F(z; R)$ is continuous in $(z, R)\in(0,\frac{1}{k+1})\times (1,+\infty)$. Indeed, alike Proposition \ref{Prop.local_solution_Q_1},  we can show that $v_F(Z; R)$ is continuous on $(Z, R)\in[0, Z_2)\times (1, +\infty)$ for some $Z_2\in(0, Z_1)$ small (also using Remark \ref{Rmk.R_parameter}); hence $u_F(z; R)$ is continuous on $(z, R)\in(\frac{1}{k+1}-z_2, \frac{1}{k+1})\times(1, +\infty)$ for some $z_2>0$ small; then consider the initial value problem $\du/\dz=\Delta_u/\Delta_z$ with $u(\frac{1}{k+1}-z_2)=u_F(\frac{1}{k+1}-z_2)$, and by using the standard theory about the continuous dependence on parameters and initial data for the solutions to ODEs, we can conclude that $u_F(z; R)$ is continuous for $(z, R)\in(0,\frac{1}{k+1})\times (1, +\infty)$.
\end{remark}

\section{Qualitative properties of local analytic solution}\label{Sec.qualitative}
In Proposition \ref{Prop.local_solution_Q_1}, we show that for $(k,\ell)\in\mathcal K$ and $R\in(3,4)$, the $z-u$ ODE \eqref{Eq.ODE_P_1} has a unique local analytic solution
\[u_L(z)=\sum_{n=0}^\infty a_nz^n\]
with $u_L(0)=\varepsilon$ and $u_L'(0)=a_1$. Sometimes, we denote $u_L(z)$ by $u_L(z; R)$ to emphasize the dependence on $R$. Recall that we also construct the $Q_0-Q_1$ solution curve $u=u_F$. If $u_L$ and $u_F$ are equal in a neighborhood of $z=0$, then we know that $u_F$ is smooth at $Q_1$. However, it turns out that $u_L$ and $u_F$ are not the same for some parameters. Our strategy is to use the intermediate value theorem to find the suitable parameters such that $u_L$ and $u_F$ are equal.

To this end, we establish some qualitative properties of the coefficients $a_2, a_3$ and $a_4$. It turns out that if $(k,\ell)$ belongs to a subset of $\mathcal K$, then $a_3>0$, $a_4<0$ for all $R\in(3,4)$. Now we give a heuristic argument on how to use this to show that $u_L=u_F$. If $R$ is close to $3$, then $a_3\sim\frac1{R-3}$ is very large, and the behavior of $u_L$ will be determined by the term $a_3z^3$, hence $u_L$ will be very large such that $u_L(z; R)>u_F(z; R)$ for $z\in(0, \zeta_1(R)]$; similarly, if $R$ is close to $4$, then $-a_4\sim\frac1{4-R}$ is very large, and the behavior of $u_L$ will be determined by the term $a_4z^4$, hence $u_L$ will be very negative such that $u_L(z; R)<u_F(z; R)$ for $z\in(0, \zeta_2(R)]$. Therefore, we can find $3<R_1<R_2<4$ and $\zeta>0$ such that $u_L(\zeta; R_1)>u_F(\zeta; R_1)$ and $u_L(\zeta; R_2)<u_F(\zeta; R_2)$, hence by the intermediate value theorem, there exists $R_0\in(3,4)$ such that $u_L(\zeta; R_0)=u_F(\zeta; R_0)$. For this $R_0$, by the uniqueness of solutions to the ODE \eqref{Eq.ODE_P_1} with $u(\zeta)=u_L(\zeta)$, we get $u_L=u_F$.

We emphasize that for $k=2$ and $\ell$ large, we may not have $a_4<0$ for $R\in(3,4)$, hence the proof for this case is much involved, see Section \ref{Sec.Proof_k=2_l_large}.

\begin{lemma}\label{Lem.a_2>0}
	Let
\begin{equation}\label{B0}
	B_0=(k-2)a_1+2A-4k.\index{$B_0$}
\end{equation}
	Then $B_0>0$ and
	\begin{equation}\label{Eq.a_2>0}
		f\left(-\frac\varepsilon{a_1}\right)\frac{a_1^2}{\varepsilon^2}=-B_0+1-\ell<0,\quad (R-2)a_2=(\ell-1)R+B_0R-\varepsilon \frac{B_0}{\delta}>0.
	\end{equation}
\end{lemma}
\begin{proof}
	We first prove $B_0>0$. If $k\geq 2$, by \eqref{Eq.epsilon_A_B_gamma} and $a_1>2(1+\ell\g)$ we have
	\begin{align*}
		B_0=a_1(k-2)+2A-4k\geq2(1+\ell\g)(k-2)+2(k+2-(k-2\ell)\g)-4k=2k(\ell-1)\g>0.
	\end{align*}
	If $k=1$, then $B_0=-a_1+2A-4k=-a_1+2(k+2-(k-2\ell)\g)-4k=-a_1+2(1+(2\ell-1)\g)$. 
	Since $a_1$ is the larger root of $p_0$ in \eqref{Eq.quad_u_1}	and $p_0(2(1+\ell\g))<0,$ $2(1+(2\ell-1)\g)>2(1+\ell\g) $, and (take $ s=2\ell-1$, $k=1$ in \eqref{p0g})
	\begin{align*}
		p_0(2(1+(2\ell-1)\g))&=2\big(4\ell-1\big)\big(\ell-1\big)\g^2>0,
	\end{align*}we have $a_1<2(1+(2\ell-1)\g)$ and $B_0=-a_1+2(1+(2\ell-1)\g)>0$.
	Thus, we always have $B_0>0$. 	
	
	Let $\xi_1=-\varepsilon/a_1<0$\index{$\xi_1=-\varepsilon/a_1$}, then $f(\xi_1)=-\varepsilon-A\xi_1+B\xi_1^2=-\varepsilon+A\varepsilon/a_1+B\varepsilon^2/a_1^2$. By \eqref{Eq.quad_u_1},
	\begin{align*}
		f(\xi_1)\frac{a_1^2}{\varepsilon}&=-a_1^2+Aa_1+B\varepsilon=-\big((k-2)\varepsilon+A\big)a_1+2(k-A)\varepsilon+Aa_1+B\varepsilon\\
		&=-(k-2)a_1\varepsilon+(2k-2A+B)\varepsilon=(-B_0+B-2k)\varepsilon.
	\end{align*}
	Now it follows from $B_0>0$, \eqref{Eq.epsilon_A_B_gamma}, and $\ell>1$ that
	\begin{equation}\label{Eq.f(xi_1)<0}
		f\left(-\frac\varepsilon{a_1}\right)\frac{a_1^2}{\varepsilon^2}=-B_0+B-2k=-B_0+1-\ell<0.
	\end{equation}
	
	Let $u_{(1)}(z)=\varepsilon+a_1z$\index{$u_{(1)}(z)$}, then we get by Lemma \ref{Lem.recurrence_relation} and \eqref{Eq.a_n_recurrence} that
	\begin{equation}\label{Lu1}\mathcal L\left(u_{(1)}\right)(z)=-(R-2)\delta a_2z^2+(B-2k)a_1z^3=-(R-2)\delta a_2z^2-(\ell-1)a_1z^3.\end{equation}
	Plugging $u_{(1)}$ into the definition of $\mathcal L$ in \eqref{Luz} gives
	\begin{equation}\label{Eq.L(u_1)(xi_1)}
		\mathcal L\left(u_{(1)}\right)(\xi_1)=(1-\xi_1)f(\xi_1)a_1,
	\end{equation}
	hence $(1-\xi_1)f(\xi_1)a_1=-(R-2)\delta a_2\xi_1^2-(\ell-1)a_1\xi_1^3$, and thus
	\begin{align*}
		(R-2)\delta a_2&=-(\ell-1)a_1\xi_1-(1-\xi_1)f(\xi_1)\frac{a_1}{\xi_1^2}=(\ell-1)\varepsilon-(a_1+\varepsilon)f(\xi_1)\frac{a_1^2}{\varepsilon^2}\\
		&=(\ell-1)\varepsilon-(R\delta-\varepsilon)(-B_0+1-\ell)=(\ell-1)R\delta+B_0R\delta-B_0\varepsilon,
	\end{align*}
	where we have used \eqref{Eq.a_1+c_1=Rdelta} and \eqref{Eq.f(xi_1)<0}. Dividing the both sides of the above equality by $\delta>0$ and using $R\delta-\varepsilon=a_1+\varepsilon>0 $ gives the second item of \eqref{Eq.a_2>0}.
\end{proof}

Next we prove several identities that will be very useful in the sequel. Define
\begin{equation}\label{Eq.B_1_2_def}
	B_1:=(3k-1)a_1-2a_2-2(k+B), \qquad B_2:=-5a_2+4ka_1-(2k+A+3B).\index{$B_1, B_2$}
\end{equation}
Note that by \eqref{Eq.a_3} and \eqref{Eq.a_4}, we have
\begin{equation}\label{Eq.a_3_4}
	(R-3)\delta a_3=B_1a_2+(\ell-1)a_1,\qquad (R-4)\delta a_4=B_2a_3+2ka_2(a_2+1).
\end{equation}

\begin{lemma}
	It  holds that
	\begin{align}
		(R-2)a_2&=\big((k-3)R+1\big)a_1+3AR-(2k+B)R,\label{Eq.a_2_expression}\\
		(R-2)B_1&=\big((k+5)R-6k\big)a_1-6AR+2kR+4k+4B,\label{Eq.B_1_expression}\\
		(R-2)B_2&=-\big((k-15)R+8k+5\big)a_1-(16R-2)A+2(4k+B)R+4k+6B.\label{Eq.B_2_expression}
	\end{align}
\end{lemma}

\begin{proof}
	By \eqref{Eq.epsilon_A_B_gamma}, \eqref{Eq.a_1+c_1=Rdelta} and \eqref{Eq.Rdelta^2}, we have
	\begin{align}\notag
		B_0&=(k-2)a_1+2A-4k=(k-2)(R\delta-2\varepsilon)+2((R+1)\delta-(k+2)\varepsilon)-4k\\
		\label{B01}&=(kR+2)\delta-4k(1+\varepsilon),\\
		\varepsilon B_0&=(kR+2)\varepsilon\delta-4k\varepsilon(1+\varepsilon)=(kR+2)\varepsilon\delta-2R\delta^2,\\
		\varepsilon B_0/\delta&=(kR+2)\varepsilon-2R\delta=(kR-2)\varepsilon-2a_1.
	\end{align}	
	Eliminating $\delta$ in \eqref{Eq.a_1+c_1=Rdelta} and \eqref{Eq.Rdelta^2} gives 
	\begin{equation}\label{Eq.eliminate_delta}
		(R+1)a_1-AR=(kR-2)\varepsilon.
	\end{equation}
	Then we have (as $B_0>0$, $ \varepsilon>0$, $ \delta>0$)\begin{align}\label{B0d}
		0<\varepsilon B_0/\delta&=(kR-2)\varepsilon-2a_1=(R+1)a_1-AR-2a_1=(R-1)a_1-AR.
	\end{align}
	By \eqref{Eq.a_2>0}, \eqref{B0} and \eqref{B0d}, we obtain
	\begin{align*}
		a_2(R-2)&=(\ell-1)R+B_0R-\varepsilon{B_0}/{\delta}\\
		&=(\ell-1)R+((k-2)a_1+2A-4k)R-((R-1)a_1-AR)\\&=\big((k-3)R+1\big)a_1+3AR-(4k+1-\ell)R.
	\end{align*}
	Then \eqref{Eq.a_2_expression} follows from $B=2k+1-\ell$.
		
	By \eqref{Eq.B_1_2_def} and \eqref{Eq.a_2_expression}, we get
	\begin{align*}
		(R-2)B_1&=(R-2)(3k-1)a_1-2\big((k-3)R+1\big)a_1-6AR\\
		&\qquad\qquad+2(2k+B)R-2(k+B)(R-2)\\
		&=\big((k+5)R-6k\big)a_1-6AR+2kR+4k+4B,
	\end{align*}
	which proves \eqref{Eq.B_1_expression}. Finally, by \eqref{Eq.B_1_2_def} and \eqref{Eq.a_2_expression}, we also have
	\begin{align*}
		(R-2)B_2&=4k(R-2)a_1-5\big((k-3)R+1\big)a_1-15AR+5(2k+B)R\\
		&\qquad\qquad-A(R-2)-(2k+3B)(R-2)\\
		&=-\big((k-15)R+8k+5\big)a_1-(16R-2)A+2(4k+B)R+4k+6B.
	\end{align*}
	This proves \eqref{Eq.B_2_expression}.
\end{proof}

\begin{lemma}\label{Lem.a_4<0}
	For each $k\geq 2$, there exists $\ell_1(k)\in(1,2)$\index{$\ell_1(k)$} such that for all $1<\ell<\ell_1(k)$ and all $R\in[3,4]$, we have $B_2>0$. Moreover, when $k=1$, the inequality $B_2>0$ holds for all $\ell>1$ and all $R\in[3,4]$.
\end{lemma}
\begin{proof}
	See Appendix \ref{Appen.B_2>0}.
\end{proof}

Let $\ell_1(k)$ be given by Lemma \ref{Lem.a_4<0}, and for $k\in\mathbb Z_{+}$ we let
\begin{equation}\label{Eq.ell_*}\index{$\ell_*(k)$}
	\ell_*(k)=\begin{cases}
		+\infty, & \text{if }k\in\{1,2\},\\
		\ell_1(k), & \text{if }k\geq 3.
	\end{cases}
\end{equation}
We define the subsets $\mathcal K_1$, $\mathcal K_*$ of $\mathcal K$ by
\begin{align}
	\mathcal K_1&=\left\{(k,\ell)\in \mathcal K: 1<\ell<\ell_1(k)\quad \text{if }k\geq2\quad \text{and }\  \ell>1 \quad \text{if }k=1\right\},\label{Eq.mathcal_K_0}\index{$\mathcal K_1$}\\
	\mathcal K_*&=\left\{(k,\ell)\in \mathcal K: 1<\ell<\ell_*(k),\quad k\in\mathbb{N}_{\geq 1}\right\}.\label{Eq.mathcal_K_1}\index{$\mathcal K_*$}
\end{align}
Then $\mathcal K_1\subset \mathcal K_*\subset \mathcal K$ and $\mathcal K_*\setminus\mathcal K_1=\{(2,\ell): \ell\geq \ell_1(2)\}$.
See Appendix \ref{Appen.B_2>0} for the proof of $\mathcal K_*\subset \mathcal K$.

\begin{lemma}\label{Lem.B_1>_k>2}
	If $k\in [3,+\infty)\cap \Z$ and $\ell>1$, then we have
	\[I:=5\left(B_1-\left(\frac{kR}{2}-1\right)\delta\right)-2B_2>0.\]
\end{lemma}
\begin{proof}
	It follows from the definition of $B_1$ and $B_2$ that $5B_1-2B_2=(7k-5)a_1+2A-4B-6k$. Eliminating $\varepsilon$ in
	\eqref{Eq.a_1+c_1=Rdelta} and \eqref{Eq.Rdelta^2}, we obtain
	\begin{equation}\label{Eq.eliminate_epsilon}
		(k+2)a_1-2A=(kR-2)\delta.
	\end{equation}
	Hence, by Lemma \ref{Lem.a_2>0} and \eqref{Eq.epsilon_A_B_gamma}, we get
	\begin{align*}
		I&=5B_1-2B_2-\frac52(kR-2)\delta=(7k-5)a_1+2A-4B-6k-\frac52\big((k+2)a_1-2A\big)\\
		&=\left(9k/2-10\right)a_1+7A-4B-6k\\
		&=(k-3)a_1+7((k-2)a_1+2A-4k)/2+4(2k-B)\\
		&=(k-3)a_1+7B_0/2+4(\ell-1)>0.
	\end{align*}
	\end{proof}

\begin{lemma}\label{Lem.A_a_1}
	Let
	\begin{equation}\label{Eq.mu}
		\mu:=(k+2)^2-\frac{(k-2\ell)^2}{\ell}\in\R.\index{$\mu$}
	\end{equation}
	Then $(2R+k+4)A=\left(2-\frac kR\right)(R+1)a_1+\mu$.
\end{lemma}
\begin{proof}
	It follows from \eqref{Eq.4.39} that
	\begin{equation}\label{Eq.5}(k+2)\sqrt{1+\varepsilon}-(1+R)\sqrt{\frac{2k\varepsilon}{R}}=\frac{k-2\ell}{\sqrt\ell}.\end{equation}
	Taking square and using \eqref{Eq.Rdelta^2}, \eqref{Eq.eliminate_delta}, \eqref{Eq.eliminate_epsilon} , we obtain
	\begin{align*}
		\frac{(k-2\ell)^2}{\ell}&=(k+2)^2(1+\varepsilon)+(1+R)^2\frac{2k\varepsilon}{R}-2(k+2)(1+R)\sqrt{\frac{2k\varepsilon(1+\varepsilon)}{R}}\\
		&=(k+2)^2+\left[(k+2)^2+\frac{2k(1+R)^2}{R}\right]\varepsilon-2(k+2)(1+R)\delta\\
		&=(k+2)^2+\left[(k+2)^2+\frac{2k(1+R)^2}{R}\right]\frac{(R+1)a_1-RA}{kR-2}\\
		&\qquad\qquad-2(k+2)(1+R)\frac{(k+2)a_1-2A}{kR-2}.
	\end{align*}
	Rearranging terms gives $\left(2-\frac kR\right)(R+1)a_1-(2R+k+4)A=-\mu$.
\end{proof}

\begin{lemma}\label{Lem.B_1>_k=2}
	If $k\in\{1,2\},$ $\ell>1,$ and $R\in[3, 4]$, then $(R-2)B_1>(R-3)(kR-2)\delta$.
\end{lemma}
\begin{proof}
	By \eqref{Eq.B_1_expression} and \eqref{Eq.eliminate_epsilon}, we compute
	\begin{align*}
		&(R-2)B_1-(R-3)(kR-2)\delta\\
		=&\big((k+5)R-6k\big)a_1-6RA+2kR+4k+4B-(R-3)\big((k+2)a_1-2A\big)\\
		=&3(R-k+2)a_1-(4R+6)A+2kR+4k+4B.
	\end{align*}
	Now it follows from Lemma \ref{Lem.A_a_1} that
	\begin{align*}
		\wt J_k:&=(2R+k+4)\left[(R-2)B_1-(R-3)(kR-2)\delta\right]\\
		&=3(2R+k+4)(R-k+2)a_1-(4R+6)\left(2-k/R\right)(R+1)a_1\\
		&\qquad\qquad-(4R+6)\mu+(2R+k+4)(2kR+4k+4B)\\
		&=\wt g_k(R)a_1-(4R+6)\mu+(2R+k+4)(2kR+4k+4B),
	\end{align*}
	where (for $k=1,2$, $R\in[3,4]$)
	\begin{align*}
		\wt g_k(R)&=3(2R+k+4)(R-k+2)-(4R+6)\left(2-k/R\right)(R+1)\\
		&=-2R^2+(k+4)R-3k^2+4k+12+{6k}/{R}\\
		&=(4-R)(2R+4-k)+(2-k)(3k-2)+{6k}/{R}>0,
	\end{align*}
	and
	\begin{align}\label{mu}\mu=(k+2)^2-\frac{(k-2\ell)^2}{\ell}=k^2+8k+4-4\ell-\frac{k^2}{\ell}<k^2+8k+4-4\ell.\end{align}
	Hence,
	\begin{align*}
		\wt J_k&>-(4R+6)\mu+(2R+k+4)(2kR+4k+4B)\\
		&>-(4R+6)(k^2+8k+4-4\ell)+(2R+k+4)(2kR+4k+4(2k+1-\ell))\\
		&=2k(R(2R-k)+3k)+4(\ell-1)(2R+2-k)>0.
	\end{align*}
	\end{proof}
	
\begin{lemma}\label{Lem.a_3>0}
	For $(k,\ell)\in\mathcal K_*$ and $R\in[3,4]$, we have $B_1>0$.
\end{lemma}
\begin{proof}
	If $k\geq 3$, then $\ell<\ell_*(k)=\ell_1(k)$, and by Lemma \ref{Lem.a_4<0} we have $B_2>0$. Hence by Lemma \ref{Lem.B_1>_k>2}, we obtain 
	$B_1>\left(\frac{kR}2-1\right)\delta\geq\left(\frac R2-1\right)\delta\geq\frac12\delta>0.$
	For $k\in\{1,2\}$, we get by Lemma \ref{Lem.B_1>_k=2} that $(R-2)B_1>(R-3)(kR-2)\delta\geq0$, hence $B_1>0$.
\end{proof}

\begin{lemma}\label{Lem.a_3>0a_4<0}
	For $(k,\ell)\in\mathcal K_1$ and $R\in(3,4)$, we have
	\begin{equation}\label{Eq.a_3>0a_4<0}
		a_3>0,\qquad a_4<0,\qquad \frac{a_3}{a_2}>2k-1>\frac{kR}2-1>\frac{a_1}\varepsilon>0.
	\end{equation}
\end{lemma}\begin{proof}
	By \eqref{Eq.a_3_4}, Lemma \ref{Lem.a_4<0}, Lemma \ref{Lem.a_3>0}, $a_2>0$ and $a_1>4>0$, we get
	\begin{align*}
		(R-3)\delta a_3=B_1a_2+(\ell-1)a_1>0&\Longrightarrow a_3>0,\\
		(R-4)\delta a_4=B_2a_3+2ka_2(a_2+1)>0&\Longrightarrow a_4<0.
	\end{align*}
	
	By \eqref{Eq.a_3_4}, Lemma \ref{Lem.B_1>_k>2}, for $k\in\Z$, $k\geq3$ we have
	\[(R-3)\delta a_3>B_1a_2>\left(\frac{kR}2-1\right)\delta a_2,\]
	hence,
	\[\frac{a_3}{a_2}>\frac{{kR}/{2}-1}{R-3}=\frac k2+\frac{3k-2}{2(R-3)}>\frac k2+\frac{3k-2}{2}=2k-1,\qquad R\in(3,4).\]
	For $k=1,2$, we get by Lemma \ref{Lem.B_1>_k=2} that  
	$$(R-3)\delta a_3>B_1a_2>\frac{(R-3)(kR-2)\delta a_2}{R-2},$$
	hence $\frac{a_3}{a_2}>\frac{kR-2}{R-2}=k+\frac{2k-2}{R-2}\geq k+\frac{2k-2}{2}=2k-1$. Hence $\frac{a_3}{a_2}>2k-1$. As $R<4$, we have $2k-1>\frac{kR}2-1$.
	Thus, we have shown that $\frac{a_3}{a_2}>2k-1>\frac{kR}2-1$.
	
	Finally, $\frac{kR}2-1>\frac{a_1}\varepsilon>0$ follows from \eqref{B0d}, \eqref{Eq.a_1>4}  and $ \varepsilon>0$. 
\end{proof}

\begin{remark}\label{rem3}
	For $(k,\ell)\in\mathcal K_*$, we still have $a_3>0$.
\end{remark}

\section{Solution curve passing through the sonic point}\label{Sec.left_P_1}

Fix $(k,\ell)\in\mathcal K_*$. In this section, we prove that there exists $R_0\in(3,4)$ such that the local solution $u_L(\cdot; R_0)$ connects to the $Q_0-Q_1$ curve
$u_F(\cdot; R_0)$. We use the barrier function method. 

\begin{lemma}\label{Lem.Rto3}
	For $(k,\ell)\in \mathcal K_*$, there exists $\theta_1\in(0,1)$ such that for all $R\in(3, 3+\theta_1)$,
	\begin{equation}\label{Eq.u_L>u_F}
		u_L(z; R)>u_F(z; R)\quad \forall \ z\in(0, \zeta_1(R)]\ \ \text{for some }\ \zeta_1(R)>0.
	\end{equation}
\end{lemma}
\begin{proof}
	Let
	\begin{equation*}
		U_1(z)=U_1(z; R)=\varepsilon+a_1z+a_2z^2+a_3z^3+M_1z^4,\quad M_1:=a_4-1.
	\end{equation*}
	Then there exists $\zeta_{1,1}(R)>0$ on $R\in(3, 4)$ such that
	\begin{equation}\label{uL>U1}
		u_L(z; R)>U_1(z; R)\quad\forall\ z\in(0, \zeta_{1,1}(R)].
	\end{equation}
	We define $\zeta_{1,2}(R)=(R-3)^{7/16}$ for $R\in(3,4)$ and we claim that there exists $\theta_1\in(0,1)$ such that for all $R\in(3, 3+\theta_1)$, we have
	\begin{equation}\label{Eq.U_1>u_F}
		U_1(z; R)>u_F(z; R)\quad \forall \ z\in(0,\zeta_{1,2}(R)].
	\end{equation}
	Assuming \eqref{Eq.U_1>u_F}, we take $\zeta_1(R)=\min\{\zeta_{1,1}(R), \zeta_{1,2}(R)\}>0$, then by \eqref{uL>U1} we know that \eqref{Eq.u_L>u_F} holds for $R\in(3, 3+\theta_1)$.
	
	Recalling \eqref{Eq.u_b} and \eqref{Eq.u_p<u_F<u_b}, we first prove \eqref{Eq.U_1>u_F} assuming the following two inequalities
	\begin{equation}\label{Eq.U_1(z)>u_F}
		U_1(\zeta_{1,2}(R))>u_{\text{b}}(\zeta_{1,2}(R))>u_F(\zeta_{1,2}(R)),
	\end{equation}
	\begin{equation}\label{Eq.U_1compare}
		\mathcal L(U_1)(z)>0\qquad \forall \ z\in(0, \zeta_{1,2}(R)].
	\end{equation}
	Indeed, assume on the contrary that \eqref{Eq.U_1>u_F} does not hold. By continuity and \eqref{Eq.U_1(z)>u_F}, there exists $\zeta_*\in(0, \zeta_{1,2}(R))$ such that $U_1(\zeta_*)=u_F(\zeta_*)$ and $U_1(z)>u_F(z)$ for $z\in(\zeta_*, \zeta_{1,2}(R))$, hence $U_1'(\zeta_*)\geq u_F'(\zeta_*)$. Since $\Delta_z(\zeta_*, U_1(\zeta_*))=\Delta_z(\zeta_*, u_F(\zeta_*))>0$ (see \eqref{Eq.u_p<u_F<u_b}), we have
	\begin{align*}
		\mathcal L(U_1)(\zeta_*)&=-\Delta_z(\zeta_*, U_1(\zeta_*))U_1'(\zeta_*)+\Delta_u(\zeta_*, U_1(\zeta_*))\\
		&\leq -\Delta_z(\zeta_*, u_F(\zeta_*))u_F'(\zeta_*)+\Delta_u(\zeta_*, u_F(\zeta_*))=\mathcal L(u_F)(\zeta_*)=0,
	\end{align*}
	which contradicts with \eqref{Eq.U_1compare}. 
	
	So, it suffices to show \eqref{Eq.U_1(z)>u_F} and \eqref{Eq.U_1compare} for all $0<R-3\ll1$. To start with, we recall that from Remark \ref{Rmk.R_parameter} and the proof of Lemma \ref{Lem.a_n_boun}, we have
	\begin{equation}\label{Eq.a_1_2bounded}
		|\varepsilon|+|\delta|+|A|+|B|+|a_1|+|a_2|\lesssim_{k,\ell}1,\quad \delta>4/R\geq1\quad\forall \ R\in[3,4].
	\end{equation}
	By \eqref{Eq.a_3_4}, \eqref{Eq.B_1_2_def}, \eqref{Eq.a_1>4}, \eqref{Eq.a_2>0}, Lemma \ref{Lem.a_3>0} and \eqref{Eq.a_1_2bounded}, we have
	\begin{equation}\label{Eq.asymptotic_R_3}
		a_3\sim (R-3)^{-1}>0,\qquad |M_1|\leq|a_4|+1\lesssim(R-3)^{-1}\quad \text{as}\quad R\downarrow3.
	\end{equation}
	
	Now we prove \eqref{Eq.U_1(z)>u_F}. Thanks to \eqref{Eq.u_p<u_F<u_b}, it suffices to show that $U_1(\zeta_{1,2}(R))>u_{\text{b}}(\zeta_{1,2}(R))$, which by \eqref{Eq.u_b} is equivalent to
	\begin{equation*}
		\big(1-(k+1)z\big)(\varepsilon+a_1z+a_2z^2+a_3z^3+M_1z^4)>\varepsilon+(A-\varepsilon)z-(A+B)z^2+Bz^3
	\end{equation*}
	for $z=\zeta_{1,2}(R)=(R-3)^{7/16}>0$, which is further equivalent to
	\begin{equation}\label{Eq.6.10}
		\begin{aligned}
			&a_1-k\varepsilon-A+\big(a_2-(k+1)a_1+A+B\big)\zeta_{1,2}(R)+\big(a_3-(k+1)a_2-B\big)\zeta_{1,2}(R)^2\\
			&\qquad+\big(M_1-(k+1)a_3\big)\zeta_{1,2}(R)^3-(k+1)M_1\zeta_{1,2}(R)^4>0.
		\end{aligned}
	\end{equation}
	By \eqref{Eq.a_1_2bounded} and \eqref{Eq.asymptotic_R_3}, as $R\downarrow 3$ we have
	\begin{align*}
		&\left|a_1-k\varepsilon-A\right|\lesssim 1,\qquad \left|\big(a_2-(k+1)a_1+A+B\big)\zeta_{1,2}(R)\right|\lesssim \zeta_{1,2}(R)\lesssim 1,\\
		&\big(a_3-(k+1)a_2-B\big)\zeta_{1,2}(R)^2\sim a_3\zeta_{1,2}(R)^2\sim (R-3)^{-1/8}>0,\\
		&\left|\big(M_1-(k+1)a_3\big)\zeta_{1,2}(R)^3-(k+1)M_1\zeta_{1,2}(R)^4\right|\lesssim (R-3)^{-1}\zeta_{1,2}(R)^3\lesssim 1,
	\end{align*}
	hence, in \eqref{Eq.6.10}, as $R\downarrow3$, the main contribution comes from the $\zeta_{1,2}(R)^2$ term, which is positive. This checks the desired \eqref{Eq.U_1(z)>u_F} for $0<R-3\ll1$.
	
	Finally, we prove \eqref{Eq.U_1compare}. Plugging the definition of $U_1$ into Lemma \ref{Lem.recurrence_relation} and using \eqref{Eq.a_n_recurrence}, we compute that $\mathcal L(U_1)(z)=z^4q_1(z)$, where $q_1$ is a polynomial of degree $4$ given by
	\[q_1(z)=\mathcal U_4^{(1)}+\mathcal U_5^{(1)}z+\mathcal U_6^{(1)}z^2+\mathcal U_7^{(1)}z^3+\mathcal U_8^{(1)}z^4,\]
	where
	\begin{align*}
		\mathcal U_4^{(1)}&=-(R-4)\delta,\\
		\mathcal U_5^{(1)}&=\left(6a_2-(5k+1)a_1+2A+4B+2k\right)M_1+3a_3^2-(5k+1)a_2a_3-(B+2k)a_3,\\
		\mathcal U_6^{(1)}&=7a_3M_1-(3k+1)(2a_2M_1+a_3^2)-2(B+k)M_1,\\
		\mathcal U_7^{(1)}&=4M_1^2-(7k+3)a_3M_1,\qquad \mathcal U_8^{(1)}=-2(2k+1)M_1^2.
	\end{align*}
	By \eqref{Eq.a_1_2bounded} and \eqref{Eq.asymptotic_R_3}, as $R\downarrow 3$ we have
	\begin{align*}
		\mathcal U_4^{(1)}>0,\quad \mathcal U_5^{(1)}\sim 3a_3^2\sim (R-3)^{-2}>0,\quad\left|\mathcal U_6^{(1)}\right|+\left|\mathcal U_7^{(1)}\right|+\left|\mathcal U_8^{(1)}\right|\lesssim (R-3)^{-2}.
	\end{align*}
	If $z\in(0,(R-3)^{7/16}]$, then $z\ll1$,
	$\mathcal U_5^{(1)}z\gtrsim (R-3)^{-2}z$ and $\left|{\mathcal U_6^{(1)}}z^2+\mathcal U_7^{(1)}z^3+\mathcal U_8^{(1)}z^4\right|\lesssim (R-3)^{-2}z^2\ll(R-3)^{-2}z$, hence $q_1(z)>\mathcal U_5^{(1)}z+\mathcal U_6^{(1)}z^2+\mathcal U_7^{(1)}z^3+\mathcal U_8^{(1)}z^4\gtrsim (R-3)^{-1}z>0$. This proves $\mathcal L(U_1)(z)=z^4q_1(z)>0$ for $z\in(0, \zeta_{1,2}(R)]$ if $0<R-3\ll1$, i.e., \eqref{Eq.U_1compare}.
\end{proof}

Next we consider the regime where $R\uparrow4$. To start with, we prove the following lemma.
\begin{lemma}\label{Lem.u_2<u_F}
	Assume that $(k,\ell)\in\mathcal K_*$ and $R\in(3,4)$. Let $u_{(2)}(z)=\varepsilon+a_1z+a_2z^2$.\index{$u_{(2)}(z)$} Then
	\begin{equation}\label{Eq.u_2<u_F}
		u_{(2)}(z)<u_F(z)\quad \forall\ z\in\left(0, \frac1{k+1}\right).
	\end{equation}
\end{lemma}
\begin{proof}
	By Lemma \ref{Lem.recurrence_relation}, \eqref{Eq.a_n_recurrence} and Remark \ref{rem3}, we have
	\begin{equation}\label{Eq.u_2compare}
		\mathcal L\left(u_{(2)}\right)(z)=-z^3\big((R-3)\delta a_3+2ka_2(a_2+1)z\big)<0\quad \forall \ z>0.
	\end{equation}
	Since $\lim_{z\uparrow 1/(k+1)}u_F(z)=+\infty$ and $u_{(2)}$ is bounded on any finite interval, there exists $\wt \zeta_*\in\left(0,1/(k+1)\right)$ such that $u_{(2)}(z)<u_F(z)$ for all $z\in\left[\wt\zeta_*, 1/(k+1)\right)$. Assume for contradiction that \eqref{Eq.u_2<u_F} does not hold. Then by the continuity, there exists $\zeta_*\in \left(0,\wt\zeta_*\right)$ such that $u_{(2)}\left(\zeta_*\right)=u_F\left(\zeta_*\right)$ and $u_{(2)}(z)<u_F(z)$ for all $z\in\left(\zeta_*, \wt \zeta_*\right)$, then $u_{(2)}'(\zeta_*)\leq u_F'(\zeta_*)$. Since $\Delta_z\left(\zeta_*, u_{(2)}(\zeta_*)\right)=\Delta_z(\zeta_*, u_F(\zeta_*))>0$ (see \eqref{Eq.u_p<u_F<u_b}), we have
	\begin{align*}
		\mathcal L\left(u_{(2)}\right)(\zeta_*)&=-\Delta_z\left(\zeta_*, u_{(2)}(\zeta_*)\right)u_{(2)}'(\zeta_*)+\Delta_u\left(\zeta_*, u_{(2)}(\zeta_*)\right)\\
		&\geq -\Delta_z(\zeta_*, u_F(\zeta_*))u_F'(\zeta_*)+\Delta_u(\zeta_*, u_F(\zeta_*))=\mathcal L(u_F)(\zeta_*)=0,
	\end{align*}
	which contradicts with \eqref{Eq.u_2compare}. This completes the proof of \eqref{Eq.u_2<u_F}.
\end{proof}

\begin{lemma}\label{Lem.Rto4}
	For $(k,\ell)\in\mathcal K_1$, there exists $\theta_2\in(0, 1)$ such that for all $R\in(4-\theta_2, 4)$,
	\begin{equation}\label{Eq.u_L<u_F}
		u_L(z; R)<u_F(z; R)\quad \forall \ z\in(0,\zeta_2(R)]\ \ \text{for some }\ \zeta_2(R)>0.
	\end{equation}
\end{lemma}
\begin{proof}
	Let
	\begin{equation}\label{Eq.U_2}
		U_2(z)=U_2(z; R)=\varepsilon+a_1z+a_2z^2+a_3z^3+\frac12a_4z^4+M_2z^5,\quad M_2:=(4-R)^{-5/4}.
	\end{equation}
	Since $a_4<0$ by \eqref{Eq.a_3>0a_4<0}, there exists $\zeta_{2,1}(R)>0$ on $R\in(3,4)$ such that
	\begin{equation}\label{Eq.u_L<U_2}
		u_L(z; R)<U_2(z; R)\quad \forall \ z\in(0, \zeta_{2,1}(R)].
	\end{equation}
	We define $\zeta_{2,2}(R)=(4-R)^{3/4}>0$ for $R\in(3,4)$ and we claim that there exists $\theta_2\in(0, 1)$ such that for all $R\in(4-\theta_2, 4)$, we have
	\begin{equation}\label{Eq.U_2<u_F}
		U_2(z; R)< u_F(z; R)\quad \forall\ z\in\left(0, \zeta_{2,2}(R)\right].
	\end{equation}
	Assuming \eqref{Eq.U_2<u_F}, we take $\zeta_2(R)=\min\left\{\zeta_{2,1}(R), \zeta_{2,2}(R)\right\}>0$, then by \eqref{Eq.u_L<U_2} we know that \eqref{Eq.u_L<u_F} holds for $R\in(4-\theta_2, 4)$.
	
	We first prove \eqref{Eq.U_2<u_F} assuming the following two inequalities:
	\begin{equation}\label{Eq.U_2(z)<u_2}
		U_2(\zeta_{2,2}(R))<u_{(2)}(\zeta_{2,2}(R)),
	\end{equation}
	\begin{equation}\label{Eq.U_2compare}
		\mathcal L(U_2)(z)<0\qquad \forall\ z\in(0, \zeta_{2,2}(R)].
	\end{equation}
	Indeed, assume on the contrary that \eqref{Eq.U_2<u_F} does not hold, by \eqref{Eq.U_2(z)<u_2} and \eqref{Eq.u_2<u_F} we have $U_2(\zeta_{2,2}(R))<u_{F}(\zeta_{2,2}(R))$,
	then by the continuity, there exists $\zeta_*\in(0, \zeta_{2,2}(R))$ such that $U_2(\zeta_*)=u_F(\zeta_*)$ and $U_2(z)<u_F(z)$ for $z\in(\zeta_*, \zeta_{2,2}(R))$, hence $U_2'(\zeta_*)\leq u_F'(\zeta_*)$. Since $\Delta_z(\zeta_*, U_2(\zeta_*))=\Delta_z(\zeta_*, u_F(\zeta_*))>0$ (see \eqref{Eq.u_p<u_F<u_b}), we have
	\begin{align*}
		\mathcal L(U_2)(\zeta_*)&=-\Delta_z(\zeta_*, U_2(\zeta_*))U_2'(\zeta_*)+\Delta_u(\zeta_*, U_2(\zeta_*))\\
		&\geq -\Delta_z(\zeta_*, u_F(\zeta_*))u_F'(\zeta_*)+\Delta_u(\zeta_*, u_F(\zeta_*))=\mathcal L(u_F)(\zeta_*)=0,
	\end{align*}
	which contradicts with \eqref{Eq.U_2compare}. This completes the proof of \eqref{Eq.U_2<u_F}.
	
	So, it suffices to show our claim  \eqref{Eq.U_2(z)<u_2} and \eqref{Eq.U_2compare} for all $0<4-R\ll1$. Now \eqref{Eq.a_1_2bounded}
	is still true. By \eqref{Eq.a_3_4}, \eqref{Eq.B_1_2_def}, \eqref{Eq.a_1>4}, Lemma \ref{Lem.a_2>0}, Lemma \ref{Lem.a_3>0}, Lemma \ref{Lem.a_4<0} and \eqref{Eq.a_1_2bounded}, we have
	\begin{equation}\label{eq:R4}
		a_2>R(\ell-1)/(R-2)>0,\quad 0<a_3\sim 1,\quad 0<-a_4\sim (4-R)^{-1}\quad \text{as}\quad R\uparrow4.
	\end{equation}
	Thanks to $u_{(2)}(z)=\varepsilon+a_1z+a_2z^2$ and \eqref{Eq.U_2}, \eqref{Eq.U_2(z)<u_2} is equivalent to $
	a_3z^3+\frac12a_4z^4+M_2z^5<0$ for $z=\zeta_{2,2}(R)=(4-R)^{3/4},$ $ M_2=(4-R)^{-5/4}$, and is further equivalent to
	\begin{align}\label{Eq.6.20}
		a_3+\frac12a_4\zeta_{2,2}(R)+M_2\zeta_{2,2}(R)^2<0\quad \text{for}\quad M_2=(4-R)^{-5/4}.
	\end{align}By \eqref{eq:R4}, as $R\uparrow4$ we have $a_4\zeta_{2,2}(R)\sim -(4-R)^{-1/4}<0$, $|a_3|\lesssim1$ and $|M_2\zeta_{2,2}(R)^2|= (4-R)^{1/4}\leq 1$, hence \eqref{Eq.6.20} holds for $0<4-R\ll1$, and thus \eqref{Eq.U_2(z)<u_2} is checked.

	Finally, we prove \eqref{Eq.U_1compare}. Plugging the definition of $U_2$ into Lemma \ref{Lem.recurrence_relation} and using \eqref{Eq.a_n_recurrence}, we infer that $\mathcal L(U_2)(z)=z^4q_2(z)$, where $q_2$ is a polynomial of degree $6$ given by
	\[q_2(z)=\mathcal U_4^{(2)}+\mathcal U_5^{(2)}z+\mathcal U_6^{(2)}z^2+\mathcal U_7^{(2)}z^3+\mathcal U_8^{(2)}z^4
	+\mathcal U_9^{(2)}z^5+\mathcal U_{10}^{(2)}z^6,\]
	where
	\begin{align*}
		\mathcal U_4^{(2)}=&-(R-4)\delta a_4/2,\\
		\mathcal U_5^{(2)}=&(R-5)\delta M_2+\left(6a_2-(5k+1)a_1+2A+4B+2k\right)a_4/2\\
		&+3a_3^2-(5k+1)a_2a_3-(B+2k)a_3,\\
		\mathcal U_6^{(2)}=&7\left(M_2a_2+a_3a_4/2\right)-(3k+1)(2M_2a_1+a_2a_4+a_3^2)+(3A+5B+2k)M_2-(B+k)a_4,\\
		\mathcal U_7^{(2)}=&8M_2a_3+a_4^2-(7k+3)\left(M_2a_2+a_3a_4/2\right)-(2k+3B)M_2,\\
		\mathcal U_8^{(2)}=&9a_{4}M_2/2-(4k+2)(2M_2a_{3}+a_4^2/4),\\
		\mathcal U_9^{(2)}=&5M_2^2-(9k+5)a_{4}M_2/2,\quad \mathcal U_{10}^{(2)}=-(5k+3)M_2^2.
	\end{align*}By \eqref{Eq.a_1_2bounded} and \eqref{eq:R4}, as $R\uparrow 4$ we have
	\begin{align*}
		&-\mathcal U_4^{(2)}\sim 1>0,\quad |\mathcal U_5^{(2)}-(R-5)\delta M_2|\lesssim|a_4|+1\lesssim(4-R)^{-1},\\
		&-(R-5)\delta M_2\sim M_2=(4-R)^{-5/4},\quad -\mathcal U_5^{(2)}\sim-(R-5)\delta M_2\sim(4-R)^{-5/4}>0,\\
		&\left|\mathcal U_6^{(2)}\right|\lesssim M_2+|a_4|+1\lesssim(4-R)^{-5/4},\quad
		\left|\mathcal U_7^{(2)}\right|\lesssim M_2+|a_4|^2+|a_4|\lesssim(4-R)^{-2},\\
		&\left|\mathcal U_8^{(2)}\right|+\left|\mathcal U_9^{(2)}\right|+\left|\mathcal U_{10}^{(2)}\right|\lesssim M_2^2+|a_4|^2+1\lesssim(4-R)^{-5/2}.
	\end{align*}If $z\in(0, (4-R)^{3/4}]$, then 
	\begin{align*}
		&\left|\mathcal U_6^{(2)}z^2+\mathcal U_7^{(2)}z^3+\mathcal U_8^{(2)}z^4
		+\mathcal U_9^{(2)}z^5+\mathcal U_{10}^{(2)}z^6\right|\\
		&\lesssim (4-R)^{-\frac54}(4-R)^{\frac32}+(4-R)^{-2}(4-R)^{\frac94}+(4-R)^{-\frac52}(4-R)^{3}\lesssim(4-R)^{\frac14}\ll|\mathcal U_4^{(2)}|,
	\end{align*} hence $q_2(z)<\mathcal U_4^{(2)}+\mathcal U_6^{(2)}z^2+\mathcal U_7^{(2)}z^3+\mathcal U_8^{(2)}z^4
	+\mathcal U_9^{(2)}z^5+\mathcal U_{10}^{(2)}z^6<\mathcal U_4^{(2)}/2<0$.  This proves $\mathcal L(U_2)(z)=z^4q_2(z)<0$ for $z\in(0, (4-R)^{3/4}]=(0, \zeta_{2,2}(R)]$, if $0<4-R\ll1$, i.e.,  \eqref{Eq.U_2compare}.
\end{proof}

\begin{proposition}\label{Prop.R_0}
	Let $(k,\ell)\in \mathcal K_*$. Then there exists $R_0\in(3,4)$ such that $u_L(\cdot; R_0)\equiv u_F(\cdot; R_0)$. As a result, for this special $R=R_0$, the $Q_0-Q_1$ solution curve $u_F$ of the $z-u$ ODE can be continued to pass through $Q_1$ smoothly.
\end{proposition}
Since Lemma \ref{Lem.Rto4} holds only for $(k,\ell)\in \mathcal K_1$, here we only prove Proposition \ref{Prop.R_0} for $(k,\ell)\in\mathcal K_1$. The proof of Proposition \ref{Prop.R_0} for $(k,\ell)\in\mathcal K_*\setminus\mathcal K_1$ can be found in Section \ref{Sec.Proof_k=2_l_large}.

\begin{proof}[Proof of Proposition \ref{Prop.R_0} for $(k,\ell)\in\mathcal K_1$]
	Assume that $(k,\ell)\in\mathcal K_1$. Let $\theta_1, \theta_2\in(0,1)$ and $\zeta_1, \zeta_2$ be given by Lemma \ref{Lem.Rto3} and Lemma \ref{Lem.Rto4}. Let $R_1=3+\theta_1/2\in(3,4)$, $R_2=4-\theta_2/2\in(3,4)$ and $\zeta=\min\{\zeta_1(R_1), \zeta_2(R_2),\zeta_0([R_1,R_2])\}$. Lemma \ref{Lem.Rto3} implies that $u_L(\zeta; R_1)>u_F(\zeta; R_1)$ and Lemma \ref{Lem.Rto4} implies that $u_L(\zeta; R_2)<u_F(\zeta; R_2)$.	By Proposition \ref{Prop.local_solution_Q_1} and Remark \ref{Rmk.u_F_continuous}, we know that the function $[R_1,R_2]\ni R\mapsto u_L(\zeta; R)-u_F(\zeta; R)$ is continuous. By the {intermediate} value theorem, there exists $R_0\in[R_1,R_2]\subset (3,4)$ such that $u_L(\zeta; R_0)=u_F(\zeta; R_0)$. Therefore, by the uniqueness of solutions to the $z-u$ ODE \eqref{Eq.ODE_z_u} with $u(\zeta)=u_F(\zeta)$, we know that $u_L(z; R_0)=u_F(z; R_0)$ for $z\in(0, \zeta)$, hence the solution $u_F$ can be continued using $u_L$ to pass through $Q_1$ smoothly.
\end{proof}

\section{Global extension of the solution curve}\label{Sec.right_P_1}

In previous sections, we have seen that for $R=R_0\in(3,4)$ given by Proposition \ref{Prop.R_0}, the $P_0-P_1$ solution $v_F$ of the $Z-v$ ODE \eqref{Eq.mainODE} crosses the sonic point $P_1$ smoothly, and thus it is a smooth solution defined on $Z\in[0, Z_1+\Upsilon_1']$ for some small enough $\Upsilon_1'>0$. In this section, we prove that on the right of $P_1$, this extended solution will leave the region $\mathcal R_2$ by crossing the black curve $\Delta_v=0$ in Figure \ref{Fig.Phase_Portrait_vZ} and then it can be continued to $Z\to+\infty$, where we recall
\begin{align*}
	\mathcal R_2&=\left\{(Z, v): v\in(v_1, 1), Z\in\left(\frac{\sqrt\ell v+1}{v+\sqrt\ell}, \frac{(\g+1)v}{1+\g v^2}\right)\right\}\\
	&=\left\{(Z, v)\in\mathcal R_0:\Delta_v(Z, v)<0, \Delta_{Z}(Z, v)<0, v<Z\right\}.
\end{align*}

\subsection{Barrier functions}\label{Subsec.barrier_functions}

\begin{lemma}\label{Lem.u_3_compare}
	Assume that $(k,\ell)\in\mathcal K_1$ and $R\in(3,4)$. Let $u_{(3)}(z)=\varepsilon+a_1z+a_2z^2+a_3z^3$.\index{$u_{(3)}(z)$} Then we have
	\begin{equation}\label{Eq.u_3_compare}
		\mathcal L\left(u_{(3)}\right)(z)<0\quad \forall \ z\in\left[-\varepsilon/a_1, 0\right).
	\end{equation}
\end{lemma}
\begin{proof}
	Lemma \ref{Lem.a_4<0} implies that $B_2>0$. By Lemma \ref{Lem.recurrence_relation} and \eqref{Eq.a_n_recurrence}, we have
	\[\mathcal L\left(u_{(3)}\right)(z)=-(R-4)\delta a_4z^4+\big(3a_3-(5k+1)a_2-B-2k\big)a_3z^5-(3k+1)a_3^2z^6=z^4q_4(z),\]
	where by \eqref{Eq.a_3_4},
	\begin{align*}
		q_4(z)&=-(R-4)\delta a_4+\big(3a_3-(5k+1)a_2-B-2k\big)a_3z-(3k+1)a_3^2z^2\\
		&=-B_2a_3-2ka_2^2-2ka_2+B_3a_3z-(3k+1)a_3^2z^2,
	\end{align*}
	with $B_3:=3a_3-(5k+1)a_2-B-2k$\index{$B_3$}. It follows from \eqref{Eq.a_3>0a_4<0}, \eqref{Eq.a_2>0} and $B=2k+1-\ell<2k$ that
	\[B_3>3(2k-1)a_2-(5k+1)a_2-2k-2k=(k-4)a_2-4k\geq -3a_2-4k.\]
	Hence, for $z<0$, we get by $a_3>0$ (Lemma \ref{Lem.a_3>0a_4<0})  that
	\begin{align*}
		q_5(z):&=-2ka_2^2-2ka_2+B_3a_3z-(3k+1)a_3^2z^2\\
		&\leq-2a_2^2-2ka_2-(3a_2+4k)a_3z-(3k+1)a_3^2z^2\\
		&=-\frac{(4a_2+3a_3z)^2}8-k\left(2a_2+4a_3z+\frac{23}8a_3z^2\right)-\frac{k-1}8a_3^2z^2\\
		&\leq -k\left(2a_2+4a_3z+\frac{23}8a_3z^2\right)=-2k\left(a_2-\frac{16}{23}\right)-k\frac{(23a_3z+16)^2}{184}.
	\end{align*}
	If $a_2\geq\frac{16}{23}$, then $q_5(z)\leq 0$, hence $q_4(z)=-B_2a_3+q_5(z)<0$ for $z<0$. If $0<a_2<\frac{16}{23}$ and $z\in[-\varepsilon/a_1, 0)$, then $q_5(z)\leq -k(2a_2+4a_3z+23a_3z^2/8)< -k(4a_3z)\leq 4ka_3\varepsilon/a_1$, hence,
	\begin{align*}
		&q_4(z)=-B_2a_3+q_5(z)<-B_2a_3+4ka_3\varepsilon/a_1
		=-a_3(B_2-4k\varepsilon/a_1)<0
	\end{align*}
	due to the fact
	\begin{align*}
		&B_2-4k\varepsilon/a_1=4ka_1-5a_2-(2k+A+3B)-4k\varepsilon/a_1\\
		&=(3k-1)(a_1-4)+(a_1-A)+3(2k-B)+8(k-1)+(4-5a_2)+k(a_1-4-4\varepsilon/a_1)>0,
	\end{align*}
	where we have used $a_1>4$, $2k-B=\ell-1>0$ and
	\begin{equation}\label{Eq.a_1>A}
		a_1>A,\quad a_1-4>4\varepsilon/a_1.
	\end{equation}
	In fact, by \eqref{B0d}, we have $RA<(R-1)a_1<Ra_1$, thus $a_1>A$. By \eqref{Eq.a_1>4}, we have $a_1>2(1+\ell\g)$, hence $a_1(a_1-4)>2(1+\ell\g)\cdot2(\ell\g-1)=4(\ell^2\g^2-1)>4(\ell\g^2-1)=4\varepsilon$, $ a_1-4>4\varepsilon/a_1 $.
\end{proof}

\begin{lemma}\label{Lem.f<0}
Recall that $f(z)=-\varepsilon-Az+Bz^2$. It hods that
	\[f(z)<0\quad \forall\ z\in\left[-\varepsilon/{a_1}, 0\right].\]
\end{lemma}
\begin{proof}
	We have $f(0)=-\varepsilon<0$ and by Lemma \ref{Lem.a_2>0} $f\left(-\varepsilon/{a_1}\right)<0$. If  $B=2k+1-\ell>0$, then $f(z)\leq \max\left\{f\left(-\varepsilon/{a_1}\right),f(0)\right\}<0$ for all $z\in\left[-\varepsilon/{a_1}, 0\right]$.  If $B\leq 0$, as $a_1>A$ we have $f(z)\leq -\varepsilon-Az<-\varepsilon-a_1z\leq0$ for $z\in\left[-\varepsilon/{a_1}, 0\right)$ and $f(0)=-\varepsilon<0$.
	\end{proof}

\begin{lemma}\label{Lem.u_g_compare}
	It holds that
	\begin{equation}\label{Eq.u_g_compare}
		\mathcal L(u_g)(z)>0,\quad {-\Delta_z(z, u_g(z))>0},\quad \forall\ z\in\left[z_g, 0\right).
	\end{equation}
\end{lemma}
\begin{proof}
	This is a brute force computation. We claim that
	\begin{equation}\label{Eq.L(u_g)}
		\mathcal L(u_g)(z)=2k\ell(1-\ell)z(\g+z)^3.
	\end{equation}
	Recall that {(see \eqref{Luz})}
	\[\mathcal L(u_g)(z)=-\Delta_z(z, u_g(z))u_g'(z)+\Delta_u(z, u_g(z)).\]
	By \eqref{Eq.epsilon_A_B_gamma}, {\eqref{Eq.u_g} and \eqref{uz}}, we have
	\begin{align}
		-\Delta_z(z, u_g(z))&=\big(1-(k+1)z\big)u_g(z)+(1-z)f(z)\nonumber\\
		&=\big(1-(k+1)z\big)\big(\varepsilon+2(1+\ell\g)z+(\ell-1)z^2\big)+(1-z)(-\varepsilon-Az+Bz^2)\nonumber\\
		&=\big(2(1+\ell\g)-(k+1)\varepsilon-A+\varepsilon\big)z+\big(\ell-1-2(k+1)(1+\ell\g)+A+B\big)z^2\nonumber\\
		&\qquad\qquad-\big((\ell-1)(k+1)+B\big)z^3\nonumber\\
		&=k\g(1-\ell\g)z-k(-1+\g+2\ell\g)z^2-k(\ell+1)z^3\nonumber\\
		&=-kz(z+\g)\big((\ell+1)z+\ell\g-1\big),\label{Eq.Delta_z(u_g)}\\
		\Delta_u(z, u_g(z))&=2u_g(z)\big(u_g(z)+f(z)+kz(1-z)\big)\nonumber\\
		&=2u_g(z)\big(\varepsilon+2(1+\ell\g)z+(\ell-1)z^2-\varepsilon-Az+Bz^2+kz-kz^2\big)\nonumber\\
		&=2kzu_g(z)(z+\g)=2kz(z+\g)\big(\ell(z+\g)^2-(z-1)^2\big).\nonumber
	\end{align}
	Therefore, we arrive at
	\begin{align*}
		\mathcal L(u_g)(z)&=-2kz(z+\g)\big((\ell+1)z+\ell\g-1\big)\big((\ell-1)z+\ell\g+1\big)\\
		&\qquad\qquad+2kz(z+\g)\big(\ell(z+\g)^2-(z-1)^2\big)\\
		&=-2kz(z+\g)\big(\ell^2(z+\g)^2{-}(z-1)^2\big)+2kz(z+\g)\big(\ell(z+\g)^2-(z-1)^2\big)\\
		&=2k\ell(1-\ell)z(z+\g)^3,
	\end{align*}
	which gives our claim \eqref{Eq.L(u_g)}. Now \eqref{Eq.u_g_compare} follows from \eqref{Eq.L(u_g)}, \eqref{Eq.Delta_z(u_g)}, $z_g+\g=\frac{1+\g}{\sqrt\ell+1}>0$ and
	\begin{align*}
		(\ell+1)z_g+\ell\g-1&=(\ell+1)\frac{1-\sqrt\ell\g}{\sqrt\ell+1}+\ell\g-1=\frac{(\ell-\sqrt\ell)(1+\gamma)}{\sqrt\ell+1}>0.
	\end{align*}This completes the proof.
\end{proof}

\begin{lemma}[Relative positions of barrier functions]\label{Lem.relative_positions}
	It holds that
	\begin{align}
		&u_g(z)>u_{(1)}(z), \quad  \forall\ z<0;\quad -\g<z_g<-\varepsilon/a_1;\label{Eq.u_g_u_1_3}\\
		&u_{(1)}(z)<u_{(3)}(z)\quad \forall\ z\in\left(-\frac{a_2}{a_3}, 0\right); \quad u_{(1)}(z)\geq u_{(3)}(z)\quad \forall\ z\in\left(-\infty, -\frac{a_2}{a_3}\right].\label{Eq.u_1><u_3}
	\end{align}
	{Moreover, if $(k,\ell)\in\mathcal K_1$ and $R\in(3,4)$ then $u_g(z)>u_{(3)}(z)$ for $z<0$}.
\end{lemma}
\begin{proof}
	The inequality \eqref{Eq.u_1><u_3} follows from $u_{(3)}(z)-u_{(1)}(z)=z^2(a_3z+a_2)$.
	
	Now we prove \eqref{Eq.u_g_u_1_3}. Recall that $u_{(1)}(z)=\varepsilon+a_1z$, by \eqref{Eq.a_1>4}, \eqref{Eq.u_g} and $\ell>1$, we have
	\[u_{(1)}(z)-u_g(z)=z\big(a_1-2(1+\ell\g)-(\ell-1)z\big)<0,\qquad\forall\ z<0.\]
	Taking $z=z_g$ we have $0=u_g(z_g)>u_{(1)}(z_g)=\varepsilon+a_1z_g$, thus $-\g<z_g<-\varepsilon/a_1=:\xi_1$.	We also have $u_g(z)>u_{(1)}(z)\geq u_{(3)}(z)$ for $z\leq \xi_2:=-a_2/a_3$.
	Assume for contradiction that there exists $\zeta_*\in(\xi_2, 0)$ such that $u_g(z)>u_{(3)}(z)$ for all $z<\zeta_*$ but $u_g(\zeta_*)=u_{(3)}(\zeta_*)$, then $u_g'(\zeta_*)\leq u_{(3)}'(\zeta_*)$. 
	Since $\zeta_*\in(\xi_2, 0)\subset (\xi_1, 0)\subset(z_g, 0)$ due to Lemma \ref{Lem.a_3>0a_4<0} and \eqref{Eq.u_g_u_1_3}, by \eqref{Eq.u_g_compare} we obtain $-\Delta_z(\zeta_*, u_g(\zeta_*))>0$ and then we get by \eqref{Luz} and Lemma \ref{Lem.u_3_compare} that
	\begin{align*}
		\mathcal L(u_g)(\zeta_*)&=-\Delta_z(\zeta_*, u_g(\zeta_*))u_g'(\zeta_*)+\Delta_u(\zeta_*, u_g(\zeta_*))\\
		&\leq -\Delta_z\left(\zeta_*, u_{(3)}(\zeta_*)\right)u_{(3)}'(\zeta_*)+\Delta_u\left(\zeta_*, u_{(3)}(\zeta_*)\right)=\mathcal L\left(u_{(3)}\right)(\zeta_*)<0,
	\end{align*}
	which contradicts with Lemma \ref{Lem.u_g_compare}.
\end{proof}

\subsection{Global extension of the $Z-v$ solution curve}\label{Subsec.Z-v_right_P_1}
Here we prove that the local solution $v_L=v_L(z)$ near $P_1$ constructed in Section \ref{Sec.local_P_1} leaves the region $\mathcal R_2$ through the black curve $Z=Z_\text{b}$ solving $\Delta_v=0$.

\begin{lemma}\label{Lem.Z_v_alternative}
	Given $(Z_0, v_0)\in\mathcal R_2$ and $Z_0<1$. Let $v(Z)$ be the unique solution of
	\begin{equation}\label{Eq.Z_v_tau}
		\frac{\mathrm{d}v}{\mathrm{d}Z}=\frac{\Delta_v(Z, v)}{\Delta_Z(Z, v)}\quad
		\text{ with } v(Z_0)=v_0.
	\end{equation}
	Then one of the following holds: either
	\begin{enumerate}[(i)]
		\item the solution exists for $Z\in[Z_0,+\infty)$ and $\Delta_Z(Z, v(Z))<0 $, $|v(Z)|<1$, $v(Z)<Z$ for $Z\in[Z_0,+\infty)$ or
		\item there exists $Z_*\in(Z_0,1]$ such that the solution exists for $Z\in[Z_0,Z_*)$ and $(Z, v(Z))\in\mathcal R_2$ for all $Z\in[Z_0,Z_*)$, and
		$Z_*=\frac{\sqrt\ell v_*+1}{v_*+\sqrt\ell}$, here $v_*:=\lim_{Z\uparrow Z_*}v(Z)$.
	\end{enumerate}
\end{lemma}
\begin{proof}
	Let $U=\{(Z,v):\Delta_Z(Z, v)<0\}$. Then $(Z_0, v_0)\in\mathcal R_2\subset U$, $U\subset\R^2$ is an open set and $ \frac{\Delta_v}{\Delta_Z}$ is a smooth function on $U$. Assume that the maximal interval of existence of the solution to \eqref{Eq.Z_v_tau} (such that $(Z,v)\in U$) is $(Z_-,Z_*)$ (then $Z_0\in(Z_-,Z_*)$).  Note that\begin{align*}
		\frac{\mathrm{d}v}{\mathrm{d}Z}=\frac{\Delta_v(Z, v)}{\Delta_Z(Z, v)}=(1-v^2)F_1(Z,v),\quad
		F_1(Z,v):=\frac{m(1-v^2)Z-kv\left(1-vZ\right)}{\Delta_Z(Z, v)}\in C(U),
	\end{align*}
	and $v(Z_0)=v_0\in(-1,1)$. Then we have $v(Z)\in(-1,1)$ for all $Z\in[Z_0,Z_*)$.
	In fact, let
	$$\xi(Z)=F_1(Z,v(Z)),\qquad \xi_0=\frac{1}{2}\ln\frac{1+v_0}{1-v_0},$$
	then  $ \xi\in C([Z_0,Z_*))$, $v(Z_0)=v_0=\tanh\xi_0$ and
	\[v(Z)=\tanh\left(\xi_0+\int_{Z_0}^Z\xi(s)\mathrm{d}s\right)\in(-1,1)\quad\forall\ Z\in[Z_0,Z_*).\]
	
	As $(Z,v(Z))\in U$, we have $ \Delta_Z(Z, v(Z))<0$, $Z\neq Z_g(v(Z))$ for all $Z\in[Z_0,Z_*)$. Due to $(Z_0, v_0)\in\mathcal R_2$, we have $Z_0>Z_g(v_0)=Z_g(v(Z))$, then by the continuity of $v$, we have $$Z>\frac{\sqrt\ell v(Z)+1}{v(Z)+\sqrt\ell}=Z_g(v(Z))>v(Z)\quad  \forall\ Z\in[Z_0,Z_*).$$
	
	If $Z_*=+\infty$, then (i) holds. Now we assume $Z_*<+\infty$. Let $V=Z-\frac{(\g+1)v}{1+\g v^2}=Z-Z_\text b(v)$. Then we have
	\begin{align*}
		&F_1(Z,v)=\frac{m(1+\gamma v^2)V}{\Delta_Z(Z, v)},\quad
		\frac{\mathrm{d}v}{\mathrm{d}Z}=(1-v^2)F_1(Z,v)=VF_2(Z,v)=V\eta_1(Z),\\
		& F_2(Z,v):=\frac{m(1+\gamma v^2)(1-v^2)}{\Delta_Z(Z, v)}\in C(U),\quad \eta_1(Z):=F_2(Z,v(Z))\in C([Z_0,Z_*)),\\
		&\frac{\mathrm{d}V}{\mathrm{d}Z}=1-\frac{(\g+1)(1-\g v^2)}{(1+\g v^2)^2}\frac{\mathrm{d}v}{\mathrm{d}Z}=1-V\eta_2(Z),\quad \eta_2(Z):=
		\frac{(\g+1)(1-\g v(Z)^2)}{(1+\g v(Z)^2)^2}\eta_1(Z).
	\end{align*}
	We also have $F_2(Z,v)<0 $ for $(Z,v)\in U$, $|v|<1$; $\eta_1(Z)<0 $ for all $Z\in[Z_0,Z_*)$; and $ \eta_2\in C([Z_0,Z_*))$.
	Let $$ \eta_3(Z):=\int_{Z_0}^Z\eta_2(s)\mathrm{d}s,\qquad \widetilde{V}(Z)=V(Z)\mathrm{e}^{\eta_3(Z)},\qquad \forall\ Z\in[Z_0, Z_*),$$
	then $\frac{\mathrm{d}\widetilde{V}}{\mathrm{d}Z}=\mathrm{e}^{\eta_3(Z)}>0 $, hence $\widetilde{V}$ is strictly increasing on $[Z_0,Z_*)$, then one of the following holds:
	\[\text{(a)}\quad \widetilde{V}(Z)<0\text{ for all }Z\in[Z_0,Z_*);\qquad\text{(b)}\quad \widetilde{V}(Z_2)\geq0 \text{ for some }Z_2\in[Z_0,Z_*).\]
	
	If (b) holds, then $\widetilde{V}(Z)>0$, ${V}(Z)>0$, $ \frac{\mathrm{d}v}{\mathrm{d}Z}=V\eta_1(Z)<0$  for all $Z\in(Z_2,Z_*)$, and then $v$ is strictly decreasing; also we have $v>-1$ in $[Z_2,Z_*)$, thus the limit $v_*:=\lim_{Z\uparrow Z_*}v(Z)$ exists, and $-1\leq v_*\leq v(Z_2)<1$. By Remark \ref{rem1}, we have $ (Z_*,v_*)\not\in U$.
	Since $ \Delta_Z(Z, v(Z))<0$ for all $Z\in(Z_1,Z_*)$, we have $ \Delta_Z(Z_*, v_*)\leq0$. As $ (Z_*,v_*)\not\in U$, we have $ \Delta_Z(Z_*, v_*)=0$  and then $$Z_*=0\quad  \text{or}\quad Z_*=\frac{\sqrt\ell v_*+1}{v_*+\sqrt\ell}\quad \text{or}\quad Z_*=\frac{\sqrt\ell v_*-1}{\sqrt\ell-v_*}.$$ But in fact, $$Z_*>Z_0>0,\quad Z_*>Z_2>\frac{\sqrt\ell v(Z_2)+1}{v(Z_2)+\sqrt\ell}\geq\frac{\sqrt\ell v_*+1}{v_*+\sqrt\ell}\geq\frac{\sqrt\ell v_*-1}{\sqrt\ell-v_*},$$
	which reaches a contradiction.
	
	So, (a) have to hold. Hence, $\widetilde{V}(Z)<0$, ${V}(Z)<0$ and $ \frac{\mathrm{d}v}{\mathrm{d}Z}=V\eta_1(Z)>0$  for all $Z\in[Z_0,Z_*)$, then
	$v$ is strictly increasing and $v<1$ in $[Z_0,Z_*)$, thus the limit $v_*:=\lim_{Z\uparrow Z_*}v(Z)$ exists and $0<v_0=v(Z_0)\leq v_*\leq1$.
	By Remark \ref{rem1}, we have $ (Z_*,v_*)\not\in U$. Thanks to $$Z>\frac{\sqrt\ell v(Z)+1}{v(Z)+\sqrt\ell}, \quad {V}(Z)=Z-\frac{(\g+1)v(Z)}{1+\g v(Z)^2}<0$$
	and $v_1<v_0=v(Z_0)\leq v(Z)<1$ for all $Z\in[Z_0,Z_*)$, we have $(Z, v(Z))\in\mathcal R_2$ for all $Z\in[Z_0,Z_*)$,
	and $Z_*\geq\frac{\sqrt\ell v_*+1}{v_*+\sqrt\ell} $.
	As $ \Delta_Z(Z, v(Z))<0$ for all $Z\in[Z_0,Z_*)$, we have $ \Delta_Z(Z_*, v_*)\leq0$. As $ (Z_*,v_*)\not\in U$, we have $ \Delta_Z(Z_*, v_*)=0$. Thus, $$Z_*=0\quad  \text{or}\quad Z_*=\frac{\sqrt\ell v_*+1}{v_*+\sqrt\ell}\quad \text{or}\quad Z_*=\frac{\sqrt\ell v_*-1}{\sqrt\ell-v_*}.$$ But as $Z_*>Z_0>0$, $Z_*\geq\frac{\sqrt\ell v_*+1}{v_*+\sqrt\ell}$ and
	\[\frac{\sqrt\ell v+1}{v+\sqrt\ell}\geq\frac{\sqrt\ell v-1}{\sqrt\ell-v}\qquad\forall\ v\in(0,1],\]
	we must have $Z_*=\frac{\sqrt\ell v_*+1}{v_*+\sqrt\ell}=Z_g(v_*)$. As $0<v_*\leq 1$, we have $Z_*\leq1$.
	Therefore,  (ii) holds. 
	
	This completes the proof.
\end{proof}

Now we define a domain $\mathcal D_2'$ by
\begin{equation}
	\mathcal D_2'=\{(z, u): 0<u<u_{(3)}(z),\ z_g<z<0\}.\index{$\mathcal D_2'$}
\end{equation}
By Lemma \ref{Lem.relative_positions}, we have $\mathcal D_2'\subset\mathcal D_2 $ (for $(k,\ell)\in\mathcal K_1$ and $R\in(3,4)$). Let $\mathcal R_2'=\Psi^{-1}(\mathcal D_2')$\index{$\mathcal R_2'$} then $\Psi: \mathcal R_2'\to\mathcal D_2'$ is a bijection and $\mathcal R_2'\subset\mathcal R_2 $. Let $F(z,u)=u-u_{(3)}(z)$. Then
\begin{align}
	&\label{Fzu}\Delta_u\frac{\partial F}{\pa u}+\Delta_z\frac{\partial F}{\pa z}=\Delta_u(z,u)-\Delta_z(z,u)u_{(3)}'(z)\\
	\notag&=\mathcal L(u_{(3)})(z)+\Delta_u(z,u)-\Delta_u(z,u_{(3)}(z))-\big(\Delta_z(z,u)-\Delta_z(z,u_{(3)}(z))\big)u_{(3)}'(z)\\ \notag&=\mathcal L(u_{(3)})(z)+(u-u_{(3)}(z))G(z,u)=\mathcal L(u_{(3)})(z)+F(z,u)G(z,u),\\
	&\label{Gzu}G(z,u):=2\big(u+u_{(3)}(z)+f(z)+kz(1-z)\big)+(1-(k+1)z)u_{(3)}'(z).
\end{align}

\begin{lemma}\label{Lem.Z_v_global}
	Assume that $(k,\ell)\in\mathcal K_1$ and $R\in(3,4)$. Given $(Z_0, v_0)\in\mathcal R_2'$ and $Z_0<1$. Let $v(Z)$ be the unique solution of \eqref{Eq.Z_v_tau}. Then the solution exists for $Z\in[Z_0,+\infty)$ and $\Delta_Z(Z, v(Z))<0,$ $|v(Z)|<1$, $v(Z)<Z$  for $Z\in[Z_0,+\infty)$.
\end{lemma}

\begin{proof}
	As $\mathcal R_2'\subset\mathcal R_2 $, by Lemma \ref{Lem.Z_v_alternative}, we only need to exclude the case (ii). If (ii) holds, then $(Z, v(Z))\in\mathcal R_2$ for all $Z\in[Z_0,Z_*)$. Let $(z(Z),u(Z)):=\Psi(Z, v(Z))$ then $(z(Z),u(Z))\in\mathcal D_2$, $u(Z)>0$, $z(Z)<0$ for all $Z\in[Z_0,Z_*)$. By \eqref{Eq.du}, \eqref{Eq.dz}, \eqref{Eq1} and \eqref{Eq2}, we have
	\begin{align*}
		\frac{\mathrm{d}z}{\mathrm{d}Z}=\frac{-\Delta_z(z,u)}{\Delta_Z(Z,v)\mathcal{N}(z,u)},
		\quad\frac{\mathrm{d}u}{\mathrm{d}Z}=\frac{-\Delta_u(z,u)}{\Delta_Z(Z,v)\mathcal{N}(z,u)}.
	\end{align*}
	Let $\eta_1(Z):=-[\Delta_Z(Z,v(Z))\mathcal{N}(z(Z),u(Z))]^{-1} $. As $\Delta_Z<0 $ in $ \mathcal{R}_2$ and $\mathcal{N}>0 $ in $ \mathcal{D}_0$, we have \begin{align}\label{dz1}
		\frac{\mathrm{d}z}{\mathrm{d}Z}=\eta_1(Z)\Delta_z(z,u),
		\quad\frac{\mathrm{d}u}{\mathrm{d}Z}=\eta_1(Z)\Delta_u(z,u),\quad \eta_1(Z)>0\quad \text{for}\ Z\in[Z_0,Z_*).
	\end{align}
	 Let $Y(Z):=F(z(Z),u(Z))$ (for $F(z,u)=u-u_{(3)}(z)$), which satisfies
	 \begin{align}
	 	\notag\frac{\mathrm{d}Y}{\mathrm{d}Z}&=\frac{\mathrm{d}u}{\mathrm{d}Z}\frac{\pa F}{\pa u}+\frac{\mathrm{d}z}{\mathrm{d}Z}\frac{\pa F}{\pa z}=
	 	\eta_1(Z)\Delta_u(z,u)\partial_uF(z,u)+\eta_1(Z)\Delta_z(z,u)\partial_zF(z,u)\\
	 	\label{dYdZ}&=\eta_1(Z)\left(\mathcal L(u_{(3)})(z)+F(z,u)G(z,u)\right)=\eta_1(Z)\big(\eta_2(Z)+Y(Z)\eta_3(Z)\big),
	 \end{align}
	 where $ \eta_2(Z):=\mathcal L(u_{(3)})(z(Z))$, $\eta_3(Z):=G(z(Z),u(Z))$, {and $G$ is defined in \eqref{Gzu}}.   As\\ $(Z_0, v(Z_0))=(Z_0, v_0)\in\mathcal R_2'$, we have
	 $(z(Z_0),u(Z_0))=\Psi(Z_0, v(Z_0))\in\mathcal D_2'$,
	 then\\ $u(Z_0)<u_{(3)}(z(Z_0)) $, $Y(Z_0)=F(z(Z_0),u(Z_0))=u(Z_0)-u_{(3)}(z(Z_0))<0 $. \smallskip
	 
	 \underline{Claim 1.}
	 If $Y(Z)\leq0$, $Z\in[Z_0, Z^*)$, then $z(Z)>-\varepsilon/a_1$.\smallskip
	 
	 Otherwise, we have $z(Z)\leq-\varepsilon/a_1<-a_2/a_3$ (see \eqref{Eq.a_3>0a_4<0}), by \eqref{Eq.u_1><u_3} we have $u_{(3)}(z(Z))\leq u_{(1)}(z(Z))\le u_{(1)}(-\varepsilon/a_1)=0$,
	 but $0\geq Y(Z)=F(z(Z),u(Z))=u(Z)-u_{(3)}(z(Z))>-u_{(3)}(z(Z))$, which is a contradiction.\smallskip

	 \underline{Claim 2.} $Y(Z)<0$ for all $Z\in[Z_0, Z_*)$.\smallskip
	 
	 Otherwise, let $Z^*:=\inf\{Z\in[Z_0, Z_*): Y(Z)\geq0\}$, then $Z^*\in(Z_0, Z_*)$ and $Y(Z)<0$ for all $Z\in[Z_0, Z^*)$, $Y(Z^*)=0$, hence $Y'(Z^*)\geq 0$.
	 By Claim 1, we have $z^*:=z(Z^*)\in[-\varepsilon/a_1,0)$. On the other hand, by $Y(Z^*)=0$, $Y'(Z^*)\geq 0$, $\eta_1(Z^*)>0$ {and \eqref{dYdZ}} we have $\eta_2(Z^*)\geq0$, hence $\mathcal L(u_{(3)})(z^*)\geq0$, which contradicts with \eqref{Eq.u_3_compare}.

	 We also have $Z_*=\frac{\sqrt\ell v_*+1}{v_*+\sqrt\ell}$ with $v_*:=\lim_{Z\uparrow Z_*}v(Z)$ and $Z_*\in(Z_0,1]$. If $Z_*<1$, then $v_1<v_0\leq v_*<1$, hence $(Z_*,v_*)\in \mathcal R_2^+\subset\mathcal R_0$, and thus $(z_*,u_*):=\Psi(Z_*,v_*)\in \mathcal D_2^+$. By the continuity of $ \Psi$, we have
	 $$(z(Z),u(Z))=\Psi(Z, v(Z))\to\Psi(Z_*,v_*)=(z_*,u_*)\qquad \text{as }Z\uparrow Z_{*}.$$
	 By the continuity of $F$, we have $0>Y(Z)=F(z(Z),u(Z))\to F(z_*,u_*)$ as  $Z\uparrow Z_{*}$, thus $F(z_*,u_*)\leq0$. On the other hand, as $(z_*,u_*)\in \mathcal D_2^+$ we have $u_*=u_g(z_*)$, $z_g<z_*<0$. By Lemma \ref{Lem.relative_positions}, we have $u_g(z_*)>u_{(3)}(z_*) .$ Then $F(z_*,u_*)=u_*-u_{(3)}(z_*)=u_g(z_*)-u_{(3)}(z_*)>0$, which is a contradiction.
	 
	 So, we must have $ Z_*=1$, then $v_*=1$. 
	 By Claim 1 and Claim 2, we have  $z(Z)>-\varepsilon/a_1$ for all $Z\in[Z_0, Z_*)$. By Lemma \ref{Lem.f<0}, we have $ {L_1}:=-\sup_{z\in[-\varepsilon/a_1,0]}f(z)>0$, then $f(z(Z))\leq -{L_1} $ for  all $Z\in[Z_0,Z_*)$.
	 
	 As $(Z, v(Z))\in\mathcal R_2$, we have $0<v(Z)<Z$, $0<1-v(Z)Z<1-v(Z)^2$ for  all $Z\in[Z_0,Z_*)$. Due to $(z(Z),u(Z))=\Psi(Z, v(Z))$, we get $$u(Z)=\frac{(1+\g)^2(1-v(Z)Z)^2}{(1-v(Z)^2)Z^2}  \quad \text{and}\quad  0<u(Z)<\frac{(1+\g)^2(1-v(Z)Z)}{Z^2},\quad  \forall\ Z\in[Z_0,Z_*).$$
	 As $ Z_*=1$, $v_*=1$ and $v_*:=\lim_{Z\uparrow Z_{*}}v(Z)$, we have
	 $$\lim_{Z\uparrow1}(1-v(Z)Z)=1-v_*Z_*=0\quad\Longrightarrow\quad\lim_{Z\uparrow1}u(Z)=0.$$
	 We define $$ {L_2}:=\sup_{z\in\left[-\varepsilon/a_1,0\right]}|f(z)+kz(1-z)|\geq0,\qquad c_0:=\frac{L_1}{2L_2+1}>0,$$ and $h(Z):=z(Z)+c_0\ln u(Z)$. Then by \eqref{dz1} and \eqref{uz},
	 we get
	 \begin{align*}
	 	&\frac{\mathrm{d}h}{\mathrm{d}Z}=\frac{\mathrm{d}z}{\mathrm{d}Z}+\frac{c_0}{u}\frac{\mathrm{d}u}{\mathrm{d}Z}
	 	=\eta_1(Z)\Delta_z(z,u)+\frac{c_0}{u}\eta_1(Z)\Delta_u(z,u)\\
	 	&=\eta_1(Z)\Big(-\big(1-(k+1)z\big)u-(1-z)f(z)+2c_0\big(u+f(z)+kz(1-z)\big)\Big).
	 \end{align*}
	 Since $\eta_1(Z)>0$, $-\varepsilon/a_1<z(Z)<0 $, $u(Z)>0$ for all $Z\in[Z_0,Z_*)=[Z_0,1)$, we have
	 \begin{align*}
	 	&\frac{\mathrm{d}h}{\mathrm{d}Z}>\eta_1(Z)\left(-\left(1+(k+1)\frac\varepsilon{a_1}\right)u+{L_1}-2c_0{L_2}\right)=
	 	\eta_1(Z)\left(c_0-\left(1+(k+1)\frac\varepsilon{a_1}\right)u\right).
	 \end{align*}
	 Thanks to $c_0>0$  and $\lim_{Z\uparrow 1}u(Z)=0$, there exists $Z_2\in (Z_0,1)$ such that
	 $$c_0-\left(1+(k+1)\frac\varepsilon{a_1}\right)u(Z)>0\quad \forall\ Z\in(Z_2, 1),$$
	 then $\frac{\mathrm{d}h}{\mathrm{d}Z} >0$ and $h(Z)>h(Z_2)$ for $Z\in(Z_2,1)$. However, we have
	 $h(Z)=z(Z)+c_0\ln u(Z)<c_0\ln u(Z)\to-\infty$ as $Z\uparrow1,$
	 which reaches a contradiction. So, (ii) can't be true, and (i) holds.
\end{proof}

\section{Proof of  main results}\label{Sec.proof}

This section is devoted to proving  our main theorems.

\begin{proof}[Proof of Theorem \ref{Thm.mainthm_ODE} for $(k,\ell)\in\mathcal K_1$]
	Let $(k, \ell)\in \mathcal K_1$ and let $R=R_0\in(3, 4), \zeta>0$ be given by Proposition \ref{Prop.R_0}. It follows from $a_4<0$ (by Lemma \ref{Lem.a_3>0a_4<0}) that there exists $\zeta_0'\in(0, \zeta)$ such that for all $z\in(-\zeta_0', 0)$, we have
	\[u_L(z)=\varepsilon+a_1z+a_2z^2+a_3z^3+a_4z^4+\cdots<\varepsilon+a_1z+a_2z^2+a_3z^3=u_{(3)}(z),\]
	hence,
	\[(z, u_L(z))\in \mathcal D_2'\quad\forall\ z\in(-\zeta_0', 0).\]
By Proposition \ref{Prop.local_solution_Q_1}, we have $u_L(0)=\varepsilon$ and $u_L'(0)=a_1$. Let $(Z(z),v(z))=\Theta(z,u_L(z))$. Then $Z(z)=\Theta_Z\big(z, u_L(z)\big)$ is a smooth function, and $(Z(0),v(0))=\Theta(0,\varepsilon)=\Theta(Q_1)=P_1=(Z_1,v_1)$. By Remark \ref{rem2},  $ \Delta_Z(Z,v)\mathrm{d}v=\Delta_v(Z,v)\mathrm{d}Z$.
	A direct calculation gives
	\begin{align*}
		\frac{\mathrm{d}Z}{\mathrm{d}z}(0)&=\frac{\pa\Theta_Z}{\pa z}(Q_1)+\frac{\pa\Theta_Z}{\pa u}(Q_1)u_L'(0)=\frac{(1+\g)(\ell\g^2-1)}{\sqrt\ell(\ell\g+1)^2\g^2}-\frac{(1+\g)(\ell\g-1)}{2\sqrt\ell(\ell\g+1)^2\g^2}a_1\\
		&=\frac{(1+\g)}{2\sqrt\ell(\ell\g+1)^2\g^2}\left[2(\ell\g^2-1)-(\ell\g-1)a_1\right]<0,
	\end{align*}
	where we have used $2(\ell\g^2-1)-(\ell\g-1)a_1<2(\ell\g^2-1)-2(\ell\g-1)(\ell\g+1)=2\ell\g^2(1-\ell)<0$ due to \eqref{Eq.a_1>4}. Therefore, $Z'(z)<0$ and $Z(z)$ is strictly decreasing for $z\in(-\zeta_0^*, \zeta_0^*)$, where $\zeta_0^*\in(0, \zeta_0')$ is small.
	Then $Z(z)$ is a bijection from $z\in(-\zeta_0^*, \zeta_0^*)$ onto $Z\in(Z(-\zeta_0^*), Z(\zeta_0^*))=:I$. As $Z_1=Z(0)\in I$, there exists $\Upsilon_1>0$ such that $(Z_1-\Upsilon_1,Z_1+\Upsilon_1)\subset I$ and $\zeta_0''\in(0, \zeta_0^*)$ such that $Z((-\zeta_0'', \zeta_0''))=(Z(\zeta_0''), Z(-\zeta_0''))\subset(Z_1-\Upsilon_1,Z_1+\Upsilon_1) $.
	
	We denote the inverse function of $Z(z):(-\zeta_0^*, \zeta_0^*)\to I$ by $z(Z):I\to (-\zeta_0^*, \zeta_0^*)$ and define
	\begin{equation}\label{Eq.wt_v_L(Z)}
		v_L(Z):=v(z(Z))=\Theta_v\big(z(Z), u_L(z(Z))\big)\quad \forall\ Z\in(Z_1-\Upsilon_1, Z_1+\Upsilon_1).\index{$v_L(Z)$}
	\end{equation}
	Then $ v_L(Z_1)=v_1$, $z(Z_1)=0$. Due to $ \Delta_Z(Z,v)\mathrm{d}v=\Delta_v(Z,v)\mathrm{d}Z$,  we have $ \Delta_Z(Z,v_L)\frac{\mathrm{d}v_L}{\mathrm{d}Z}=\Delta_v(Z,v_L)$, i.e., $v_L$ solves \eqref{Eq.main_ODE}.
	By Proposition \ref{Prop.R_0},  $u_F(z)=u_L(z)$ for all $z\in(0, \zeta_0')$. Thus, we get 
	by \eqref{Eq.v_F_expression} and \eqref{Eq.wt_v_L(Z)}  that for $Z=\Theta_Z(z, u_F(z))=\Theta_Z(z, u_L(z))\in (Z_1-\Upsilon_1, Z_1)$,
	\[v_F(Z)=\Theta_v(z, u_F(z))=\Theta_v(z, u_L(z))=v_L(Z).\]
	This shows that the $P_0-P_1$ curve $v_F$ can be continued using $v_L$ to pass through $P_1$ smoothly. Moreover, we have
	\[\Psi(Z, v_L(Z))=\big(z(Z), u_L(z(Z))\big),\qquad\ \forall\ Z\in(Z_1-\Upsilon_1, Z_1+\Upsilon_1).\]

	For $Z\in (Z_1, Z_1+\Upsilon_1)$, we have $z(Z)\in(-\zeta_0', 0)$, $\Psi(Z, v_L(Z))=\big(z(Z), u_L(z(Z))\big)\in\mathcal D_2'$,
hence $(Z, v_L(Z))\in\mathcal R_2'\subset\mathcal R_2$, and $\Delta_Z(Z, v_L(Z))<0$, $0<v_L(Z)<Z<1$ (see \eqref{R2}).
Now we take $Z_0\in (Z_1, Z_1+\Upsilon_1)$ then $(Z_0, v_L(Z_0))\in\mathcal R_2'$. Let $v_C(Z)$ be the unique solution to the initial value problem
	\begin{align*}
		\frac{\mathrm{d}v}{\mathrm{d}Z}=\frac{\Delta_v(Z, v)}{\Delta_Z(Z, v)}\qquad
		\text{ with } v(Z_0)=v_L(Z_0).
	\end{align*}
	Then Lemma \ref{Lem.Z_v_global} implies that $v_C$ is defined on $[Z_0, +\infty)$ with $\Delta_Z(Z, v_C(Z))<0$, $|v_C(Z)|<1$ and $v_C(Z)<Z$ for all $Z\in[Z_0, +\infty)$. Also, the uniqueness implies that $v_C(Z)=v_L(Z)$ for all $Z\in[Z_0, Z_1+\Upsilon_1)$. 
	
	Let
	\[v_G(Z)=\begin{cases}
		v_F(Z), & \text{if }Z\in[0, Z_1],\\ v_L(Z), & \text{if }Z\in(Z_1-\Upsilon_1, Z_1+\Upsilon_1),\\
		v_C(Z), & \text{if }Z\in(Z_0, +\infty).
	\end{cases}\]
	then $v=v_G(Z)$ is a smooth solution defined on $Z\in[0, +\infty)$ to the $Z-v$ ODE \eqref{Eq.main_ODE}. 
	Notice that (i) by \eqref{vF1}, \eqref{R1} and $Z_1\in(0,1)$, we have $\Delta_Z(Z, v_F(Z))>0$, $ (Z,v_F(Z))\in\mathcal{R}_1$, $0<v_F(Z)<Z<1$ for $Z\in(0,Z_1)$;
(ii) $v(0)=v_F(0)=0$, $v(Z_1)=v_L(Z_1)=v_1$, $0<v_1<Z_1<1$;\,(iii) $\Delta_Z(Z, v_L(Z))<0$, $0<v_L(Z)<Z<1$ for $Z\in (Z_1, Z_1+\Upsilon_1)$;\,
 (iv) $\Delta_Z(Z, v_C(Z))<0$, $|v_C(Z)|<1$, $v_C(Z)<Z$ for all $Z\in[Z_0, +\infty)$. Then by the definition of $v_G$, we have $|v_G(Z)|<1$ for all $Z\in[0,+\infty)$, $v_G(Z)<Z$ for all $Z\in(0, +\infty)$,
$\Delta_Z(Z, v_G(Z))>0 $ for $Z\in(0,Z_1)$, $\Delta_Z(Z, v_G(Z))<0 $ for $Z\in(Z_1, +\infty)$, and $v_G(Z_1)=v_1$.
	
	{Finally, thanks to $m=\beta\ell/(\ell+1)$, $\g=\frac km-1>\frac{1}{\sqrt{\ell}}>\frac{1}{{\ell}}$, we have $0<m<k/\left(1+\frac{1}{{\ell}}\right)$, and
		$0<\beta=m\left(1+\frac{1}{{\ell}}\right)<k$.} 
\end{proof}

The proof of Theorem \ref{Thm.mainthm_ODE} for $(k,\ell)\in\mathcal K_*\setminus\mathcal K_1$ is in Section \ref{Sec.Proof_k=2_l_large}.

Let $v=v(Z)$ be the solution to \eqref{Eq.main_ODE} obtained by Theorem \ref{Thm.mainthm_ODE}. For the proof of Theorem 
\ref{Thm.mainthm}, we need the following lemma.

\begin{lemma}\label{lem1}
	It holds that $v_\infty:=\lim_{Z\to +\infty}v(Z)\in(-1, 1)$.\index{$v_\infty$} Let $\wt v(W):=v\left(\frac 1W\right)$\index{$\wt v(W)$} for $W>0$ and $\wt v(0):=v_\infty$. Then $\wt v\in C^{\infty}([0, +\infty))$.
\end{lemma}

\begin{proof}
	First of all, $\wt v\in C^{\infty}((0, +\infty))$, and we get by \eqref{Eq.main_ODE}  that 
	\begin{equation}
		\frac{\mathrm{d}\wt v}{\mathrm{d}W}=\frac{\left(1-\wt v^2\right)\left[m\left(1-\wt v^2\right)-k\wt v\left(W-\wt v\right)\right]}{\ell\left(W\wt v-1\right)^2-\left(W-\wt v\right)^2}=\frac{\Delta_{\wt v}(W, \wt v)}{\Delta_W(W, \wt v)}=:G(W, \wt v)\quad\forall\ W>0,
	\end{equation}
	where
	\[\Delta_{\wt v}(W, \wt v)=\left(1-\wt v^2\right)\left[m\left(1-\wt v^2\right)-k\wt v\left(W-\wt v\right)\right],\quad \Delta_W(W, \wt v)=\ell\left(W\wt v-1\right)^2-\left(W-\wt v\right)^2.\]
	If $ W\in(0,1/2)$ and $|\wt v|<1$, then
	\begin{align}\label{Z1}
		\Delta_W(W, \wt v)&=\ell\left(W\wt v-1\right)^2-\left(W-\wt v\right)^2=(\ell-1)\left(W\wt v-1\right)^2+(1-W^2)(1-\wt v^2)\\
		\notag&\geq (\ell-1)\left(W\wt v-1\right)^2\geq (\ell-1)(1-W)^2\geq (\ell-1)/4>0,
	\end{align}
	and
	\[|\Delta_{\wt v}(W, \wt v)|=\left(1-\wt v^2\right)|m\left(1-\wt v^2\right)-k\wt v\left(W-\wt v\right)|\leq m+2k.\]
	As $ |v(Z)|<1$ for $Z\in(0,+\infty)$, we have $|\wt v(W)|=|v\left(\frac 1W\right)|<1$ for $W\in(0,+\infty)$,	hence
	\begin{align*}
		\left|\frac{\mathrm{d}\wt v}{\mathrm{d}W}\right|=\frac{|\Delta_{\wt v}(W, \wt v)|}{|\Delta_W(W, \wt v)|}\leq \frac{m+2k}{(\ell-1)/4}\quad \text{for } W\in(0,1/2),
	\end{align*}
	i.e., $ \wt v'(W)$ is bounded for $W\in(0,1/2)$. Now we take $\wt v(0)=v_\infty=\wt v(1/2)-\int_0^{1/2}\wt v'(W)\mathrm{d}W$, then
	$\wt v(W)=\wt v(0)+\int_0^{W}\wt v'(w)\mathrm{d}w$ for $W\in(0,1/2)$, thus $\wt v(0)=v_\infty=\lim_{W\downarrow 0}\wt v(W)=\lim_{Z\to +\infty}v(Z)$ and $\wt v\in C([0, +\infty))$.
	
	As  $|\wt v(W)|<1$ for $W\in(0,+\infty)$, we have $ \wt v(0)=v_\infty\in[-1,1]$.
		It is obvious that $G$ is smooth near $(W=0, \wt v=v_\infty)$ as $\ell>1\geq v_\infty^2$.
	Let $G_1(W):=G(W, \wt v(W))$. Then $G_1$ is continuous on $[0, +\infty)$, and $ \wt v'(W)=G_1(W)$ for $W\in(0, +\infty)$,
	hence by L'H\^opital's rule,
	\[ \wt v'(0)=\lim_{W\downarrow0}\frac{\wt v(W)-\wt v(0)}{W}=\lim_{W\downarrow0}\frac{\int_0^W G_1(s)\,\ds}{W}=\lim_{t\downarrow0}G_1(t)=\lim_{t\downarrow0}\wt v'(t)=G_1(0).\]
	Therefore, $\wt v(W)$ is $C^1$ at $W=0$ and it solves the initial value problem
	\begin{align}\label{ODE1}\frac{\mathrm{d}\wt v}{\mathrm{d}W}=G(W, \wt v) \ (W\geq0), \qquad\wt v(0)=v_\infty.\end{align}
	The standard regularity theory of ODEs implies that $\wt v(W)$ is smooth at $W=0$.
	
	It remains to prove that  $ \wt v(0)=v_\infty\in(-1,1)$. Otherwise, $v_\infty^2=1$, $ \wt v\equiv v_\infty$ solves \eqref{ODE1}. By the uniqueness, we have $ \wt v(W)=v_\infty\in\{\pm1\}$ for $W>0$, which contradicts with $|\wt v(W)|<1$.
	\end{proof}

Now we are ready to prove our main result.

\begin{proof}[Proof of Theorem \ref{Thm.mainthm}]
	For the solution $v$ obtained by Theorem \ref{Thm.mainthm_ODE}, we denote
	\begin{align*}
		J(Z):=\frac{(-m+{kv(Z)}/{Z})(1-v(Z)^2)+v'(Z)(1-Zv(Z))}{(Z-v(Z))(1-v(Z)^2)}.
	\end{align*}
	As $Z>v(Z) $ for $Z\in (0,+\infty)$ and $v\in C^\infty([0,+\infty))$,
	we have $J\in C^\infty((0,+\infty))$. By \eqref{Eq.2.15}, we have
	\begin{align}
		\label{J2}&mv(Z)(1-v(Z)^2)+\ell{v'(Z)(Z-v(Z))}=(1-Zv(Z))(1-v(Z)^2)J(Z).
	\end{align}
	Thanks to $ |Zv(Z)|\leq|v(Z)|<1$ for $Z\in[0,1]$,  $J(Z)$ can be extended to $Z=0$ by taking
	\begin{align}\label{J3}
		&J(Z)=\frac{mv(Z)(1-v(Z)^2)+\ell{v'(Z)(Z-v(Z))}}{(1-Zv(Z))(1-v(Z)^2)}\quad \text{for}\ Z\in[0,1].
	\end{align}
	Then  \eqref{J2} holds for $ Z\in[0,+\infty)$. As $v\in C^\infty([0,+\infty))$ we have $J\in C^\infty([0,1])$, thus $J\in C^\infty([0,+\infty))$.
	
	 Now we define
	\begin{align}\label{rho}
		&\widehat{\varrho}(Z)=\exp\left(\frac{\ell+1}{\ell}\int_0^ZJ(s)\mathrm{d}s\right)>0.
	\end{align}
	Then $\widehat{\varrho}\in C^\infty([0,+\infty))$ and satisfies \eqref{rho1} and \eqref{rho2}. Let $\widehat{u^0}$, $\widehat{u}$ be defined by \eqref{u1}. Then we have
 \eqref{rho3} and \eqref{rho4}. Thus, \eqref{Eq.mainPDE1}, \eqref{Eq.mainPDE2} and \eqref{Eq.mainPDE3} hold for $\varrho(t, x)=\frac1{(T_*-t)^\beta}\hat\varrho(Z),$
	$u^0(t,x)=\wh{u^0}(Z),$ $u_r(t,x)=\wh u(Z)$, $Z=r/(T_*-t)$, and $ (\varrho,\vec{u})$ defined by \eqref{rho5} is a smooth solution to \eqref{Eq.relativistic_Euler} on
$[0, T_*)\times (\R^d\setminus\{0\})$.

To prove the smoothness of the solution $(\varrho,\vec{u})$ at $[0, T_*)\times\{0\}$, it is enough to show that $ \widehat{\varrho}(Z)$, $\wh{u^0}(Z) $, $\wh{u}(Z)/Z$ are smooth functions of $Z^2$ (near $0$). From the proof of Proposition \ref{Prop.P_0-P_1}, we know $v(Z)=Z\Phi(Z^2)$, and $\Phi(\varkappa)$ is analytic near $\varkappa=0$. Then we get by \eqref{u1} that 
\begin{align*}
		&\wh{u^0}(Z)=-\frac{1}{\sqrt{1-v(Z)^2}}=\Phi_1(Z^2),\quad \Phi_1(\varkappa)=-\frac{1}{\sqrt{1-\varkappa\Phi(\varkappa)^2}},\\
&\widehat{u}(Z)=\frac{v(Z)}{\sqrt{1-v(Z)^2}},\quad \frac{\wh{u}(Z)}{Z}=\Phi_2(Z^2),\quad \Phi_2(\varkappa)=\frac{\Phi(\varkappa)}{\sqrt{1-\varkappa\Phi(\varkappa)^2}},
	\end{align*}and $\Phi_1(\varkappa)$, $\Phi_2(\varkappa)$ are analytic near $\varkappa=0$. By \eqref{J3},  $J(Z)/Z=\Phi_3(Z^2)$ with
	\begin{align*}
		&\Phi_3(\varkappa)=\frac{m\Phi(\varkappa)(1-\varkappa\Phi(\varkappa)^2)+\ell{(\Phi(\varkappa)+2\varkappa\Phi'(\varkappa))(1-\Phi(\varkappa))}}
{(1-\varkappa\Phi(\varkappa))(1-\varkappa\Phi(\varkappa)^2)},
	\end{align*}which is also analytic near $\varkappa=0$. Then by \eqref{rho}, $\widehat{\varrho}(Z)=\Phi_4(Z^2)$ with
	\begin{align*}
		&\Phi_4(\varkappa)=\exp\left(\frac{\ell+1}{2\ell}\int_0^{\varkappa}\Phi_3(s)\mathrm{d}s\right),
	\end{align*}which is also analytic near $\varkappa=0$. Therefore, $(\varrho,\vec{u})$ is smooth at $[0, T_*)\times\{0\}$.
	
	To prove the smoothness of the solution $(\varrho,\vec{u})$ at $(t=T_*, x\in\R^d\setminus\{0\})$, which corresponds to $Z=+\infty$, we consider $\wt v(W):=v\left(\frac 1W\right)$ for all $W>0$ and $\wt v(0)=v_\infty $. Then by Lemma \ref{lem1},  $\wt v\in C^\infty([0,+\infty))$.
	Let
	\[\wt \varrho(W):=\frac1{W^\beta}\hat\varrho\left(\frac1W\right)>0\quad\forall\ W>0.\]
	For all $0<W\leq 1$, we deduce from \eqref{rho} that (recalling  $m=\beta\ell/(\ell+1)$)
	\begin{align}\label{Eq.wt_rho(W)}
		\frac{\mathrm d\wt\varrho(W)/\mathrm dW}{\wt\varrho(W)}=-\frac\beta W-\frac1{W^2}\frac{\hat\varrho'(1/W)}{\hat\varrho(1/W)}=-\frac\beta W-\frac{\ell+1}{\ell W^2}J\left(\frac1W\right)=-\frac\beta m\frac{J(1/W)+mW}{W^2}.
	\end{align}
	For $Z\geq1$, $W=1/Z$, we have
	\begin{align*}
		J(Z)&=\frac{(-m+{kv(Z)}/{Z})(1-v(Z)^2)+v'(Z)(1-Zv(Z))}{(Z-v(Z))(1-v(Z)^2)}
		\\
		&=\frac{(-m+k\widetilde{v}(W)W)(1-\widetilde{v}(W)^2)-W^2\widetilde{v}'(W)(1-\widetilde{v}(W)/W)}{(1/W-\widetilde{v}(W))(1-\widetilde{v}(W)^2)}\\
		&=W\frac{(-m+k\widetilde{v}(W)W)(1-\widetilde{v}(W)^2)-W\widetilde{v}'(W)(W-\widetilde{v}(W))}{(1-W\widetilde{v}(W))(1-\widetilde{v}(W)^2)}
	\end{align*}
	hence, for $0<W\leq1$,
	\begin{align*}
		\wt J(W):=\frac{J(1/W)+mW}{W^2}=\frac{(k-m)\widetilde{v}(W)(1-\widetilde{v}(W)^2)-\widetilde{v}'(W)(W-\widetilde{v}(W))}
		{(1-W\widetilde{v}(W))(1-\widetilde{v}(W)^2)},
	\end{align*}
	and we also define
	\[\wt J(0):=\frac{(k-m)\widetilde{v}(0)(1-\widetilde{v}(0)^2)+\widetilde{v}'(0)\widetilde{v}(0)}
	{1-\widetilde{v}(0)^2}\in\R.\]
	Due to $\wt v\in C^\infty([0,+\infty))$, we have $\wt J\in C^\infty([0,1])$. Now it follows from \eqref{Eq.wt_rho(W)} that
	\[\wt\varrho(W)=\wt\varrho(1)\exp\left(\frac\beta m\int_W^1\wt J(w)\,\mathrm{d}w\right)\quad\forall\ W\in(0,1],\]
	and we define
	\[\wt\varrho(0):=\wt\varrho(1)\exp\left(\frac\beta m\int_0^1\wt J(w)\,\mathrm{d}w\right)>0.\]
	Then $\wt\varrho\in C^\infty([0,1])$. 
	
	Recall that for all $x\neq 0$ and $t\in[0,T_*)$, we have
	\begin{align*}
		\varrho(t,x)&=(T_*-t)^{-\beta}\hat\varrho\left(\frac{|x|}{T_*-t}\right)=|x|^{-\beta}\left(\frac{|x|}{T_*-t}\right)^\beta
		\hat\varrho\left(\frac{|x|}{T_*-t}\right), \\
		u^0(t, x)&=\wh{u^0}\left(\frac{|x|}{T_*-t}\right), \quad u^i(t, x)=\hat u\left(\frac{|x|}{T_*-t}\right)\frac{x^i}{|x|} \ (1\leq i\leq d).
	\end{align*}
	For $x\in\R^d\setminus\{0\}$, we define
	\[\varrho(T_*, x)=|x|^{-\beta}\wt\varrho(0),\quad u^0(T_*, x)=-\frac1{\sqrt{1-v_\infty^2}},\quad u^i(T_*, x)=\frac{v_\infty}{\sqrt{1-v_\infty^2}}\frac{x^i}{|x|} \ (1\leq i\leq d).\]
	For fixed $(T_*, x_0\in\R^d\setminus\{0\})$, we consider a neighborhood $\mathcal O=(T_*-\delta_1, T_*]\times B_{\R^d}(x_0, \delta_2)$, where $\delta_1>0$ and $\delta_2>0$ are such that $\delta_1<\delta_2<|x_0|/2$. Thus, for $(t,x)\in \mathcal O$, we have $x\neq0$ and $|T_*-t|<|x|$, hence (using \eqref{u1})
	\begin{align*}
		&\varrho(t, x)=|x|^{-\beta}\wt\varrho\left(\frac{T_*-t}{|x|}\right),\quad u^0(t, x)=-\frac1{\sqrt{1-\wt v\big((T_*-t)/|x|\big)^2}},\\
		&u^i(t, x)=\frac{\wt v\big((T_*-t)/|x|\big)}{\sqrt{1-\wt v\big((T_*-t)/|x|\big)^2}}\frac{x^i}{|x|}\quad (1\leq i\leq d).
	\end{align*}
Thanks to $\wt v\in C^\infty([0,1])$ and $\wt\varrho\in C^\infty([0,1])$, we know that $(\varrho,\vec{u})$ is smooth at $(T_*, x_0)$.
	
	Finally, we prove the asymptotics.  By \eqref{u1} and Lemma \ref{lem1}, we have
	\begin{equation*}
		\lim_{|Z|\to+\infty}\wh{u^0}(Z)=-\frac{1}{\sqrt{1-v_{\infty}^2}}=:u^0_*\in\R\setminus\{0\}, \quad \lim_{|Z|\to+\infty}\wh{u}(Z)=\frac{v_{\infty}}{\sqrt{1-v_{\infty}^2}}=:u_*\in\R.
	\end{equation*}
	As $Z=|x|/(T_*-t)\to+\infty$, we have
	\[\varrho(t, x)=(T_*-t)^{-\beta}\hat\varrho(Z)=|x|^{-\beta}\wt\varrho(1/Z)=\frac{\wt\varrho(0)\big(1+o(1)\big)}{|x|^{\beta}}.\]
	
	The proof of Theorem \ref{Thm.mainthm} is completed.
	\end{proof}

\section{Proof of Theorem \ref{Thm.mainthm_ODE} for $k=2$ and $\ell$ large}\label{Sec.Proof_k=2_l_large}

This section is devoted to proving Theorem \ref{Thm.mainthm_ODE} for $(k,\ell)\in\mathcal K_*\setminus\mathcal K_1$, i.e., $k=2$ and $\ell\geq \ell_1(2)>3/2$ (by Table \ref{Tab.1}). The proof is much involved and delicate.  \smallskip

Let's start with some barrier functions.
\begin{lemma}\label{Lem.u_1_compare}
	Let $u_{(1)}(z)=\varepsilon+a_1z$. Then we have
	\begin{equation}\label{Eq.u_1_compare}
		\mathcal L\left(u_{(1)}\right)(z)<0\quad \forall \ z\in\left[-\varepsilon/{a_1}, 0\right).
	\end{equation}
\end{lemma}
\begin{proof}
	By \eqref{Lu1} we have $\mathcal L\left(u_{(1)}\right)(z)=z^2q_3(z), $
	where $q_3(z)=-(R-2)\delta a_2+(1-\ell)a_1z$ is an affine function in $z$. Let $\xi_1=-\varepsilon/a_1<0$. By \eqref{Eq.L(u_1)(xi_1)},  \eqref{Eq.a_2>0} and \eqref{Eq.a_1>4}, we have
	\[\mathcal L\left(u_{(1)}\right)(\xi_1)=(1-\xi_1)f(\xi_1)a_1<0,\]
	which gives $q_3(\xi_1)<0$. Since the slope of $q_3$ is negative, we know that $q_3(z)<0$ for all $z\in[\xi_1, +\infty)$. Now \eqref{Eq.u_1_compare} follows.
\end{proof}

\begin{lemma}\label{Lem.u_4_compare}
	Let $\xi_2=-a_2/a_3$, $M=a_3^2/(4a_2)$, $U_3(z)=\varepsilon+a_1z+a_2z^2+a_3z^3+Mz^4$. Assume that $M>a_4$, $a_3/a_2>2a_1/\varepsilon>0$, $R\in(3,4)$, and\begin{align}\label{q51}
		&5a_3-2(4k+1)a_2-2(B+k)>0,\quad 3a_3-2ka_2-2(B+k)>0.
	\end{align}
	Then we have
	\begin{equation}\label{Eq.u_4_compare}
		\mathcal L\left(U_3\right)(z)<0\quad \forall \ z\in\left[-2{a_2}/{a_3}, 0\right).
	\end{equation}
\end{lemma}\begin{proof}
	Plugging the definition of $U_3$ into Lemma \ref{Lem.recurrence_relation} and using \eqref{Eq.a_n_recurrence}, we find that $\mathcal L(U_3)(z)=z^4q_6(z)$, where $q_6$ is a polynomial of degree $4$ given by
	\[q_6(z)=\mathcal U_4^{(3)}+\mathcal U_5^{(3)}z+\mathcal U_6^{(3)}z^2+\mathcal U_7^{(3)}z^3+\mathcal U_8^{(3)}z^4,\]
	where
	\begin{align*}
		\mathcal U_4^{(3)}&=(R-4)\delta (M-a_4),\\
		\mathcal U_5^{(3)}&=\left(6a_2-(5k+1)a_1+2A+4B+2k\right)M+3a_3^2-(5k+1)a_2a_3-(B+2k)a_3,\\
		\mathcal U_6^{(3)}&=7a_3M-(3k+1)(2a_2M+a_3^2)-2(B+k)M,\\
		\mathcal U_7^{(3)}&=4M^2-(7k+3)a_3M,\qquad \mathcal U_8^{(3)}=-2(2k+1)M^2.
	\end{align*}
	As $\xi_2=-a_2/a_3$, $M=a_3^2/(4a_2)$, we have $a_3^2=4a_2M$,  $a_3=-4\xi_2M$, $a_2=-a_3\xi_2=4M\xi_2^2$ and
	\begin{align*}
		&q_6(z)\\
		=&z(Mz+a_3/2)\big[-2(2k+1)Mz^2+(4M-(5k+2)a_3)z+5a_3-2(4k+1)a_2-2(B+k)\big]\\
		&+(R-4)\delta (M-a_4)+(B_4M+a_3^2/2-k(a_2+1)a_3)z\index{$B_4$}\\
		=&z(Mz+a_3/2)q_{6,1}(z)+q_{6,2}(z)=zM(z-2\xi_2)q_{6,1}(z)+q_{6,2}(z),\end{align*}
	here
	\begin{align*}
		&B_4:=6a_2-(5k+1)a_1+2A+4B+2k,\index{$B_4$}\\
		&q_{6,1}(z):=-2(2k+1)Mz^2+(4M-(5k+2)a_3)z+5a_3-2(4k+1)a_2-2(B+k),\\
		&q_{6,2}(z):=(R-4)\delta (M-a_4)+(B_4M+a_3^2/2-k(a_2+1)a_3)z.
	\end{align*}
	Since $0<a_2=4M\xi_2^2$, $q_{6,1}''(z)=-4(2k+1)M<0 $, $q_{6,2}''(z)=0 $, we infer that for $z\in [2\xi_2,0]$,
	\begin{align*}
		& q_{6,1}(z)\geq \min\{q_{6,1}(0),q_{6,1}(2\xi_2)\},\quad q_{6,2}(z)\leq \max\{q_{6,2}(0),q_{6,2}(2\xi_2)\}.\end{align*}
	As $a_3=-4M\xi_2$, $a_2=-a_3\xi_2=4M\xi_2^2$, we get by \eqref{q51}  that
	\begin{align*}
		q_{6,1}(0)&=5a_3-2(4k+1)a_2-2(B+k)>0,\\
		q_{6,1}(2\xi_2)&=-8(2k+1)M\xi_2^2+2(4M-(5k+2)a_3)\xi_2+5a_3-2(4k+1)a_2-2(B+k)\\
		&=-2(2k+1)a_2-2(a_3-(5k+2)a_2)+5a_3-2(4k+1)a_2-2(B+k)\\
		&=3a_3-2ka_2-2(B+k)>0.
	\end{align*}
	Thus, $q_{6,1}(z)>0 $, $q_6(z)=zM(z-2\xi_2)q_{6,1}(z)+q_{6,2}(z)\leq q_{6,2}(z)$ for $z\in [2\xi_2,0]$.
	As $M>a_4$, $R<4$, $ \delta>0$, we have $q_{6,2}(0)=(R-4)\delta (M-a_4)<0$. As $a_3=-4M\xi_2$, $a_2=4M\xi_2^2$, we have
	\begin{align}\label{u4-u1}
		U_3(z)=\varepsilon+a_1z+a_2z^2+a_3z^3+Mz^4=u_{(1)}(z)+Mz^2(z-2\xi_2)^2.
	\end{align}
	Thus $U_3(2\xi_2)=u_{(1)}(2\xi_2) $,
	$U_3'(2\xi_2)=u_{(1)}'(2\xi_2) $, and $\mathcal L\left(U_3\right)(2\xi_2)=\mathcal L\left(u_{(1)}\right)(2\xi_2) $.
	
	As $a_3/a_2>2a_1/\varepsilon>0$, we have $2\xi_2=-2a_2/a_3\in \left[-\varepsilon/{a_1}, 0\right)$, then we get by Lemma \ref{Lem.u_1_compare}  that
	$\mathcal L\left(u_{(1)}\right)(2\xi_2)<0$, thus $\mathcal L\left(U_3\right)(2\xi_2)<0 $.
	As $\mathcal L(U_3)(z)=z^4q_6(z)$, we have
	$q_6(2\xi_2)<0 $. As $q_6(z)=zM(z-2\xi_2)q_{6,1}(z)+q_{6,2}(z)$ we have $q_6(2\xi_2)=q_{6,2}(2\xi_2)<0 $.
	Thus, $q_6(z)\leq q_{6,2}(z)\leq \max\{q_{6,2}(0),q_{6,2}(2\xi_2)\}<0 $ for $z\in [2\xi_2,0]$,  and
	$\mathcal L(U_3)(z)=z^4q_6(z)<0$ for $z\in [2\xi_2,0)$.
\end{proof}
	
Now we assume $k=2$. Let $R_*=\frac{100}{27}$ then $3<R_*<2+\sqrt{3}<\sqrt{14}<4$. Recall that $B_2>0$ does not necessarily hold for $(k,\ell)\in\mathcal K_*\setminus \mathcal K_1$, hence we can not deduce $a_4<0$ as in Lemma \ref{Lem.a_3>0a_4<0}.  Thus, we need to reconsider the qualitative properties of the coefficients $a_n$.

\begin{lemma}\label{Lem.a_4<M}
	Assume that $k=2$, $\ell>1$, $R\in(3,R_*)$, $M=a_3^2/(4a_2)$. Then $M>a_4$, $a_3>4a_2$.
\end{lemma}

\begin{lemma}\label{Lem.a_1<}
	Assume that $k=2$, $\ell>3/2$, $R\in(3,4-1/\ell)$. Then $a_1<2\varepsilon$, $a_2>7$.
\end{lemma}

In fact, if $k=2$, $ \gamma=2$, then $R=4-1/\ell$ and $a_1=2\varepsilon$. The proof of Lemma \ref{Lem.a_4<M} and Lemma \ref{Lem.a_1<} will be delayed to Appendix  \ref{Subsec.proof_C}.

\begin{lemma}\label{Lem.u_4<}
	Assume that $k=2$, $\ell>3/2$, $R\in(3,R_*)\cap(3,4-1/\ell)$. Let $\xi_2=-a_2/a_3$,
	\begin{align}\label{u4*}
		U_3^*(z)=U_3(z)\quad \text{for } z\in[2\xi_2,0],\quad U_3^*(z)=u_{(1)}(z)\quad \text{for } z\in(-\infty,2\xi_2].
	\end{align}
	Then $\mathcal L(U_3^*)(z)<0$ for $z\in[\xi_1,0)$ and $u_g(z)>U_3^*(z)$ for $z<0$. Here $\xi_1=-\varepsilon/a_1$.
\end{lemma}\begin{proof}
	Note that the conditions in Lemma \ref{Lem.a_4<M} and Lemma \ref{Lem.a_1<} are satisfied.
	Now we verify the conditions in Lemma \ref{Lem.u_4_compare}. As $R_*<4$, $ 4-1/\ell<4$, we have $R\in(3,4)$.
	By Lemma \ref{Lem.a_4<M}, we have $M>a_4$, $a_3/a_2>4$. By Lemma \ref{Lem.a_1<} and \eqref{Eq.a_1>4}, we have $0<a_1/\varepsilon<2$, $a_2>7$. Thus, $a_3/a_2>4>2a_1/\varepsilon>0$. It remains to check \eqref{q51}. Thanks to $a_3/a_2>4$, $a_2>7$, $k=2$, $B=2k+1-\ell<2k$, we get
	\begin{align*}
		&5a_3-2(4k+1)a_2-2(B+k)>5\cdot4a_2-18a_2-2(3k)=2a_2-12>0,\\
		&3a_3-2ka_2-2(B+k)>3\cdot4a_2-4a_2-2(3k)=8a_2-12>0.
	\end{align*}
	Thus, the conditions in Lemma \ref{Lem.u_4_compare} are satisfied, and \eqref{Eq.u_4_compare} is true.
	By \eqref{u4-u1}, \eqref{u4*}, \eqref{Eq.u_4_compare} and \eqref{Eq.u_1_compare}, we have $U_3^*\in C^1((-\infty,0])$ and
	$\mathcal L(U_3^*)(z)=\mathcal L(U_3)(z)<0$ for $z\in [2\xi_2,0)$, $\mathcal L(U_3^*)(z)=\mathcal L(u_{(1)})(z)<0$ for $z\in [\xi_1,2\xi_2]$. Thus, $\mathcal L(U_3^*)(z)<0$, for $z\in[\xi_1,0)$. Here we have used $\xi_1=-\varepsilon/a_1$, $\xi_2=-a_2/a_3$, $a_3/a_2>2a_1/\varepsilon>0$ to deduce that $ \xi_1<2\xi_2<0$. By Lemma \ref{Lem.relative_positions}, we have $u_g(z)>u_{(1)}(z)= U_3^*(z)$ for $z\leq 2\xi_2$.
	Assume for contradiction that there exists $\zeta_*\in(2\xi_2, 0)$ such that $u_g(z)>U_3^*(z)$ for all $z<\zeta_*$ but $u_g(\zeta_*)=U_3^*(\zeta_*)=U_3(\zeta_*)$, then $u_g'(\zeta_*)\leq {U_3^*}'(\zeta_*)=U_3'(\zeta_*)$.
Thus, $\zeta_*\in(2\xi_2, 0)\subset (\xi_1, 0)\subset (z_g, 0)$, by \eqref{Eq.u_g_compare} we obtain $-\Delta_z(\zeta_*, u_g(\zeta_*))>0$,
	and then by \eqref{Luz} and  \eqref{Eq.u_4_compare}, we have
	\begin{align*}
		\mathcal L(u_g)(\zeta_*)&=-\Delta_z(\zeta_*, u_g(\zeta_*))u_g'(\zeta_*)+\Delta_u(\zeta_*, u_g(\zeta_*))\\
		&\leq -\Delta_z\left(\zeta_*, U_3(\zeta_*)\right)U_3'(\zeta_*)+\Delta_u\left(\zeta_*, U_3(\zeta_*)\right)=\mathcal L\left(U_3\right)(\zeta_*)<0,
	\end{align*}
	which contradicts with \eqref{Eq.u_g_compare}. Thus, $u_g(z)>U_3^*(z)$ for all $z<0$.
\end{proof}
We have used the fact that $a_1>4>0$, $a_2>0$, $ \varepsilon>0$, $ \delta>0$, $ \xi_1=-\varepsilon/{a_1}>z_g$ (see \eqref{Eq.u_g_u_1_3}). We also used the following fact for convex functions:
\begin{align*}
	\text{if $f''(z)\geq0$ for $z\in[a,b]$ then $f(z)\leq\max\{f(a),f(b)\}$ for $z\in[a,b]$};\\
	\text{if $f''(z)\leq0$ for $z\in[a,b]$ then $f(z)\geq\min\{f(a),f(b)\}$ for $z\in[a,b]$}.
\end{align*}

Now we define a domain $\mathcal D_2''$ by
\begin{equation}
	\mathcal D_2''=\{(z, u): 0<u<U_3^*(z),\ z_g<z<0\}.
\end{equation}By Lemma \ref{Lem.u_4<} we have $\mathcal D_2''\subset\mathcal D_2 $.
Let $\mathcal R_2''=\Psi^{-1}(\mathcal D_2'')$ then $\Psi: \mathcal R_2''\to\mathcal D_2''$ is a bijection and
$\mathcal R_2''\subset\mathcal R_2 $. Similar to Lemma \ref{Lem.Z_v_global}, we have the following result (with $u_{(3)}$ replaced by $U_3^*$).
\begin{lemma}\label{lem4}
	Assume that $k=2$, $\ell>3/2$, $R\in(3,R_*)\cap(3,4-1/\ell)$.	Given $(Z_0, v_0)\in\mathcal R_2''$ and $Z_0<1$. Let $v(Z)$ be the unique solution of
	\eqref{Eq.Z_v_tau}.
	Then  the solution exists for $Z\in[Z_0,+\infty)$ and $\Delta_Z(Z, v(Z))<0$, $|v(Z)|<1$, $v(Z)<Z$  for $Z\in[Z_0,+\infty)$.
\end{lemma}

In fact, Lemma \ref{lem4} follows from Lemma \ref{Lem.Z_v_alternative}, Lemma \ref{Lem.u_4<}, Lemma \ref{Lem.f<0} and the fact that $z\leq-\varepsilon/a_1$ implies $U_3^*(z)\leq 0$ (using the argument as in the proof of Lemma \ref{Lem.Z_v_global}).\smallskip

In place of Lemma \ref{Lem.Rto4}, we have the following
\begin{lemma}\label{Lemu3a}
	Assume that $k=2$, $\ell>3/2$, $R=\min\{4-1/\ell,R_*\}$. Then $a_4<0$, and
	$$ \mathcal L\left(u_{(3)}\right)(z)<0, \qquad\forall\ z>0.$$
\end{lemma}
We postpone the proof of Lemma \ref{Lemu3a} until Appendix \ref{Subsec.proof_C}.

\begin{lemma}\label{Lem.u_3<u_F}
	Assume that $k=2$, $\ell>3/2$, $ R=\min\{4-1/\ell,R_*\}$. Then we have
	\begin{equation}\label{Eq.u_3<u_F}
		u_{(3)}(z)<u_F(z),\qquad \forall\ z\in\left(0, \frac1{k+1}\right).
	\end{equation}
\end{lemma}
\begin{proof}
	This is a consequence of $\lim_{z\uparrow1/(k+1)}u_F(z)=+\infty$, $\mathcal L(u_F)=0$, Lemma \ref{Lemu3a} and $\Delta_z(z, u_F(z))>0$ for $z\in(0,\frac{1}{k+1})$ (see \eqref{Eq.u_p<u_F<u_b}) combining with a contradiction argument as in the proof of Lemma \ref{Lem.u_2<u_F}. We omit the details here.
\end{proof}

Now we are in a position to prove Proposition \ref{Prop.R_0} and Theorem \ref{Thm.mainthm_ODE} for $(k,\ell)\in \mathcal K_*\setminus\mathcal K_1$.

\if0\begin{proposition}\label{Prop.R_0a}
	Let $k=2$, $\ell>3/2$ then there exists $R_0\in(3,\min\{4-1/\ell,R_*\})$ such that $u_L(\cdot; R_0)\equiv u_F(\cdot; R_0)$. As a result, for this particular $R=R_0$, the $Q_0-Q_1$ solution curve $u_F$ of the $z-u$ ODE can be continued to pass $Q_1$ smoothly.
\end{proposition}\fi
\begin{proof}[Proof of Proposition \ref{Prop.R_0} for $(k,\ell)\in \mathcal K_*\setminus\mathcal K_1$]
	Assume that $(k,\ell)\in \mathcal K_*\setminus\mathcal K_1$, then $k=2$ and $\ell\geq \ell_1(2)>3/2$. Let $\theta_1\in(0,1)$ and $\zeta_1(R)$ be given by Lemma \ref{Lem.Rto3}. Let $R_1=3+\theta_1/4\in(3,4)$, $R_2=\min\{4-1/\ell,R_*\}\in(3,4)$.
	Then $R_1<3+1/3<R_2$. By $a_4(R_2)<0$ (Lemma \ref{Lemu3a}), there exists $\zeta_2>0 $ such that $u_L(z; R_2)<u_{(3)}(z; R_2)$ for $z\in(0,\zeta_2]$.
	Now let $\zeta=\min\{\zeta_1(R_1),\zeta_2, \zeta_0([R_1,R_2])\}$. Lemma \ref{Lem.Rto3} implies that
	$u_L(\zeta; R_1)>u_F(\zeta; R_1)$ and
	Lemma \ref{Lem.u_3<u_F}  implies that
	\[u_L(\zeta; R_2)<u_{(3)}(\zeta; R_2)<u_F(\zeta; R_2).\]
	By Proposition \ref{Prop.local_solution_Q_1} and Remark \ref{Rmk.u_F_continuous}, we know that the function $[R_1,R_2]\ni R\mapsto u_L(\zeta; R)-u_F(\zeta; R)$ is continuous. By the intermediate value theorem, there exists $R_0\in (3,R_2)$ such that $u_L(\zeta; R_0)=u_F(\zeta; R_0)$. Therefore, by the uniqueness of solutions to the $z-u$ ODE \eqref{Eq.ODE_z_u} with $u(\zeta)=u_F(\zeta)$, we know that $u_L(z; R_0)\equiv u_F(z; R_0)$ for
	$z\in(0, \zeta)$, hence the solution $u_F$ can be continued using $u_L$ to pass through $Q_1$ smoothly.
\end{proof}

\begin{proof}[Proof of Theorem \ref{Thm.mainthm_ODE} for $(k,\ell)\in \mathcal K_*\setminus\mathcal K_1$]
	Let $k=2$, $\ell\geq \ell_1(2)>3/2$. Let $R=R_0\in(3, \min\{4-1/\ell,R_*\})$ and $\zeta>0$ be given as above. It follows from $a_4<M:=a_3^2/(4a_2)$ (Lemma \ref{Lem.a_4<M}) that there exists $\zeta_0'\in(0, \zeta)\cap(0,2a_2/a_3)$ such that for all $z\in(-\zeta_0', 0)$ we have
	\[u_L(z)=\varepsilon+a_1z+a_2z^2+a_3z^3+a_4z^4+\cdots<\varepsilon+a_1z+a_2z^2+a_3z^3+Mz^4=U_3(z)=U_3^*(z),\]
	hence, $(z, u_L(z))\in \mathcal D_2''$ for all $z\in(-\zeta_0', 0)$. The rest of proof is very similar to the proof of Theorem \ref{Thm.mainthm_ODE} for $(k,\ell)\in \mathcal K_1$, by replacing the application of Lemma \ref{Lem.Z_v_global} with the application of Lemma \ref{lem4}. We leave the details to the reader.
\end{proof}

\appendix
\section{Proof of Lemma \ref{Lem.a_4<0}}\label{Appen.B_2>0}
In this appendix, we prove Lemma \ref{Lem.a_4<0}. By \eqref{Eq.B_2_expression} and Lemma \ref{Lem.A_a_1}, we have
\begin{align*}
	&\quad(2R+k+4)(R-2)B_2=-(2R+k+4)\big((k-15)R+8k+5\big)a_1\\
	&\qquad -(16R-2)\left(\left(2-\frac kR\right)(R+1)a_1+\mu\right)+(2R+k+4)\big(2(4k+B)R+4k+6B\big)\\
	&=-\left(2(k+1)R^2+(k^2-11k-22)R+8k^2+23k+16+\frac{2k}R\right)a_1\\
	&\qquad +4(4k+B)R^2+2\big(4k(k+5)+(k+10)B-8\mu\big)R+(k+4)(4k+6B)+2\mu,
\end{align*}
which gives
\begin{align}
	(2R+k+4)(R-2)B_2=-Rg_k(R)a_1+2Rh_{k, \ell}(R)=-Rg_k(R)(a_1-4)+2Rh_{k,\ell}^*(R),\label{Eq.A.1}
\end{align}
where
\begin{align*}
	g_k(R)&=2(k+1)R+k^2-11k-22+\frac{8k^2+23k+16}R+\frac{2k}{R^2},\\
	h_{k,\ell}(R)&=2(4k+B)R+4k(k+5)+(k+10)B-8\mu+\frac{(k+4)(2k+3B)+\mu}{R},\\
	h_{k,\ell}^*(R)&=h_{k,\ell}(R)-2g_k(R).
\end{align*}

\begin{lemma}\label{Lem.A.1}
	If $k=1, \ell>1$ and $R\in[3,4]$, then $B_2>0$.
\end{lemma}
\begin{proof}
	We have
	\begin{align*}
		g_1(R)=4R-32+\frac{47}R+\frac2{R^2}=\frac{4(R-2)(R-6)}{R}-\frac{R-2}{R^2}<0\quad \forall\ R\in[3,4].
	\end{align*}
	Thanks to  $B=3-\ell$ and $\mu<13-4\ell$ (by \eqref{mu}), for all $\ell>1$ and $R\in[3,4]$, we have
	\begin{align*}
			Rh_{1,\ell}(R)&=2(4+B)R^2+(24+11B-8\mu)R+{5(2+3B)+\mu}\\
			&=2(4+B)R^2+(24+11B)R+{5(2+3B)}-\left(8R-1\right)\mu\\
			&>2(7-\ell)R^2+(57-11\ell)R+{5(11-3\ell)}-\left(8R-1\right)(13-4\ell)\\
			&=14R^2-47R+{68}-\left(2R^2-21R+{19}\right)\ell\\
			&=(R-3)(14R-5)+53+(R-1)(19-2R)\ell>0.
	\end{align*}
	Therefore, $B_2>0$ by \eqref{Eq.A.1}.
	\end{proof}

\begin{lemma}\label{Lem.A.2}
	For $k\geq 2, 1<\ell<k, R\in[3,4]$, we have
	\begin{equation*}
		g_k'(R)\leq 0,\qquad g_k(R)>0,\qquad \left(h_{k,\ell}^{*}\right)'(R)>0.
	\end{equation*}
\end{lemma}
\begin{proof}
	For $k\geq 2$ and $R\in[3,4]$, we have
	\begin{align*}
		g_k'(R)&=2(k+1)-\frac{8k^2+23k+16}{R^2}-\frac{4k}{R^3}\leq 2(k+1)-\frac{8k^2+23k+16}{4^2}-\frac{4k}{4^3}\\
		&=-\frac{(k+1)(k-2)}{2}\leq 0,
	\end{align*}
	which gives
	\begin{align*}
		g_k(R)&\ge g_k(4)=8(k+1)+k^2-11k-22+\frac{8k^2+23k+16}4+\frac{2k}{4^2}\\
		&=3k^2+\frac{23}8k-10>3k^2-10>0.
	\end{align*}
	We also have
	\begin{align*}
		\left(h_{k,\ell}^{*}\right)'(R)=h_{k,\ell}'(R)-2g_k'(R)=2(4k+B)-4(k+1)+\frac D{R^2}+\frac{8k}{R^3},
	\end{align*}
	where $D=2(8k^2+23k+16)-\big((k+4)(2k+3B)+\mu\big)$. Since $1<\ell<k$, we have $0<B=2k+1-\ell<2k$; by \eqref{Eq.mu}, we also have $\mu\leq (k+2)^2$, hence $2(4k+B)-4(k+1)=2(2k-2+B)>0$ and $(k+4)(2k+3B)+\mu<(k+4)(8k)+(k+2)^2=9k^2+36k+4$, then
	\[D>2(8k^2+23k+16)-(9k^2+36k+4)=7k^2+10k+28>0.\]
	Thus, $\left(h_{k,\ell}^{*}\right)'(R)>0$.
\end{proof}

\begin{lemma}\label{Lem.A.3}
	If $k\geq 1, \ell>1$, then the function $R\mapsto a_1$ is strictly decreasing for $R$ satisfying $\g=\g(R)\in \Gamma(k,\ell)$.
\end{lemma}
\begin{proof}
	Recall from Remark \ref{Rmk.R_parameter} that when $\gamma\in \Gamma(k,\ell)$, the function $\g\mapsto R$ is strictly decreasing. We first prove
	\eqref{p0g}. In fact, we get by \eqref{Eq.epsilon_A_B_gamma}  that
	\begin{align*}
		&p_0(2(1+s\g))\\
		&=4(1+s\g)^2-2\big(4-(k-2\ell)\g+(k-2)\ell\g^2\big)(1+s\g)+2\big((k-2\ell)\g-2\big)(\ell\g^2-1)\\
		&=2\left[2s^2-k\ell+(k-2\ell)s\right]\g^2+2\left[(k-2\ell)-(k-2)s\right]\ell\g^3\\
		&=2(2s+k)(s-\ell)\g^2+2\left[(k-2\ell)-(k-2)s\right]\ell\g^3,
	\end{align*}
	which gives \eqref{p0g}.
	Now we take
	$s:=\frac{a_1-2}{2\g}$. By \eqref{Eq.a_1>4},  we have $s>\ell.$
	Now $a_1=2(1+s\g)$, by $p_0(a_1)=0$ and \eqref{p0g}, we deduce
	\begin{align*}
		0&=p_0(a_1)=p_0(2(1+s\g))\\
		&=2(2s+k)(s-\ell)\g^2+2\left[(k-2\ell)-(k-2)s\right]\ell\g^3,
	\end{align*}
	which gives
	\[\frac1{\ell\g}=\frac{(k-2)s-(k-2\ell)}{(2s+k)(s-\ell)}=\frac1{2s+k}\left(k-2+\frac{k(\ell-1)}{s-\ell}\right).\]
	As $s>\ell$, the function $s\mapsto\frac1{\g}$ is strictly decreasing and thus $s\mapsto \g$ is strictly increasing, hence $\g\mapsto s$ is also strictly increasing. Then it follows from $a_1=2(1+s\g)$, $s>\ell>0$ and the increase of $\g\mapsto s$ that $\g\mapsto a_1$ is strictly increasing. Composing $\g\mapsto a_1$ with a strictly decreasing function $R\mapsto \g$ implies that $R\mapsto a_1$ is strictly decreasing.
\end{proof}

In Lemma \ref{Lem.A.1}, we have shown that for $k=1$, $B_2>0$ holds for all $\ell>1$ and $R\in[3,4]$. In what follows, we assume $k\geq 2$. By \eqref{Eq.A.1}, for $k\geq 2$ and $R\in[3,4]$, we have
\begin{equation*}
	B_2>0\Longleftrightarrow (a_1-4)g_k(R)<{2h_{k,\ell}^*(R)}{}.
\end{equation*}
Lemma \ref{Lem.A.2} implies that $R\in[3, 4]\mapsto{g_k(R)}\in\R_+$ is strictly decreasing; Lemma \ref{Lem.A.3} and \eqref{Eq.a_1>4} imply that $R\in[3,4]\mapsto a_1-4\in\R_+$ is strictly decreasing. Thus, $R\in[3, 4]\mapsto(a_1-4){g_k(R)}$ is strictly decreasing.
Lemma \ref{Lem.A.2} also implies that $R\in[3, 4]\mapsto{h_{k,\ell}^*(R)}$ is strictly increasing if $1<\ell<k$. As a consequence, we conclude
\begin{equation}\label{B2>0}
	B_2>0\ \text{for all }R\in[3,4]\qquad\Longleftrightarrow\qquad B_2>0\ \text{for }R=3 \quad \text{if }1<\ell<k.
\end{equation}

Now we evaluate $B_2(R=3)$ and assume $R=3$ and $k\geq 2$ in the sequel.\smallskip

For $R=3$, by \eqref{Eq.a_1+c_1=Rdelta} and \eqref{Eq.Rdelta^2}, we have $ a_1=3\delta-2\varepsilon$, $A=(R+1)\delta-(k+2)\varepsilon=4\delta-(k+2)\varepsilon$. As $A>0$, $ \varepsilon>0$, we get by \eqref{Eq.Rdelta^2}  that
\begin{align*}
	&(k+2)\varepsilon<4\delta,\quad R(k+2)^2\varepsilon^2<16R\delta^2=32k\varepsilon(1+\varepsilon),\quad R(k+2)^2\varepsilon<32k(1+\varepsilon),\\
	&(R(k+2)^2-32k)\varepsilon<32k,\quad R(k+2)^2-32k=3(k+2)^2-32k=(3k-2)(k-6).
\end{align*}
Thus, $ \varepsilon\in E_k$ with 
\begin{align}\label{Ek}E_k:=\begin{cases}
		(0, +\infty), & k\in\{2, 3, 4, 5, 6\},\\
		\left(0,\frac{32k}{(3k-2)(k-6)}\right), & k\geq 7.
\end{cases}\end{align}
For $R=3$, by \eqref{Eq.B_2_expression}, \eqref{Eq.Rdelta^2} and \eqref{Eq.epsilon_A_B_gamma}, we get
\begin{align*}
	B_2(R=3)&=-(11k-40)a_1-46A+6(6k+1-\ell)+4k+6(2k+1-\ell)\\
	&=-(11k-40)(3\delta-2\varepsilon)-46\big(4\delta-(k+2)\varepsilon\big)+52k-12\ell+12\\
	&=(68k+12)\varepsilon-(33k+64)\sqrt{\frac{2k}{3}}\sqrt{\varepsilon(1+\varepsilon)}+52k-12\ell+12.
\end{align*}
Recall that
\[(k+2)\sqrt{1+\varepsilon}=(1+R)\sqrt{\frac{2k\varepsilon}{R}}+\frac{k-2\ell}{\sqrt\ell},\]
hence (by $R=3$)
\[(k+2)\sqrt{1+\varepsilon}-4\sqrt{\frac{2k\varepsilon}{3}}=\frac{k-2\ell}{\sqrt\ell}=\frac k{\sqrt\ell}-2\sqrt\ell.\]
We denote
\[f_1(k,\varepsilon)=(k+2)\sqrt{1+\varepsilon}-4\sqrt{\frac{2k\varepsilon}{3}},\qquad f_2(k, \ell)=\frac k{\sqrt\ell}-2\sqrt\ell.\]
Then $\ell\mapsto f_2(k,\ell)$ is strictly decreasing for $\ell>0$ with
\begin{equation}\label{Eq.f_2_limits}
	\lim_{\ell\downarrow0}f_2(k,\ell)=+\infty,\qquad \lim_{\ell\to+\infty}f_2(k,\ell)=-\infty.
\end{equation}

\begin{lemma}\label{Lem.f_1_decreasing}
	For $k\geq2$, let $E_k$ be defined in \eqref{Ek}. Then the function $\varepsilon\mapsto f_1(k, \varepsilon)$ is strictly decreasing for $\varepsilon\in E_k$.
\end{lemma}

\begin{proof}
	For $\varepsilon\in E_k$, a direct calculation gives
	\[\frac{\pa f_1}{\pa\varepsilon}=\frac{2\sqrt{2k/3}}{\sqrt{1+\varepsilon}}\left(\frac{k+2}{4\sqrt{2k/3}}-\sqrt{\frac{1+\varepsilon}{\varepsilon}}\right)<0.\]
	\end{proof}

As a consequence, for each $\varepsilon\in E_k$, by \eqref{Eq.f_2_limits}, we can find an $\ell=\ell(k,\varepsilon)>0$ such that $f_1(k, \varepsilon)=f_2(k, \ell)$ and the function $\varepsilon\mapsto \ell(k, \varepsilon)$ is strictly increasing. We denote the range of the function $\varepsilon\mapsto \ell(k, \varepsilon)$ by $(\ell_-(k), \ell_+(k))$. Since $f_1(k, 0)=k+2>f_2(k, 1)$, we know that $\ell_-(k)<1$. Note that for $2\leq k\leq 5$, we have $\lim_{\varepsilon\to+\infty}f_1(k, \varepsilon)=-\infty$, hence $\ell_+(k)=+\infty$; for $k=6$,  $\lim_{\varepsilon\to+\infty}f_1(k, \varepsilon)=0$, hence $\ell_+(6)=k/2=3$. For $k\geq 7$, we have
\[f_1\left(k, \frac{32k}{(3k-2)(k-6)}\right)=\sqrt{\frac{(3k-2)(k-6)}{3}}.\]
As a matter of fact, we note that
\[f_1\left(k, \frac{32k}{(3k-2)(k-6)}\right)=f_2(k,\ell)\quad \Longleftrightarrow\quad 12\ell^2-(3k^2-8k+12)\ell+3k^2=0.\]
Therefore, $\ell_+(k)=\ell_0(k)$ for all $k\geq 6$ by recalling \eqref{Eq.ell_0(k)} and $\ell_+(6)=3$.

Now we denote the inverse function of $\varepsilon\mapsto\ell(k,\varepsilon)$ by $\varepsilon=\varepsilon(k, \ell)$ hence $\ell\mapsto \varepsilon(k,\ell)$ is a strictly increasing function from $(\ell_-(k), \ell_+(k))$ onto $E_k$. Next we consider the function
\begin{equation}
	F(k,\ell):=f_3(k, \varepsilon(k, \ell))+52k-12\ell+12,\quad \ell\in(\ell_-(k), \ell_+(k)),
\end{equation}
where
\[f_3(k, \varepsilon):=(68k+12)\varepsilon-(33k+64)\sqrt{\frac{2k}{3}}\sqrt{\varepsilon(1+\varepsilon)}.\]
Then $B_2(R=3)=F(k,\ell)$.

\begin{lemma}\label{Lem.f_3}
	If $k\geq 2$, then the function $\varepsilon\mapsto f_3(k, \varepsilon)$ is strictly decreasing for $\varepsilon>0$.
\end{lemma}
\begin{proof}
	A direct calculation gives
	\begin{align*}
		\frac{\pa f_3}{\pa\varepsilon}&=68k+12-(33k+64)\sqrt{\frac{2k}{3}}\frac{\varepsilon+\frac12}{\sqrt{\varepsilon(1+\varepsilon)}}\\
		&<68k+12-(33k+64)\sqrt{\frac{2k}{3}}=-\sqrt{\frac23}\left(\sqrt k-\frac{\sqrt6}{3}\right)\left(33k-23\sqrt{6k}+18\right)<0,
	\end{align*}
	where we have used $33\times 2-23\sqrt{12}+18=2(42-23\sqrt3)>0$ (for $k=2$) and $33\sqrt k-23\sqrt 6\geq \sqrt3(33-23\sqrt2)>0$ for $k\geq 3$ and  then
	\[33k-23\sqrt{6k}+18= (33\sqrt k-23\sqrt{6})\sqrt k+18>0.\]
	\end{proof}

It follows from the monotonicity of $\ell\mapsto\varepsilon(k, \ell)$ and Lemma \ref{Lem.f_3} that the $\ell\mapsto F(k, \ell)$ is strictly decreasing for $\ell\in(\ell_-(k), \ell_+(k))$. Note that $$f_1\left(k,\frac{24k}{(3k-2)^2}\right)=(k+2)\frac{3k+2}{3k-2}-4\frac{4k}{3k-2}=k-2=f_2(k,1)$$ and $\frac{24k}{(3k-2)^2}\in E_k $ we have $ \varepsilon(k,1)=\frac{24k}{(3k-2)^2}$, then
\begin{align*}
	&\sqrt{\frac{2k}{3}}\sqrt{\varepsilon(k,1)(1+\varepsilon(k,1))}=\sqrt{\frac{2k}{3}}\sqrt{\frac{24k(3k+2)^2}{(3k-2)^4}}=\frac{4k(3k+2)}{(3k-2)^2},
\end{align*}
which gives
\begin{align*}
	f_3(k, \varepsilon(k,1))&=(68k+12)\varepsilon(k,1)-(33k+64)\sqrt{\frac{2k}{3}}\sqrt{\varepsilon(k,1)(1+\varepsilon(k,1))}\\
	&=(68k+12)\frac{24k}{(3k-2)^2}-(33k+64)\frac{4k(3k+2)}{(3k-2)^2}=-\frac{4k(33k-28)}{3k-2},
\end{align*}
hence,
\begin{align*}
	F(k,1)=f_3(k, \varepsilon(k, 1))+52k-12+12=-\frac{4k(33k-28)}{3k-2}+52k=\frac{4k(6k+2)}{3k-2}>0.
\end{align*}

Now we claim that
\begin{equation}\label{Eq.F(k,l_+)<0}
	\lim_{\ell\uparrow\ell_+(k)}F(k, \ell)<0.
\end{equation}
If $2\leq k\leq 6$, as $\ell\uparrow\ell_+(k)$, we have $\varepsilon\to+\infty$, then by $68k+12<(33k+64)\sqrt{\frac{2k}{3}}$ and $\ell>0$, we get
\[F(k,\ell)<\left[68k+12-(33k+64)\sqrt{\frac{2k}{3}}\right]\varepsilon+52k+12,\]
hence $\lim_{\ell\uparrow\ell_+(k)}F(k, \ell)=-\infty<0$. If $k\geq 7$, as $\ell\uparrow\ell_+(k)>1$, we have $\varepsilon\uparrow\frac{32k}{(3k-2)(k-6)}$, then
\begin{align*}
	\lim_{\ell\uparrow\ell_+(k)}F(k, \ell)&\leq f_3\left(k, \frac{32k}{(3k-2)(k-6)}\right)+52k\\
	&=\frac{32k(68k+12)-8k(k+2)(33k+64)}{(3k-2)(k-6)}+52k=-\frac{4k(9k-2)}{k-6}<0.
\end{align*}
Then it follows from the continuity of $\ell\mapsto F(k,\ell)$, $F(k, 1)>0$,  \eqref{Eq.F(k,l_+)<0} and the intermediate value theorem that there exists $\ell_1(k)\in (1, \ell_+(k))$ such that $F(k, \ell_1(k))=0$; moreover, by the monotonicity we have
\begin{equation}\label{Fkl}F(k, \ell)>0,\quad \forall\ \ell\in(1,\ell_1(k))\qquad\text{and}\qquad F(k, \ell)<0,\quad \forall\ \ell\in(\ell_1(k), \ell_+(k)).\end{equation}

\begin{lemma}\label{Lem.l_*_upbound}
	We have $4/5<\ell_1(4)<3/4<\ell_1(3)<3/2<\ell_1(2)$ and
	\begin{equation}
		\ell_1(k)<2\quad\forall\ k\geq 2.
	\end{equation}
\end{lemma}

\begin{proof}
	First of all, we deal with the case of  $k\geq 7$. We claim that
	\[\ell_1(k)<\ell_+(k)<2\quad\forall\ k\geq7.\]
	Indeed, for $k\geq 7$, we have
	\begin{align*}
		&f_1\left(k, \frac{32k}{(3k-2)(k-6)}\right)^2-f_2(k, 2)^2=\frac{(3k-2)(k-6)}{3}-\frac{(k-4)^2}{2}\\
		&=\frac{3k^2-16k-24}{6}=\frac{(3k-16)k-24}{6}\geq \frac{(3\times7-16)\times7-24}{6}=\frac{11}6>0,
	\end{align*}
	hence, $\ell_+(k)<2$ for all $k\geq 7$ and thus $\ell_1(k)<2$.
	
	For $2\leq k\leq 6$, the result follows from Table \ref{Tab.1}.
\end{proof}

Thus,  $\ell_1(k)< k$ for $k\geq2$. Lemma \ref{Lem.a_4<0} for $k\geq2$ follows from \eqref{B2>0}, \eqref{Fkl} and $B_2(R=3)=F(k,\ell)$. Moreover, by $\ell_+(k)=\ell_0(k)$ for all $k\geq 6$, \eqref{Eq.mathcal_K_0} and \eqref{Eq.admissible_set}, we know that $\mathcal K_*$ is a proper subset of $\mathcal K$.

The result of \cite{Shao-Wei-Zhang} relies on the condition $\beta>\ell+1$, now we discuss this condition for $k\in\{3, 4\}$.

\begin{lemma}\label{Lem.beta_neq_l+1}
	If $R\in[3,4]$, $k\in\{3, 4\}$ then
	\begin{equation}\label{Eq.A.14}
		k>\ell(\g+1)\ \ \text{ for }\ \ k=4,\ 1<\ell<\ell_1(4)\ \ \text{ or }\ k=3,\ 1<\ell<\ell^*(3).
	\end{equation}
	Here $\ell^*(3)=\frac{76-4\sqrt{154}}{23}\in(\frac{8}{7},\frac{7}{6})\subset(1,\ell_1(3)) $.	As a direct consequence, we have
	\begin{equation}\label{Eq.A.15}
		\beta>\ell+1\ \ \text{ for }\ \ k=4,\ 1<\ell<\ell_1(4)\ \ \text{ or }\ k=3,\ 1<\ell<\ell^*(3).
	\end{equation}
\end{lemma}

Thanks to $\ell=1+\frac{4}{p-1} $, $d=k+1$, $\frac{8}{7}<\ell^*(3)< \frac{5}{4}<\ell_1(4)<\ell_1(3)$. This lemma ensures that we can prove the blow-up for nonlinear
wave equations for $d=4$, $p\geq29$ or $d=5$, $p\geq17$. See \cite{Shao-Wei-Zhang}. As $\g>\frac{1}{\sqrt{\ell}}$, $\ell>1$, we have $\ell(\g+1)>\ell\left(\frac{1}{\sqrt{\ell}}+1\right)=\sqrt{\ell}+{\ell}>2\geq k$ for $k\in\{1, 2\}$.

\begin{proof}
	Thanks to $m=\beta\ell/(\ell+1)$ and $\g=\frac km-1$, one can easily deduce \eqref{Eq.A.15} from \eqref{Eq.A.14}.
	
	For $k\in\{3, 4\}$, as $ \ell^*(3)< \ell_1(4)< \frac{4}{3}$ we have $k>2\ell$, recall (see \eqref{Eq.(R+1)^2/R_1}) that
	\begin{align*}
		\frac{(R+1)^2}{R}=\frac{[(k+2)\ell\g-(k-2\ell)]^2}{2k\ell(\ell\g^2-1)}=:F(\g),\quad F'(\g)<0 \text{ for } \frac{1}{\sqrt{\ell}}<\g<\frac{k+2}{k-2\ell}.
	\end{align*}If $\ell(\g+1)\geq k$, then $\g\geq k/\ell-1>1>\frac{1}{\sqrt{\ell}}$ and $\g<\frac{k+2}{k-2\ell}$ then\begin{align*}
		\frac{16}{3}\leq\frac{(R+1)^2}{R}=F(\g)\leq F(k/\ell-1)=\frac{[(k+2)(k-\ell)-(k-2\ell)]^2}{2k((k-\ell)^2-\ell)}
		=\frac{k(k+1-\ell)^2}{2((k-\ell)^2-\ell)}.
	\end{align*}
	If $k=4$, then $ \frac{16}{3}\leq\frac{4(5-\ell)^2}{2((4-\ell)^2-\ell)}$, i.e., $5\ell^2-42\ell+53\leq0$, hence
	$$ \frac{21+4\sqrt{11}}{5}\geq\ell\geq\frac{21-4\sqrt{11}}{5}>\frac{21-14}{5}>\frac{4}{3}>\ell_1(4),$$
	which implies that if $k=4$ and $1<\ell<\ell_1(4)$, then $\ell(\g+1)< k$.
	
	If $k=3$, then $ \frac{16}{3}\leq\frac{3(4-\ell)^2}{2((3-\ell)^2-\ell)}$, i.e., $23\ell^2-152\ell+144\leq0$, hence,
	$$ \frac{76+4\sqrt{154}}{23}\geq\ell\geq\frac{76-4\sqrt{154}}{23}=\ell^*(3),$$ which implies that if $k=3$ and $1<\ell<\ell^*(3)$, then $\ell(\g+1)< k$.
\end{proof}

\begin{table}[htbp]
	\centering
	\caption{The value of $\ell_1(k)$ for $2\leq k\leq 6$.}
	\begin{tabular}{|c|c|c|c|c|c|}
		\hline
		$k$   & 2     & 3     & 4     & 5     & 6     \\
		\hline
		$\ell_1(k)$ & 1.881587232 & 1.391124091 & 1.2622855 & 1.199483016 & 1.161595181 \\
		\hline
		$\varepsilon_*(k,\ell_1(k))$ & 9.581746731 & 3.045800645 & 1.74343538 & 1.207995911 & 0.92023964  \\
		\hline
		$k-k/\ell_1(k)$ & 0.937067617 & 0.8434706 & 0.83114477 & 0.831537476 & 0.834689316 \\
		\hline
	\end{tabular}%
	\label{Tab.1}%
\end{table}%


\section{Proof of Lemma \ref{Lem.a_4<M}, Lemma \ref{Lem.a_1<} and Lemma \ref{Lemu3a}}\label{Subsec.proof_C}

\begin{lemma}\label{lem5}
	Assume that $k=2$, $\ell>1$, $R\in(3,R_*]$. Then the functions $\ell\mapsto B_2/a_2$, $\ell\mapsto \delta/a_2$, $\ell\mapsto B_1/\delta$, $\ell\mapsto (\ell-1)/a_2$, $\ell\mapsto a_3/a_2$ are strictly decreasing, and\begin{align}
		&\lim_{\ell\to+\infty}\frac{B_2}{a_2}=-2\frac{(R^2-10R+11)\sqrt{R}+9R-3}{R(\sqrt{R}-1)(R-1+2\sqrt{R})},\label{Eq.B_2/a_2}\\
		&\lim_{\ell\to+\infty}\frac{\delta}{a_2}=\frac{2(R-2)}{\sqrt{R}(\sqrt{R}-1)^2(R-1+2\sqrt{R})},\label{Eq.detla/a_2}\\
		&\lim_{\ell\to+\infty}\frac{B_1}{\delta}=\frac{(R^2-10R+8)\sqrt{R}+3R^2-2}{\sqrt{R}(R-2)}.\label{Eq.B_1/delta}
	\end{align}
\end{lemma}
Now we prove Lemma \ref{Lem.a_4<M}.
\begin{proof}[Proof of Lemma \ref{Lem.a_4<M}]
	As $k=2$, $\ell>1$ we have $(k,\ell)\in \mathcal{K}_*$. By \eqref{Eq.a_3_4}, Lemma \ref{Lem.a_3>0}, $a_2>0$, $a_1>4>0$, $k=2$, $M=a_3^2/(4a_2)$ and $3<R<R_*<4$, we get
	\begin{align*}
		&(R-3)\delta a_3=B_1a_2+(\ell-1)a_1>0\Longrightarrow a_3>0,\\
		&(R-4)\delta (a_4-M)=B_2a_3+2ka_2(a_2+1)+(4-R)\delta M\\
		&>B_2a_3+4a_2^2+(4-R)\delta M\geq B_2a_3+2\sqrt{4a_2^2(4-R)\delta M}\\
		&=B_2a_3+2\sqrt{(4-R)\delta a_2}a_3.
	\end{align*}
	By Lemma \ref{lem5}, we have
	\begin{align*}
		&\frac{B_2}{a_2}>-2\frac{(R^2-10R+11)\sqrt{R}+9R-3}{R(\sqrt{R}-1)(R-1+2\sqrt{R})},
		&\frac{\delta}{a_2}>\frac{2(R-2)}{\sqrt{R}(\sqrt{R}-1)^2(R-1+2\sqrt{R})}.
	\end{align*}
	Thus we infer
	\begin{align*}
		&\frac{(R-4)\delta (a_4-M)}{a_2a_3}>\frac{B_2a_3+2\sqrt{(4-R)\delta a_2}a_3}{a_2a_3}=\frac{B_2}{a_2}+2\sqrt{(4-R)\frac{\delta}{a_2}}\\
		&>-2\frac{(R^2-10R+11)\sqrt{R}+9R-3}{R(\sqrt{R}-1)(R-1+2\sqrt{R})}+2\sqrt{\frac{2(4-R)(R-2)}{\sqrt{R}(\sqrt{R}-1)^2(R-1+2\sqrt{R})}}\\
		&=2\frac{\sqrt{2(4-R)(R-2)\sqrt{R}R(R-1+2\sqrt{R})}-\big[(R^2-10R+11)\sqrt{R}+9R-3\big]}{R(\sqrt{R}-1)(R-1+2\sqrt{R})}\\
		&=2\frac{\sqrt{\varphi_1(\sqrt{R})}-\varphi_2({R})}{R(\sqrt{R}-1)(R-1+2\sqrt{R})}.
	\end{align*}
	Here\begin{align*}
		&\varphi_1(z):=2(4-z^2)(z^2-2)z^3(z^2-1+2z),
		&\varphi_2({R}):=(R^2-10R+11)\sqrt{R}+9R-3.
	\end{align*}Note that\begin{align*}
		&\varphi_2''({R})=\frac{15R^2-30R-11}{4R\sqrt{R}}=\frac{15(R-3)(R+1)+34}{4R\sqrt{R}}>0, \qquad\forall\ R\geq3.
	\end{align*} Then $\varphi_2({R})\leq\max\{\varphi_2(3),\varphi_2(R_*)\}$ for $R\in [3,R_*].$ Since $R_*=\frac{100}{27}$ and
	\begin{align}
		&\varphi_2(3)=(3^2-10\cdot3+11)\sqrt{3}+9\cdot3-3=24-10\sqrt{3}<24-10\cdot1.73=6.7,\\
		\label{phi2}&\varphi_2(R_*)=(R_*^2-10R_*+11)\sqrt{R_*}+9R_*-3=-\frac{8981}{729}\frac{10}{3\sqrt{3}}+\frac{91}{3}
		<-\frac{41}{\sqrt{3}}+\frac{91}{3}\\ \notag&\qquad\ \ \ <-{71}/{3}+{91}/{3}={20}/{3}<6.7,
	\end{align}
	where we used the fact that $41\cdot729\cdot3=41\cdot27\cdot81=1107\cdot81=89667<8981\cdot10$ and $ 41\sqrt{3}=\sqrt{5043}>\sqrt{5041}=71$. Thus $\varphi_2({R})\leq\max\{\varphi_2(3),\varphi_2(R_*)\}<6.7$ for $R\in [3,R_*].$
	
	We can write $\varphi_1(z)=-2z^3\prod_{j=1}^6(z-z_j)$ with\begin{align*}
		z_1=-1-\sqrt{2},\ z_2=-2,\ z_3=-\sqrt{2},\ z_4=\sqrt{2}-1,\ z_5=\sqrt{2},\ z_6=2.
	\end{align*}Then $z_1<z_2<z_3<0<z_4<z_5<\sqrt{3}<\sqrt{R_*}<z_6$, and $ \varphi_1(z)>0$ for $z\in [\sqrt{3},\sqrt{R_*}]$. Let $ \varphi_3(z)=\ln\varphi_1(z)$ then $\varphi_3''(z)=-\frac{3}{z^2}-\sum_{j=1}^6 \frac{1}{(z-z_j)^2}<0$ for $z\in [\sqrt{3},\sqrt{R_*}]$. Thus, $\varphi_3(z)\geq\min\{\varphi_3(\sqrt{3}),\varphi_3(\sqrt{R_*})\}$ and $\varphi_1(z)\geq\min\{\varphi_1(\sqrt{3}),\varphi_1(\sqrt{R_*})\}$ for $z\in [\sqrt{3},\sqrt{R_*}]$. Recall that $R_*=\frac{100}{27}$, then $\sqrt{R_*}=\frac{10}{3\sqrt{3}}$, hence,
	\begin{align*}
		&\varphi_1(\sqrt{3})=2(4-3)(3-2)3\sqrt{3}(3-1+2\sqrt{3})=36+12\sqrt{3}>36+12>6.7^2,\\
		&\varphi_1(\sqrt{R_*})=2(4-R_*)(R_*-2)\sqrt{R_*}R_*(R_*-1+2\sqrt{R_*})\\
		&=2\cdot\frac{8}{27}\cdot\frac{46}{27}\cdot\frac{10}{3\sqrt{3}}
		\cdot\frac{100}{27}\cdot\left(\frac{73}{27}+\frac{20}{3\sqrt{3}}\right)
		>2\cdot\frac{8}{27}\cdot\frac{46}{27}\cdot\frac{10}{3\sqrt{3}}
		\cdot\frac{100}{27}\cdot\frac{34}{3\sqrt{3}}\\&=\frac{25024000}{27^4}>\frac{5000^2}{27^4}=\left(\frac{5000}{729}\right)^2>6.7^2,
	\end{align*}where we used $73=\sqrt{5329}>\sqrt{5292}=42\sqrt{3}$, hence $\frac{73}{27}+\frac{20}{3\sqrt{3}}>\frac{42\sqrt{3}}{27}+\frac{20}{3\sqrt{3}}=\frac{34}{3\sqrt{3}}$. As a consequence, for $R\in (3,R_*)$ we have $\sqrt{R}\in (\sqrt{3},\sqrt{R_*})$, thus $\varphi_1(\sqrt{R})\geq\min\{\varphi_1(\sqrt{3}),\varphi_1(\sqrt{R_*})\}>6.7^2 $, $\sqrt{\varphi_1(\sqrt{R})}>6.7>\varphi_2({R}) $ and\begin{align*}
		&\frac{(R-4)\delta (a_4-M)}{a_2a_3}>2\frac{\sqrt{\varphi_1(\sqrt{R})}-\varphi_2({R})}{R(\sqrt{R}-1)(R-1+2\sqrt{R})}>0.
	\end{align*}Then by $a_2>0$, $a_3>0$, $R<4$, $ \delta>0$, we have $ a_4-M<0$, i.e.,  $a_4<M$.
	
	Next we show that $a_3>4a_2$. By Lemma \ref{lem5}, we have
	\begin{align*}
		&\frac{B_1}{\delta}>
		\frac{(R^2-10R+8)\sqrt{R}+3R^2-2}{\sqrt{R}(R-2)}.
	\end{align*}Then by \eqref{Eq.a_3_4}, $a_2>0$, $a_1>4>0$, $\ell>1$, $3<R<R_*<4$, we get\begin{align*}
		&(R-3)\delta a_3=B_1a_2+(\ell-1)a_1>B_1a_2,\quad \frac{a_3}{a_2}>\frac{B_1}{(R-3)\delta}>\frac{(R^2-10R+8)\sqrt{R}+3R^2-2}{\sqrt{R}(R-2)(R-3)}.
	\end{align*}Thus we infer \begin{align*}
		&\frac{a_3}{a_2}-4>\frac{(R^2-10R+8)\sqrt{R}+3R^2-2}{\sqrt{R}(R-2)(R-3)}-4=
		\frac{3R^2-2-(3R^2-10R+16)\sqrt{R}}{\sqrt{R}(R-2)(R-3)}.
	\end{align*}
	As $3<R<4$, we have $3R^2-10R+16=(3R-1)(R-3)+13>0 $ and $ R+4-4\sqrt{R}=(\sqrt{R}-2)^2>0$, hence \begin{align*}
		&3R^2-2-(3R^2-10R+16)\sqrt{R}>3R^2-2-(3R^2-10R+16)(R+4)/4\\
		&=-(3R^3-10R^2-24R+72)/4=[R^2-3(R^2-8)(R-3)]/4.
	\end{align*}
	As $3<R<R_*=\frac{100}{27}<\sqrt{14}$, we have $3(R-3)<3(R_*-3)=\frac{19}{9} $, $ 0<R^2-8<6$, and then
	\begin{align*}
		&\frac{R^2}{R^2-8}=1+\frac{8}{R^2-8}>1+\frac{8}{6}=\frac{7}{3}>\frac{19}{9}>3(R-3),\\
		&\frac{a_3}{a_2}-4>\frac{3R^2-2-(3R^2-10R+16)\sqrt{R}}{\sqrt{R}(R-2)(R-3)}>\frac{R^2-3(R^2-8)(R-3)}{4\sqrt{R}(R-2)(R-3)}>0.
	\end{align*}
	Therefore $a_3/a_2>4$. 
	\end{proof}

\begin{lemma}\label{lem6}
	If $k=2$, $\ell>1$, then the function $R\mapsto a_2$ is strictly decreasing for $R>2$.
\end{lemma}\begin{lemma}\label{lem7}
	If $k=2$, $\ell>1$, $R=4-1/\ell$, then $\gamma=2$, $A=4\ell$, $ \delta=4\ell$, $a_1=2(4\ell-1)$, and $a_2=\frac{7(\ell-1)(4\ell-1)}{2\ell-1},  B_2=-\frac{14\ell(\ell-3)}{2\ell-1}, B_1=\frac{28\ell^2-20\ell+10}{2\ell-1}, a_3=\frac{(4\ell-1)(51\ell^2-37\ell+18)}{(2\ell-1)^2}.$
\end{lemma}\begin{proof}
	For $k=2$, $\gamma=2$, we have $A=k+2-(k-2\ell)\g=4\ell$ and by \eqref{Eq.(R+1)^2/R_1} we have
	\begin{align*}
		&\frac{(R+1)^2}{R}=\frac{[(k+2)\ell\g-(k-2\ell)]^2}{2k\ell(\ell\g^2-1)}=
		\frac{[4\ell\cdot2-(2-2\ell)]^2}{4\ell(4\ell-1)}=\frac{(5\ell-1)^2}{\ell(4\ell-1)},
	\end{align*}then $R=4-1/\ell$ (as $R>1$). By Remark \ref{Rmk.R_parameter}, the function $\g\mapsto R$ is strictly decreasing, thus if
	$R=4-1/\ell$, then $\gamma=2$, $A=4\ell$ and $ \varepsilon=\ell\g^2-1=4\ell-1$. By \eqref{Eq.Rdelta^2}, we have $R\delta^2=2k\varepsilon(1+\varepsilon)=4(4\ell-1)\cdot4\ell=4R\ell\cdot4\ell=R(4\ell)^2 $, then $ \delta=4\ell$ and
	$R\delta=(4-1/\ell)\cdot4\ell=4(4\ell-1)=4\varepsilon$. By \eqref{Eq.a_1+c_1=Rdelta}, we have $a_1=R\delta-2\varepsilon=2\varepsilon=2(4\ell-1)=2R\ell$.
	
	As $k=2$, $B=2k+1-\ell=5-\ell$, we get by \eqref{Eq.a_2_expression}  that
	\begin{align*}
		&(R-2)a_2=\big((k-3)R+1\big)a_1+3AR-(2k+B)R\\
		&=(-R+1)\cdot2R\ell+3\cdot4\ell R-(4+5-\ell)R\\
		&=(15\ell-9-2R\ell)R=(15\ell-9-2(4\ell-1))R=7(\ell-1)R,
	\end{align*}
	hence,\[
		a_2=\frac{7(\ell-1)R}{R-2}=\frac{7(\ell-1)(4-1/\ell)}{2-1/\ell}=\frac{7(\ell-1)(4\ell-1)}{2\ell-1}.\]
	It follows from \eqref{Eq.B_1_2_def} that\begin{align*}
		&B_2=-5a_2+4ka_1-(2k+A+3B)=-5a_2+4\cdot2\cdot2(4\ell-1)-(4+4\ell+3(5-\ell))\\
		=&-5a_2+63\ell-35=-5\frac{7(\ell-1)(4\ell-1)}{2\ell-1}+7(9\ell-5)=-\frac{7(2\ell^2-6\ell)}{2\ell-1}=-\frac{14\ell(\ell-3)}{2\ell-1},
	\end{align*} and\begin{align*}
		B_1=&(3k-1)a_1-2a_2-2(k+B)=5\cdot2(4\ell-1)-2a_2-2(2+5-\ell)=42\ell-24-2a_2\\
		=&42\ell-24-2\frac{7(\ell-1)(4\ell-1)}{2\ell-1}=\frac{28\ell^2-20\ell+10}{2\ell-1}.
	\end{align*}
	By \eqref{Eq.a_3_4}, we have\begin{align*}
		&(R-3)\delta a_3=B_1a_2+(\ell-1)a_1=\frac{28\ell^2-20\ell+10}{2\ell-1}\frac{7(\ell-1)(4\ell-1)}{2\ell-1}+(\ell-1)\cdot2(4\ell-1)\\
		=&\frac{(\ell-1)(4\ell-1)}{(2\ell-1)^2}[7(28\ell^2-20\ell+10)+2(2\ell-1)^2]=\frac{(\ell-1)(4\ell-1)}{(2\ell-1)^2}(204\ell^2-148\ell+72),
	\end{align*} and\begin{align*}
		a_3&=\frac{(\ell-1)(4\ell-1)}{(2\ell-1)^2}\frac{204\ell^2-148\ell+72}{(R-3)\delta}=
		\frac{(\ell-1)(4\ell-1)}{(2\ell-1)^2}\frac{4(51\ell^2-37\ell+18)}{(4-1/\ell-3)\cdot4\ell}\\
		&=\frac{(\ell-1)(4\ell-1)}{(2\ell-1)^2}\frac{4(51\ell^2-37\ell+18)}{4(\ell-1)}=\frac{(4\ell-1)(51\ell^2-37\ell+18)}{(2\ell-1)^2}.
	\end{align*}This completes the proof.
\end{proof}

Now we prove Lemma \ref{Lem.a_1<}.

\begin{proof}[Proof of Lemma \ref{Lem.a_1<}]
	By Remark \ref{Rmk.R_parameter}, the function $\g\mapsto R$ is strictly decreasing.
	As $k=2$, $\ell>1$ and $3<R<4-1/\ell$, 
		by Lemma \ref{lem7} we have $ \gamma>\gamma|_{R=4-1/\ell}=2$, then $ \varepsilon=\ell\g^2-1>4\ell-1$. By \eqref{Eq.Rdelta^2}, $k=2$ and $3<R<4-1/\ell$, we have\begin{align*}
			(R\delta/\varepsilon)^2=&R(R\delta^2)/\varepsilon^2=R\cdot2k\varepsilon(1+\varepsilon)/\varepsilon^2=4R(1/\varepsilon+1)\\
			< &4(4-1/\ell)(1/(4\ell-1)+1)=16.
		\end{align*}Thus, $R\delta/\varepsilon<4 $, $R\delta<4\varepsilon $. 
		Then by \eqref{Eq.a_1+c_1=Rdelta}, we have $a_1=R\delta-2\varepsilon<2\varepsilon$.
		
		As $3<R<4-1/\ell$, by Lemma \ref{lem6} we have $a_2>a_2(R=4-1/\ell)$. By Lemma \ref{lem7}, we have $a_2(R=4-1/\ell)=\frac{7(\ell-1)(4\ell-1)}{2\ell-1}$. Thus(as $\ell>3/2$),
		\begin{align*}
			&a_2-7>\frac{7(\ell-1)(4\ell-1)}{2\ell-1}-7=\frac{7(4\ell^2-7\ell+2)}{2\ell-1}=\frac{7(2\ell-3)(2\ell-1)+7(\ell-1)}{2\ell-1}>0,
		\end{align*}and $ a_2>7$. 
	\end{proof}
	
	Now we prove Lemma \ref{Lemu3a}.

	\begin{proof}[Proof of Lemma \ref{Lemu3a}]
		By the proof of Lemma \ref{Lem.u_3_compare}, we know
		$\mathcal L\left(u_{(3)}\right)(z)=z^4q_4(z)$ with\begin{align*}
			q_4(z)&=-(R-4)\delta a_4+B_3a_3z-(3k+1)a_3^2z^2,\quad (R-4)\delta a_4=B_2a_3+2ka_2^2+2ka_2,
		\end{align*}
		and $B_3:=3a_3-(5k+1)a_2-B-2k$. As $k=2$, $B=2k+1-\ell=5-\ell$, we have
		\begin{align*}
			&B_3=3a_3-11a_2-(5-\ell)-4=3a_3-11a_2-9+\ell,\\
			&q_4(z)=-(R-4)\delta a_4+B_3a_3z-7a_3^2z^2=-(R-4)\delta a_4+B_3^2/28-(14a_3z-B_3)^2/28.
		\end{align*}
		
		{\bf Case 1.} $4-1/\ell\leq R_* $. Then $R=4-1/\ell$. By Lemma \ref{lem7}, we compute
		 \begin{align*}
			B_3&=3a_3-11a_2-9+\ell=3\frac{(4\ell-1)(51\ell^2-37\ell+18)}{(2\ell-1)^2}-11\frac{7(\ell-1)(4\ell-1)}{2\ell-1}-9+\ell\\
			&=\frac{(4\ell-1)[3(51\ell^2-37\ell+18)-77(\ell-1)(2\ell-1)]}{(2\ell-1)^2}-9+\ell\\
			&=\frac{(4\ell-1)(-\ell^2+120\ell-23)}{(2\ell-1)^2}-9+\ell=\frac{441\ell^2-175\ell+14}{(2\ell-1)^2}=\frac{7(7\ell-2)(9\ell-1)}{(2\ell-1)^2},
		\end{align*} and\begin{align*}
			&(R-4)\delta a_4=B_2a_3+2ka_2^2+2ka_2>B_2a_3+2ka_2^2=B_2a_3+4a_2^2\\
			=&-\frac{14\ell(\ell-3)}{2\ell-1}\frac{(4\ell-1)(51\ell^2-37\ell+18)}{(2\ell-1)^2}+4\frac{7^2(\ell-1)^2(4\ell-1)^2}{(2\ell-1)^2}\\
			=&\frac{14(4\ell-1)[-\ell(\ell-3)(51\ell^2-37\ell+18)+14(\ell-1)^2(4\ell-1)(2\ell-1)]}{(2\ell-1)^3},
		\end{align*}\begin{align*}
			\frac{B_3^2}{28}&=\frac{7(7\ell-2)^2(9\ell-1)^2}{4(2\ell-1)^4}<\frac{7\cdot8(4\ell-1)(2\ell-1)(9\ell-1)^2}{4(2\ell-1)^4}=
			\frac{14(4\ell-1)(9\ell-1)^2}{(2\ell-1)^3},
		\end{align*}
		where we have used (as $\ell>3/2>4/3$)
		\begin{align*}
			8(4\ell-1)(2\ell-1)-(7\ell-2)^2=15\ell^2-20\ell+4=5(3\ell-4)\ell+4>0.
		\end{align*}
		Thus, we infer
		\begin{align*}
			&(R-4)\delta a_4-\frac{B_3^2}{28}>(R-4)\delta a_4-\frac{14(4\ell-1)(9\ell-1)^2}{(2\ell-1)^3}\\
			&=\frac{14(4\ell-1)[-\ell(\ell-3)(51\ell^2-37\ell+18)+14(\ell-1)^2(4\ell-1)(2\ell-1)-(9\ell-1)^2]}{(2\ell-1)^3}.
		\end{align*}Now let $ \lambda=\ell-1$, then $ \lambda>0$, $\ell=\lambda+1$ and
		\begin{align*}
			&\ell(\ell-3)=(\lambda+1)(\lambda-2)=\lambda^2-\lambda-2,\quad
			51\ell^2-37\ell+18=51\lambda^2+65\lambda+32,\\
			&\ell(\ell-3)(51\ell^2-37\ell+18)=51\lambda^4+14\lambda^3-135\lambda^2-162\lambda-64,\\
			&14(\ell-1)^2(4\ell-1)(2\ell-1)=14\lambda^2(4\lambda+3)(2\lambda+1)=112\lambda^4+140\lambda^3+42\lambda^2,\\
			&(9\ell-1)^2=(9\lambda+8)^2=81\lambda^2+144\lambda+64,
		\end{align*}
		hence,
		\begin{align*}
			&-\ell(\ell-3)(51\ell^2-37\ell+18)+14(\ell-1)^2(4\ell-1)(2\ell-1)-(9\ell-1)^2\\
			&=61\lambda^4+126\lambda^3+96\lambda^2+18\lambda>0.
		\end{align*}
		Therefore, $(R-4)\delta a_4-\frac{B_3^2}{28}>0 $ and $ (R-4)\delta a_4>\frac{B_3^2}{28}\geq0$. By $R=4-1/\ell$, $\delta=4\ell$, we have
		$ (R-4)\delta=-4<0$, then $a_4<0$, and we also have
		\begin{align*}
			&q_4(z)=-(R-4)\delta a_4+B_3^2/28-(14a_3z-B_3)^2/28\leq -(R-4)\delta a_4+B_3^2/28<0,
		\end{align*}and $\mathcal L\left(u_{(3)}\right)(z)=z^4q_4(z)<0$ for $z>0$.\smallskip
		
		{\bf Case 2.} $4-1/\ell\geq R_* $. Then $R=R_*$. Recall that $R_*=\frac{100}{27} $, then $ \ell\geq \frac{1}{4-R_*}=\frac{27}{8}=:\ell_*$. We claim that
		\begin{align}\label{a2}
			&(3a_3+\ell-1)/a_2<15,\quad B_2/a_2>-2/3,\quad \text{for } k=2,\ R=R_*,\ \ell\geq\ell_*.
		\end{align}Assuming \eqref{a2}, then by $a_2>0$ we have $a_3<(3a_3+\ell-1)/3<15a_2/3=5a_2 $ and $B_2>-(2/3)a_2$. By \eqref{Eq.a_3_4}, Lemma \ref{Lem.a_3>0}, $a_2>0$, $a_1>4>0$, $k=2$, $\ell>1$, $\delta>0$, $3<R=R_*<4$ and \eqref{a2}, we get
		\begin{align*}
			&(R-3)\delta a_3=B_1a_2+(\ell-1)a_1>0\Longrightarrow a_3>0,\\
			&B_3=3a_3-11a_2-9+\ell=3a_3+\ell-1-11a_2-8<15a_2-11a_2-8<4a_2,\\
			&B_2a_3>-(2/3)a_2a_3>-(2/3)a_2\cdot5a_2=-(10/3)a_2^2,\\
			&(R-4)\delta a_4=B_2a_3+2ka_2^2+2ka_2>-(10/3)a_2^2+2ka_2^2=(2/3)a_2^2>0\Longrightarrow a_4<0.
		\end{align*}
		Then for $z>0$, we have
		\begin{align*}
			q_4(z)&=-(R-4)\delta a_4+B_3a_3z-7a_3^2z^2\\&<-(2/3)a_2^2+4a_2a_3z-7a_3^2z^2=-(2/21)a_2^2-(7a_3z-2a_2)^2/7<0,
		\end{align*}and $\mathcal L\left(u_{(3)}\right)(z)=z^4q_4(z)<0$. 
		
		Thus, it remains to prove \eqref{a2}. Thanks to Lemma \ref{lem5} and $R=R_*$, the functions $\ell\mapsto (\ell-1)/a_2$, $\ell\mapsto a_3/a_2$, $\ell\mapsto B_2/a_2$ are strictly decreasing, so is $\ell\mapsto(3a_3+\ell-1)/a_2$. Thus, it is enough to prove the following inequalities\begin{align}\label{a21}
			&(3a_3+\ell-1)/a_2<15,\quad \text{for } k=2,\ R=R_*,\ \ell=\ell_*,\\
			\label{a22}&\lim_{\ell\to+\infty}\frac{B_2}{a_2}=-2\frac{(R_*^2-10R_*+11)\sqrt{R}_*+9R_*-3}{R_*(\sqrt{R}_*-1)(R_*-1+2\sqrt{R}_*)}>-\frac{2}{3}.
		\end{align}For \eqref{a22}, by \eqref{phi2} we have $(R_*^2-10R_*+11)\sqrt{R}_*+9R_*-3=\varphi_2(R_*)<20/3$; we also have
		\begin{align*}
			&(\sqrt{R}_*-1)(R_*-1+2\sqrt{R}_*)=R_*+1+({R}_*-3)\sqrt{R}_*\\&=\frac{127}{27}+\frac{19}{27}\cdot\frac{10}{3\sqrt{3}}>
			\frac{127}{27}+\frac{18}{27}\cdot\frac{10}{6}=\frac{157}{27},
		\end{align*}
		hence $R_*(\sqrt{R}_*-1)(R_*-1+2\sqrt{R}_*)>\frac{100}{27}\cdot\frac{157}{27}>\frac{14580}{27^2}=20$, then we deduce \eqref{a22} as follows
		\begin{align*}
			&-2\frac{(R_*^2-10R_*+11)\sqrt{R}_*+9R_*-3}{R_*(\sqrt{R}_*-1)(R_*-1+2\sqrt{R}_*)}>
			-2\cdot\frac{20/3}{20}=-\frac{2}{3}.
		\end{align*}Now we prove \eqref{a21}. For $k=2,$ $R=R_*,$ $\ell=\ell_*=\frac{1}{4-R_*}=\frac{27}{8}$, we have $ R=4-1/\ell$,
		then we get by Lemma \ref{lem7}  that
		\begin{align*}
			&\frac{3a_3+\ell-1}{a_2}=3\frac{51\ell^2-37\ell+18}{7(\ell-1)(2\ell-1)}+\frac{2\ell-1}{7(4\ell-1)}<
			3\frac{51\ell^2-37\ell+18}{7(\ell-1)(2\ell-1)}+\frac{2\ell-1}{7(4\ell-4)}\\
			=&\frac{12(51\ell^2-37\ell+18)+(2\ell-1)^2}{7\cdot4(\ell-1)(2\ell-1)}=\frac{616\ell^2-448\ell+217}{7\cdot4(\ell-1)(2\ell-1)}=
			\frac{88\ell^2-64\ell+31}{4(\ell-1)(2\ell-1)},
		\end{align*}
		hence (recalling that $ \ell=\ell_*=27/8$)\begin{align*}
			&15-\frac{3a_3+\ell-1}{a_2}>15-\frac{88\ell^2-64\ell+31}{4(\ell-1)(2\ell-1)}=
			\frac{32\ell^2-116\ell+29}{4(\ell-1)(2\ell-1)}=\frac{(8\ell-27)(4\ell-1)+2}{4(\ell-1)(2\ell-1)}>0,
		\end{align*}which implies \eqref{a21}. 
	\end{proof}
	
	As a consequence, we finish the proof of Lemma \ref{Lem.a_4<M}, Lemma \ref{Lem.a_1<} and Lemma \ref{Lemu3a} as long as we prove Lemma \ref{lem5} and Lemma \ref{lem6} regarding the monotonicity of functions.
	
	\begin{lemma}\label{lem8}
		If $k=2$, $\ell>1$, $R>2$, then we have
		\begin{align*}
			&\varepsilon=\frac{R(R^2+6R+1)}{(R-1)^4}\frac{(\ell-1)^2}{\ell}+\frac{4R}{(R-1)^2}+
			\frac{4R(R+1)}{(R-1)^4}\frac{\ell-1}{\ell}\sqrt{R(\ell-1)^2+\ell(R-1)^2},\\
			&\delta=\frac{8R(R+1)}{(R-1)^4}\frac{(\ell-1)^2}{\ell}+\frac{4(R+1)}{(R-1)^2}+
			\frac{2(R^2+6R+1)}{(R-1)^4}\frac{\ell-1}{\ell}\sqrt{R(\ell-1)^2+\ell(R-1)^2},\\
			&a_1=\frac{2R(3R+1)}{(R-1)^3}\frac{(\ell-1)^2}{\ell}+\frac{4R}{R-1}+\frac{2R(R+3)}{(R-1)^3}\frac{\ell-1}{\ell}\sqrt{R(\ell-1)^2+\ell(R-1)^2},\\
			&A=\frac{4R}{(R-1)^2}\frac{(\ell-1)^2}{\ell}+4+\frac{2(R+1)}{(R-1)^2}\frac{\ell-1}{\ell}\sqrt{R(\ell-1)^2+\ell(R-1)^2},\\
			&(R-2)a_2=\frac{2R(3R-1)}{(R-1)^2}\frac{(\ell-1)^2}{\ell}+\frac{4R^2}{(R-1)^2}\frac{\ell-1}{\ell}\sqrt{R(\ell-1)^2+\ell(R-1)^2}+R(\ell-1).
	\end{align*}\end{lemma}
	
	Now we prove Lemma \ref{lem6}.
	\begin{proof}[Proof of Lemma \ref{lem6}]
		By Lemma \ref{lem8}, we have\begin{align*}
			&a_2=\frac{2R}{(R-1)^2}\frac{3R-1}{R-2}\frac{(\ell-1)^2}{\ell}+
			\frac{4R}{R-1}\frac{R}{R-2}\frac{\ell-1}{\ell}\sqrt{\frac{R}{(R-1)^2}(\ell-1)^2+\ell}+\frac{R}{R-2}(\ell-1).
		\end{align*}
		Since $\ell>1$, and $\frac{R}{(R-1)^2}$, $\frac{3R-1}{R-2}$, $\frac{R}{R-1}$, $\frac{R}{R-2}$ are strictly decreasing and positive for $R>2$, so does $a_2$. 	\end{proof}
	Next we prove Lemma \ref{lem5}.
	
	\begin{proof}[Proof of Lemma \ref{lem5}]
		{{\bf Step 1.} Monotonicity of $\ell\mapsto (\ell-1)/a_2$.} By Lemma \ref{lem8}, we have
		\begin{align*}
			&\frac{(R-2)a_2}{\ell-1}=\frac{2R(3R-1)}{(R-1)^2}\frac{\ell-1}{\ell}+\frac{4R^2}{(R-1)^2}\frac{1}{\ell}\sqrt{R(\ell-1)^2+\ell(R-1)^2}+R\\
			&=\frac{2R\left[(3R-1)(1-1/\ell)+2R\sqrt{R(1-1/\ell)^2+(R-1)^2/\ell}\right]}{(R-1)^2}+R=\frac{2RF(1/\ell)}{(R-1)^2}+R,\\
			&F(z):=(3R-1)(1-z)+2R\sqrt{R(1-z)^2+(R-1)^2z}.
		\end{align*}Since $R\in(3,R_*]$,  $ R_*<2+\sqrt{3}$, we have $0<R-2<\sqrt{3}$, $(R-2)^2<3$, $R^2-4R+1<0 $, and thus $2R>(R-1)^2$. So, for
		$z\in[0,1]$ we have $F(z)>0$ and\[F'(z)=-(3R-1)+R\frac{-2R(1-z)+(R-1)^2}{\sqrt{R(1-z)^2+(R-1)^2z}}\]with
		\begin{align*}
			&-2R(1-z)+(R-1)^2\leq -(R-1)^2(1-z)+(R-1)^2=(R-1)^2z,
		\end{align*}
		hence,
		\begin{align*}
			&\frac{-2R(1-z)+(R-1)^2}{\sqrt{R(1-z)^2+(R-1)^2z}}\leq \frac{(R-1)^2z}{\sqrt{(R-1)^2z}}=(R-1)\sqrt{z}\leq R-1,\\
			&F'(z)\leq -(3R-1)+R(R-1)=R^2-4R+1<0,
		\end{align*} and thus $F(z)$ is positive and strictly decreasing for $z\in[0,1]$. Then $\ell\mapsto F(1/\ell)$ and $$\ell\mapsto \frac{(R-2)a_2}{\ell-1}=\frac{2RF(1/\ell)}{(R-1)^2}+R$$ are positive and strictly increasing for $\ell>1$, therefore $\ell\mapsto (\ell-1)/a_2$ is positive and strictly decreasing for $\ell>1$.
		
		{{\bf Step 2.} Monotonicity of $\ell\mapsto B_2/a_2$ and \eqref{Eq.B_2/a_2}.} As $k=2$, $B=2k+1-\ell=5-\ell$, we get by  \eqref{Eq.B_1_2_def}  that \begin{align*}
			B_2&=-5a_2+4ka_1-(2k+A+3B)=-5a_2+8a_1-(4+A+3(5-\ell))\\
			&=-5a_2+8a_1-A-16+3(\ell-1).
		\end{align*}
		By Lemma \ref{lem8}, we have
		\begin{align*}
			&a_1^*:=a_1-\frac{(R+3)(R-2)}{2R(R-1)}{a_2}=
			\frac{3}{R-1}\frac{(\ell-1)^2}{\ell}+\frac{4R}{R-1}-\frac{R+3}{2(R-1)}(\ell-1),\\
			&A_*:=A-\frac{(R+1)(R-2)}{2R^2}{a_2}=\frac{1}{R}\frac{(\ell-1)^2}{\ell}+4-\frac{R+1}{2R}(\ell-1).
		\end{align*} Thus we obtain
		\begin{align*}
			B_2&=-5a_2+8a_1^*-A_*-16+3(\ell-1)+\frac{4(R+3)(R-2)}{R(R-1)}{a_2}-\frac{(R+1)(R-2)}{2R^2}{a_2},
		\end{align*}and\begin{align*}
			\frac{8a_1^*-A_*-16+3(\ell-1)}{\ell-1}=&\left(\frac{24}{R-1}-\frac{1}{R}\right)\frac{\ell-1}{\ell}+
			\left(\frac{32R}{R-1}-20\right)\frac{1}{\ell-1}\\
			&-\frac{4(R+3)}{R-1}+\frac{R+1}{2R}+3,
		\end{align*}
		and then by $\frac{32R}{R-1}-20=\frac{12R+20}{R-1}$, 
		we have (for $R\in(3,R_*]$, $\ell>1$)
		\begin{align*}
			&\frac{\partial}{\partial\ell}\frac{8a_1^*-A_*-16+3(\ell-1)}{\ell-1}=\left(\frac{24}{R-1}-\frac{1}{R}\right)\frac{1}{\ell^2}-
			\frac{12R+20}{R-1}\frac{1}{(\ell-1)^2}\\
			&<\frac{24}{R-1}\frac{1}{\ell^2}-\frac{12R+20}{R-1}\frac{1}{\ell^2}=-\frac{12R-4}{R-1}\frac{1}{\ell^2}<0,
		\end{align*}
		and
		\begin{align}\label{l-1}
			&\lim_{\ell\to+\infty}\frac{8a_1^*-A_*-16+3(\ell-1)}{\ell-1}=\frac{24}{R-1}-\frac{1}{R}-\frac{4(R+3)}{R-1}+\frac{R+1}{2R}+3\\
			&=-\frac{4(R-3)}{R-1}+\frac{R-1}{2R}+3=\frac{9-R}{R-1}+\frac{R-1}{2R}>0.\nonumber
		\end{align}
		Thus, $\ell\mapsto [8a_1^*-A_*-16+3(\ell-1)]/(\ell-1)$ is positive and strictly decreasing for $\ell>1$. Recall that $\ell\mapsto (\ell-1)/a_2$ is positive and strictly decreasing for $\ell>1$ (in {\bf Step 1}) then\begin{align*}
			&\ell\mapsto\frac{8a_1^*-A_*-16+3(\ell-1)}{a_2}=\frac{8a_1^*-A_*-16+3(\ell-1)}{\ell-1}\cdot\frac{\ell-1}{a_2},\\
			&\ell\mapsto\frac{B_2}{a_2}=-5+\frac{4(R+3)(R-2)}{R(R-1)}-\frac{(R+1)(R-2)}{2R^2}+\frac{8a_1^*-A_*-16+3(\ell-1)}{a_2},
		\end{align*}are strictly decreasing for $\ell>1$. Now we compute the limit. By Lemma \ref{lem8}, we have 
		\begin{align*}
			\lim_{\ell\to+\infty}\frac{(R-2)a_2}{\ell-1}&=
			\frac{2R(3R-1)}{(R-1)^2}+\frac{4R^2\sqrt{R}}{(R-1)^2}+R=\frac{2R(\sqrt{R}+1)^2(2\sqrt{R}-1)}{(R-1)^2}+R\\
			&=\frac{2R(2\sqrt{R}-1)}{(\sqrt{R}-1)^2}+R=\frac{R(R+2\sqrt{R}-1)}{(\sqrt{R}-1)^2}>0,
		\end{align*}
		thus $\lim\limits_{\ell\to+\infty}\dfrac{\ell-1}{a_2}=\dfrac{(R-2)(\sqrt{R}-1)^2}{R(R+2\sqrt{R}-1)} $. Then by \eqref{l-1}, we compute
		\begin{align*}
			&\lim_{\ell\to+\infty}\frac{B_2}{a_2}=-5+\frac{4(R+3)(R-2)}{R(R-1)}-\frac{(R+1)(R-2)}{2R^2}+\lim_{\ell\to+\infty}\frac{8a_1^*-A_*-16+3(\ell-1)}{a_2}\\
			&=-5+\frac{4(R+3)(R-2)}{R(R-1)}-\frac{(R+1)(R-2)}{2R^2}+\left(\frac{9-R}{R-1}+\frac{R-1}{2R}\right)\frac{(R-2)(\sqrt{R}-1)^2}{R(R+2\sqrt{R}-1)},\end{align*}\if0\\
		&=-5+\frac{\big[4(R+3)(R+2\sqrt{R}-1)+(9-R)(\sqrt{R}-1)^2\big](R-2)}{R(R-1)(R+2\sqrt{R}-1)}\\&-
		\frac{\big[(R+1)(R+2\sqrt{R}-1)-(R-1)(\sqrt{R}-1)^2\big](R-2)}{2R^2(R+2\sqrt{R}-1)}\\
		&=-5+\frac{\big[3R^2+16R-3+(10R+6)\sqrt{R}\big](R-2)}{R(\sqrt{R}+1)(\sqrt{R}-1)(R+2\sqrt{R}-1)}-
		\frac{4R\sqrt{R}(R-2)}{2R^2(R+2\sqrt{R}-1)}\\
		&=-5+\frac{\big[(3R+9)\sqrt{R}+7R-3\big](R-2)}{R(\sqrt{R}-1)(R+2\sqrt{R}-1)}-
		\frac{2\sqrt{R}(R-2)}{R(R+2\sqrt{R}-1)}\\
		&=-5+\frac{\big[(3R+11)\sqrt{R}+5R-3\big](R-2)}{R(\sqrt{R}-1)(R+2\sqrt{R}-1)}\\
		&=\frac{-5R[(R-3)\sqrt{R}+R+1]+(3R^2+5R-22)\sqrt{R}+5R^2-13R+6}{R(\sqrt{R}-1)(R+2\sqrt{R}-1)}\\
		&=\frac{-(2R^2-20R+22)\sqrt{R}-18R+6}{R(\sqrt{R}-1)(R+2\sqrt{R}-1)}=-2\frac{(R^2-10R+11)\sqrt{R}+9R-3}{R(\sqrt{R}-1)(R+2\sqrt{R}-1)}.\fi
		which implies \eqref{Eq.B_2/a_2} by a direct calculation.
		
		{{\bf Step 3.} Monotonicity of $\ell\mapsto \delta/a_2$ and \eqref{Eq.detla/a_2}.} By Lemma \ref{lem8}, we have
		\begin{align*}
			&\delta_*:=\delta-\frac{(R^2+6R+1)(R-2)}{2R^2(R-1)^2}{a_2}=
			\frac{5R+1}{R(R-1)^2}\frac{(\ell-1)^2}{\ell}+\frac{4(R+1)}{(R-1)^2}-\frac{R^2+6R+1}{2R(R-1)^2}(\ell-1),
		\end{align*}\begin{align*}
			&\frac{\delta_*}{\ell-1}=\frac{5R+1}{R(R-1)^2}\frac{\ell-1}{\ell}+\frac{4(R+1)}{(R-1)^2}\frac{1}{\ell-1}-\frac{R^2+6R+1}{2R(R-1)^2}.
		\end{align*}Then for $R\in(3,R_*]\subset(3,4)$, $\ell>1$, we have
		\begin{align*}
			&\frac{\partial}{\partial\ell}\frac{\delta_*}{\ell-1}=
			\frac{5R+1}{R(R-1)^2}\frac{1}{\ell^2}-\frac{4(R+1)}{(R-1)^2}\frac{1}{(\ell-1)^2}<
			\frac{6}{(R-1)^2}\frac{1}{\ell^2}-\frac{4(R+1)}{(R-1)^2}\frac{1}{\ell^2}<0,
		\end{align*}\begin{align*}
			&\lim_{\ell\to+\infty}\frac{\delta_*}{\ell-1}=
			\frac{5R+1}{R(R-1)^2}-\frac{R^2+6R+1}{2R(R-1)^2}=\frac{4R-R^2+1}{2R(R-1)^2}>0.
		\end{align*}
		Thus, $\ell\mapsto \delta_*/(\ell-1)$ is positive and strictly decreasing for $\ell>1$. Recall that $\ell\mapsto (\ell-1)/a_2$ is positive and strictly decreasing for $\ell>1$ (in {\bf Step 1}), then\begin{align*}
			&\ell\mapsto\frac{\delta_*}{a_2}=\frac{\delta_*}{\ell-1}\frac{\ell-1}{a_2}\quad \text{and}\quad
			\ell\mapsto\frac{\delta}{a_2}=\frac{(R^2+6R+1)(R-2)}{2R^2(R-1)^2}+\frac{\delta_*}{a_2}.
		\end{align*}are positive and strictly decreasing for $\ell>1$. Now we compute the limit. By Lemma \ref{lem8}, we get
		\begin{align*}
			&\lim_{\ell\to+\infty}\frac{\delta}{\ell-1}=\frac{8R(R+1)}{(R-1)^4}+\frac{2(R^2+6R+1)\sqrt{R}}{(R-1)^4}=
			\frac{2(\sqrt{R}+1)^4\sqrt{R}}{(R-1)^4}=\frac{2\sqrt{R}}{(\sqrt{R}-1)^4}.
		\end{align*}
		Recall that $\lim\limits_{\ell\to+\infty}\frac{\ell-1}{a_2}=\frac{(R-2)(\sqrt{R}-1)^2}{R(R+2\sqrt{R}-1)} $ (in {\bf Step 2}), then
		\begin{align*}
			&\lim_{\ell\to+\infty}\frac{\delta}{a_2}=
			\frac{2\sqrt{R}}{(\sqrt{R}-1)^4}\frac{(R-2)(\sqrt{R}-1)^2}{R(R+2\sqrt{R}-1)}=\frac{2(R-2)}{\sqrt{R}(\sqrt{R}-1)^2(R+2\sqrt{R}-1)}.
		\end{align*}
		
		{{\bf Step 4.} Monotonicity of $\ell\mapsto B_1/\delta$ and \eqref{Eq.B_1/delta}.} As $k=2$, by \eqref{Eq.Rdelta^2} we have
		$R\delta^2=2k\varepsilon(1+\varepsilon)=4\varepsilon(1+\varepsilon)=(2\varepsilon+1)^2-1 $, then $2\varepsilon=\sqrt{R\delta^2+1}-1 $.
		As $k=2$, $B=2k+1-\ell=5-\ell$, we get by Lemma \ref{Lem.A_a_1} that
		\begin{align*}
			&\mu=(k+2)^2-(k-2\ell)^2/\ell=4^2-(2-2\ell)^2/\ell=4(6-\ell-1/\ell)=4(B+1-1/\ell),\\
			&\mu=(2R+k+4)A-\left(2-k/R\right)(R+1)a_1=(2R+6)A-\left(2-2/R\right)(R+1)a_1.\end{align*}As $k=2$, $2\varepsilon=\sqrt{R\delta^2+1}-1 $, by \eqref{Eq.B_1_expression}, \eqref{Eq.a_1+c_1=Rdelta} and \eqref{Eq.Rdelta^2}, we have
			\begin{align*}
			&(R-2)B_1=\big((k+5)R-6k\big)a_1-6AR+2kR+4k+4B\\&=\big(7R-12\big)a_1-6AR+4R+8-4(1-1/\ell)+\mu
			\\&=\big(7R-12\big)a_1-6AR+4R+4+4/\ell+(2R+6)A-\left(2-2/R\right)(R+1)a_1
			\\&=\big(5R-12+2/R\big)a_1-(4R-6)A+4R+4+4/\ell\\
			&=\big(5R-12+2/R\big)(R\delta-2\varepsilon)-(4R-6)((R+1)\delta-4\varepsilon)+4R+4+4/\ell\\
			&=\big(R^2-10R+8\big)\delta+2\varepsilon\big(3R-2/R\big)+4R+4+4/\ell\\
			&=\big(R^2-10R+8\big)\delta+\big(\sqrt{R\delta^2+1}-1\big)\big(3R-2/R\big)+4R+4+4/\ell\\
			&=\big(R^2-10R+8\big)\delta+\sqrt{R\delta^2+1}\big(3R-2/R\big)+R+4+2/R+4/\ell.
		\end{align*}
		Thus we obtain
		\begin{align*}
			&\frac{B_1}{\delta}=\frac{R^2-10R+8}{R-2}+\sqrt{R+\frac{1}{\delta^2}}\frac{3R^2-2}{R(R-2)}+\frac{R^2+4R+2}{R(R-2)\delta}+\frac{4}{(R-2)\ell\delta}.
		\end{align*}
		By Lemma \ref{lem8}, $\ell\mapsto\delta$ is positive and strictly increasing for $\ell>1$ as $\ell\mapsto\frac{(\ell-1)^2}{\ell}$, $\ell\mapsto\frac{\ell-1}{\ell}$, $\ell\mapsto(\ell-1)^2$ are positive and strictly increasing for $\ell>1$ and $R>2>1$. Thus, $\ell\mapsto B_1/\delta$ is strictly decreasing for $\ell>1$. The expression of
		$ B_1/\delta $ and $\lim_{\ell\to+\infty}\delta=+\infty$ imply the following limit\begin{align*}
			&\lim_{\ell\to+\infty}\frac{B_1}{\delta}=\frac{R^2-10R+8}{R-2}+\sqrt{R}\frac{3R^2-2}{R(R-2)}=\frac{(R^2-10R+8)\sqrt{R}+3R^2-2}{\sqrt{R}(R-2)}.
		\end{align*}
		
		{{\bf Step 5.} Monotonicity of $\ell\mapsto a_3/a_2$.} By \eqref{Eq.a_3_4}, we have
		\begin{align*}
			&(R-3)\frac{a_3}{a_2}=\frac{(R-3)\delta a_3}{\delta a_2}=\frac{B_1a_2+(\ell-1)a_1}{\delta a_2}=\frac{B_1}{\delta}+\frac{a_1}{\delta}\frac{\ell-1}{a_2}.
		\end{align*}
		Since $R>3$, {it is enough to prove that} $ \ell\mapsto (R-3){a_3}/{a_2}$ is strictly decreasing, which follows from the following three facts:
		(i) $\ell\mapsto B_1/\delta$ is strictly decreasing for $\ell>1$; (ii) $\ell\mapsto (\ell-1)/a_2$ is positive and strictly decreasing for $\ell>1$; (iii)
		$ \ell\mapsto a_1/\delta$ is positive and strictly decreasing for $\ell>1$. Here (i) was proved in {\bf Step 4}, and (ii) was proved in {\bf Step 1}. It remains to prove (iii).
		
		By \eqref{Eq.a_1+c_1=Rdelta} and $2\varepsilon=\sqrt{R\delta^2+1}-1  $ (in {\bf Step 4}), we get
		\begin{align*}
			\frac{a_1}{\delta}&=\frac{R\delta-2\varepsilon}{\delta}=R-\frac{2\varepsilon}{\delta}=R-\frac{\sqrt{R\delta^2+1}-1}{\delta}=R-
			\frac{R\delta}{\sqrt{R\delta^2+1}+1}\\&=R-
			\frac{R}{\sqrt{R+1/\delta^2}+1/\delta}.
		\end{align*}
		Thus, $\delta\mapsto a_1/\delta$ is positive and strictly decreasing for $\delta>0$ (for fixed $R>1$). Since $\ell\mapsto\delta$ is positive and strictly increasing for $\ell>1$ (see {\bf Step 4}),  we deduce (iii).  
		
		This completes the proof of Lemma \ref{lem5}.
		\end{proof}
		
	Finally, we prove Lemma \ref{lem8}.
	
	\begin{proof}[Proof of Lemma \ref{lem8}]
		By \eqref{Eq.Rdelta^2}, we have
		\begin{align}
			&A=(R+1)\delta-(k+2)\varepsilon=(R+1)\sqrt{2k\varepsilon(1+\varepsilon)/R}-(k+2)\varepsilon,\\
			\label{A1}&{A}/{\sqrt{\varepsilon}}=(R+1)\sqrt{2k(1+\varepsilon)/R}-(k+2)\sqrt{\varepsilon}.\end{align}It follows from \eqref{Eq.5} and $k=2$ that\begin{align*}
			&A^2/\varepsilon-(2-2\ell)^2/\ell=A^2/\varepsilon-(k-2\ell)^2/\ell\\&=\big[(R+1)\sqrt{2k(1+\varepsilon)/R}-(k+2)\sqrt{\varepsilon}\big]^2-
			\big[(k+2)\sqrt{1+\varepsilon}-(R+1)\sqrt{2k\varepsilon/R}\big]^2\\
			&=(R+1)^2\cdot2k/R-(k+2)^2=(R+1)^2\cdot4/R-4^2=4(R-1)^2/R,
		\end{align*}
		which gives
			$$A^2/\varepsilon
			=(2-2\ell)^2/\ell+4(R-1)^2/R=4\big[(\ell-1)^2/\ell+(R-1)^2/R\big].$$
		Due to $A>0$, $ \varepsilon>0$, we obtain
		\begin{align}
			\label{A2}&A/{\sqrt{\varepsilon}}=2\sqrt{(\ell-1)^2/\ell+(R-1)^2/R}=2\sqrt{[R(\ell-1)^2+\ell(R-1)^2]/(\ell R)}.
		\end{align}
		Notice that
		\begin{align*}
			&(k+2)\bigg[(R+1)\sqrt{\frac{2k(1+\varepsilon)}{R}}-(k+2)\sqrt{\varepsilon}\bigg]-
			(R+1)\sqrt{\frac{2k}{R}}\bigg[(k+2)\sqrt{1+\varepsilon}-(R+1)\sqrt{\frac{2k\varepsilon}{R}}\bigg]\\&\quad=
			\big[(R+1)^2\cdot2k/R-(k+2)^2\big]\sqrt{\varepsilon}.\end{align*}
			Then we get by \eqref{Eq.5} and \eqref{A1} that
			\begin{align*}
			&(k+2){A}/{\sqrt{\varepsilon}}-(R+1)\sqrt{2k/{R}}(k-2\ell)/\sqrt{\ell}=\big[(R+1)^2\cdot2k/R-(k+2)^2\big]\sqrt{\varepsilon}.\end{align*}
			For $k=2$, it becomes 
			\begin{align*}
			&4{A}/{\sqrt{\varepsilon}}-(R+1)\sqrt{4/{R}}(2-2\ell)/\sqrt{\ell}=4(R-1)^2\sqrt{\varepsilon}/R.\end{align*}
			Then by \eqref{A2}, we have
			\begin{align*}
			&\frac{(R-1)^2}{R}\sqrt{\varepsilon}=\frac{{A}}{\sqrt{\varepsilon}}-\frac{R+1}{\sqrt{R}}\frac{1-\ell}{\sqrt{\ell}}=2\sqrt{\frac{R(\ell-1)^2+\ell(R-1)^2}{\ell R}}
			+\frac{R+1}{\sqrt{R}}\frac{\ell-1}{\sqrt{\ell}},\end{align*}
			thus,
			\begin{align}\label{A3}
			&\sqrt{\varepsilon}=\frac{2R}{(R-1)^2}\sqrt{\frac{R(\ell-1)^2+\ell(R-1)^2}{\ell R}}+\frac{(R+1)\sqrt{R}}{(R-1)^2}\frac{\ell-1}{\sqrt{\ell}}.\end{align}
		By \eqref{A2} and \eqref{A3}, we get
		\begin{align*}
			&A=A/{\sqrt{\varepsilon}}\cdot{\sqrt{\varepsilon}}\\=&
			2\sqrt{\frac{R(\ell-1)^2+\ell(R-1)^2}{\ell R}}\bigg[\frac{2R}{(R-1)^2}\sqrt{\frac{R(\ell-1)^2+\ell(R-1)^2}{\ell R}}+
			\frac{(R+1)\sqrt{R}}{(R-1)^2}\frac{\ell-1}{\sqrt{\ell}}\bigg]\\=&
			\frac{4R}{(R-1)^2}{\frac{R(\ell-1)^2+\ell(R-1)^2}{\ell R}}+
			\frac{2(R+1)}{(R-1)^2}\frac{\ell-1}{{\ell}}\sqrt{{R(\ell-1)^2+\ell(R-1)^2}}\\=&
			\frac{4R}{(R-1)^2}{\frac{(\ell-1)^2}{\ell }}+4+
			\frac{2(R+1)}{(R-1)^2}\frac{\ell-1}{{\ell}}\sqrt{{R(\ell-1)^2+\ell(R-1)^2}},\end{align*}
			which gives the expression of $A$. Taking the square of \eqref{A3}, we obtain
			\begin{align*}
			\varepsilon=&
			\bigg[\frac{2R}{(R-1)^2}\sqrt{\frac{R(\ell-1)^2+\ell(R-1)^2}{\ell R}}+
			\frac{(R+1)\sqrt{R}}{(R-1)^2}\frac{\ell-1}{\sqrt{\ell}}\bigg]^2\\=&
			\frac{4R^2}{(R-1)^4}{\frac{R(\ell-1)^2+\ell(R-1)^2}{\ell R}}+\frac{(R+1)^2{R}}{(R-1)^4}\frac{(\ell-1)^2}{\ell }
			\\&+\frac{4R(R+1)}{(R-1)^4}\frac{\ell-1}{{\ell}}\sqrt{{R(\ell-1)^2+\ell(R-1)^2}}.
			\end{align*}
			Then the expression of $\varepsilon$ follows from\begin{align*}
			4R^2+(R+1)^2{R}=R(R^2+6R+1),\quad \frac{4R^2}{(R-1)^4}\frac{\ell(R-1)^2}{\ell R}=\frac{4R}{(R-1)^2}.\end{align*}
			
		As $k=2$, by \eqref{Eq.Rdelta^2} we have $ (R+1)\delta=(k+2)\varepsilon+A=4\varepsilon+A$, which gives the expression of $\delta$
		by using the expressions of $\varepsilon$ and $A$. By \eqref{Eq.a_1+c_1=Rdelta}, we have $ a_1=R\delta-2\varepsilon$, then the expression of $a_1$ follows by using the expressions of $\varepsilon$ and $\delta$. As $k=2$, $B=2k+1-\ell=5-\ell$, we get by \eqref{Eq.a_2_expression}  that 
		\begin{align*}
			&(R-2)a_2=\big((k-3)R+1\big)a_1+3AR-(2k+B)R\\
			&=\big(-R+1\big)a_1+3A R-(4+5-\ell)R=3AR-\big(R-1\big)a_1-8R+R(\ell-1),
		\end{align*}
		which gives the expression of $a_2$ by using the expressions of $A$ and $a_1$. 
		\end{proof}

\fancypagestyle{plain}{\pagestyle{mystyle}}

\printindex
\pagestyle{mystyle}		

\section*{Acknowledgments}
D. Wei is partially supported by the National Key R\&D Program of China under the
grant 2021YFA1001500.	Z. Zhang is partially supported by  NSF of China  under Grant 12288101.

\end{document}